\documentclass{article}

\usepackage{enumitem}
\usepackage{graphicx} 

\usepackage{amsmath,amssymb,amsfonts,amsthm,mathrsfs,mathtools,bbm,cancel,
calc,tikz-cd}

\usepackage{bbm}

\usepackage{comment}
\usepackage{hyperref}
\usepackage[margin=1in]{geometry}

\numberwithin{equation}{section}

\theoremstyle{plain}
\newtheorem{prop}{Proposition}[section]
\makeatletter
\@addtoreset{prop}{section}
\makeatother

\newtheorem{thm}[prop]{Theorem}
\newtheorem{lem}[prop]{Lemma}
\newtheorem*{lem*}{Lemma}
\newtheorem{cor}[prop]{Corollary}
\newtheorem*{cor*}{Corollary}           
\newtheorem*{conv*}{Convention}

\theoremstyle{definition}
\newtheorem{remarkonlyheader}[prop]{Remark}
\newtheorem{definitiononlyheader}[prop]{Definition}
\newtheorem{exampleonlyheader}[prop]{Example}
\newenvironment{ex}
  {\pushQED{\qed}\exampleonlyheader}
  {\popQED\endexampleonlyheader}
\newenvironment{rem}
  {\pushQED{\qed}\remarkonlyheader}
  {\popQED\endremarkonlyheader}
\newenvironment{dfn}
  {\pushQED{\qed}\definitiononlyheader}
  {\popQED\enddefinitiononlyheader}
\newenvironment{defn}
  {\pushQED{\qed}\definitiononlyheader}
  {\popQED\enddefinitiononlyheader}
\newenvironment{remark}
  {\pushQED{\qed}\remarkonlyheader}
  {\popQED\endremarkonlyheader}
\newenvironment{definition}
  {\pushQED{\qed}\definitiononlyheader}
  {\popQED\enddefinitiononlyheader}

\theoremstyle{plain}
\newtheorem{Thm}{Theorem}

\def\sl{{\mathfrak{sl}}}
\def\Gg{{\mathfrak{g}}}

\def\GL{{\rm{GL}}}

\def\Rep{\mathcal{R}ep}

\def\FF{\mathbb{F}}
\def\ZZ{\mathbb{Z}}
\def\NN{\mathbb{Z}_{\geq0}}
\def\CC{\mathbb{C}}
\def\PP{\mathbb{P}}
\def\ll{\mathcal{L}}
\def\hh{\mathcal{H}}
\def\xx{\mathcal{X}}
\def\jj{\mathcal{J}}
\def\uu{\mathcal{U}}
\def\vv{\mathbf{v}}
\def\yy{\mathcal{Y}}
\def\cc{\mathcal{C}}

\newcommand{\uld}{{\underline{d}}}

\newcommand{\im}{\mathrm{Im}}

\newcommand{\kk}{\mathbb{F}_q}
\newcommand{\cfone}{{\mathbbm{1}}}
\newcommand{\cf}{{\Theta}}
\newcommand{\sph}{{\mathrm{sph}}}

\newcommand{\qrud}{{Q_{\mathrm{Rud}}}}
\def\mod{\mathrm{-Mod}}

\DeclareMathOperator{\Hom}{Hom}
\DeclareMathOperator{\Ext}{Ext}
\DeclareMathOperator{\End}{End}
\DeclareMathOperator{\Aut}{Aut}

\usepackage{color}

\usepackage{xcolor}
\definecolor{dark-red}{rgb}{0.7,0.25,0.25}
\definecolor{dark-blue}{rgb}{0.15,0.15,0.55}
\definecolor{medium-blue}{rgb}{0,0,0.65}
\hypersetup{
    colorlinks, linkcolor={dark-red},
    citecolor={dark-blue}, urlcolor={medium-blue}
}

\allowdisplaybreaks

\begin{document}

\title{Hall algebras and shifted quantum affine algebras}
\author{P. Goyal and P. Samuelson }

\maketitle


\begin{abstract}
    In \cite{FT19}, Finkelberg and Tsymbaliuk introduced the notion of shifted quantum affine algebras and described their role in the study of quantized Coulomb branches associated to certain 3D $N = 4$ quiver gauge theories. We describe a new geometric construction of a deformation of one of these shifted quantum affine algebras as the Hall algebra of the category of representations of a certain quiver $Q_{\textrm{Rud}}$ (modulo relations). This quiver first arose in the work of Rudakov in the study of the tame blocks of the category of restricted representations of the Lie algebra $\mathfrak{sl}_2(\mathbb{F}_q)$. 
\end{abstract}

\tableofcontents



\section{Introduction}

The shifted quantum affine algebra $\uu_v(\widehat{\Gg_0})^{FT}_{\mu_1,\mu_2}$ was defined in \cite{FT19}  as a tool for algebraically describing $K$-theoretic Coulomb branches of 3D $N=4$ SUSY quiver gauge theories. Our main theorem realizes (a deformation of) the algebra $\uu_v(\widehat{\sl_2})^{FT}_{1,1}$ inside the Hall algebra of the category of representations of a certain quiver $\qrud$ (modulo relations)\footnote{We actually work with a twisted and localized Hall algebra. Also, there are two versions of the affine algebra: the simply connected one and the adjoint one. We ignore these technicalities for now.}. Below we describe these objects in more detail, state a more precise version of our results, and describe some future directions.

\subsection{Hall Algebras}
The \emph{Hall algebra} of a category $\mathcal C$ can be viewed as a convolution algebra on the space $\mathcal{O}(\mathrm{Ob}(\mathcal C))$, where we view the set $\mathrm{Ob}(\mathcal C)$ of objects as some kind of geometric space. The convolution product comes from the diagram 
\begin{equation}\label{eq:hallconv}
\begin{tikzcd}[arrow style=tikz,>=stealth,row sep=4em]
\mathrm{Exact} \arrow[r, "\pi_2"]
\arrow[d, "\pi_1"] & \mathrm{Ob}(\mathcal C)\\
\mathrm{Ob}(\mathcal C)\times \mathrm{Ob}(\mathcal C)
\end{tikzcd}
\end{equation}
where the map $\pi_1$ sends a short exact sequence $M \to E \to N$ to the pair $(M,N)$, and $\pi_2$ sends the same sequence to $E$. For this construction to make sense, the category $\mathcal C$ needs a notion of exactness (e.g.~exact sequences in an abelian category, exact triangles in a derived category). We also need a notion of finiteness to ensure that the map $\pi_2$ is proper. 

There are several levels of generality for which the Hall algebra construction has been studied: point-counting, derived, motivic, cohomological, $K$-theoretic, etc. The  nicely written survey \cite{ref5} has become the canonical introduction to Hall algebras;  other natural surveys with complementary points of view are \cite{Bri12, Bri18, Dyc18}. In what follows, we will use the point-counting version of the Hall algebra, since this is most amenable to direct computation (see Definition \ref{def:hallalg}); in this version, we view $\mathrm{Ob}(\mathcal C)$ as a discrete space, which requires quite strong finiteness assumptions. Precisely, we will assume $\mathcal{C} $ is \emph{finitary}, which (for us) means $\mathcal{C}$ is abelian and linear over a finite field $\FF_q$, with  $\Hom_{\mathcal{C}}(M,N)$ and $\Ext^1_{\mathcal{C}}(M,N)$  finite dimensional for all  $M,N$.

One of the first  theorems about Hall algebras relates them to quantum groups. Given an acyclic quiver $Q$, let $\mathcal{C}_Q=\Rep_{\FF_q}(Q)$ be the category of representations of $Q$ over $\FF_q$, and let $\mathcal{U}_v(\mathfrak{g}^+_Q)$ be the positive part of the  quantum Kac-Moody algebra associated to $Q$. A combination of results by Ringel and Green \cite{Rin90,Gre95} constructs an injective bialgebra map
$\mathcal{U}_v(\mathfrak{g}^+_Q) \hookrightarrow Hall(\mathcal{C}_Q)$ which specializes the formal parameter $v$ to $q^{1/2}$. This result has had far-reaching consequences and generalizations; one of the most influential applications is the use of this realisation by Lusztig to define canonical bases for quantum groups and their irreducible representations with quite remarkable positivity and integrality properties \cite{Lus90, Lus91, Lus92} (see also the survey \cite{SchCan}).

The results of Ringel and Green predict that it would be fruitful to study the Hall algebras of other non-semisimple
finitary categories; a natural candidate is the category of sheaves on a smooth, projective curve $X/\FF_q$. When $X = \PP^1$, Kapronov \cite{Kap97} found a map from the positive half $\mathcal{U}_v(\ll\sl_2^+)$ of the  quantum loop algebra to $Hall(Coh(\PP^1))$,  which inspired several further results \cite{BK01, BS12affine, Sz06,ref1,ref2,ref3, ref4}.
In particular, in \cite{BS12} Burban and Schiffmann computed the Hall algebra of an elliptic curve over $\FF_q$ and showed it is the positive half of the so-called \emph{quantum toroidal $\mathfrak{gl}_1$}; this result has found applications in algebraic combinatorics, knot theory, mathematical physics, and more \cite{GN15, Mel22, Mel21, MS17, SV13}. 
Hall algebras of higher genus curves have also been studied extensively \cite{Sch11,SV12,BS13,KSV12}.

All the theorems above only realize ``halves'' of various quantum algebras; it is quite natural to ask whether there is a similar realization of the entire algebra instead. There was considerable progress on this question \cite{Kap98,PX97,PX00,To06}, but a complete answer was given by Bridgeland in \cite{Bri13}, which can be phrased as follows. Define the \emph{Bridgeland quiver}
\begin{equation}\label{eq:bridgelandq}
    Q_{\mathrm{Bri}} := \left[\begin{tikzcd}[arrow style=tikz,>=stealth,row sep=4em]
\bullet \arrow[rr, shift left = 1ex, "d"]
&&
\arrow[ll,shift left = 1ex, "d'"]
\bullet
\end{tikzcd}
\right],\qquad dd'=d'd =0.
\end{equation}
Given a quiver $Q$ with no oriented cycles, let $\mathrm{Fun}(Q_{\mathrm{Bri}},Q)$ be the abelian category of functors   from $Q_{\mathrm{Bri}}$ (viewed as a 2-object category) into $\Rep(Q)$. Then in \cite{Bri13}, Bridgeland constructed an injective algebra map $\mathcal{U}_v(\mathfrak{g}_Q) \hookrightarrow Hall(\mathrm{Fun}(Q_{\mathrm{Bri}},Q))$ (up to some technical details, see Section \ref{sec:hallbackground}). We also note that the realizations of entire quantum groups have been constructed via other (orthogonal) strategies, for example through the geometry of quiver varieties \cite{Na1, Na2, Na3} and through categorification via $2$-categories \cite{KL09,KL11,Rou11}.


Another natural source of finitary non-semisimple categories comes from positive characteristic representation theory. In this paper, we initiate the study of Hall algebras of such categories by computing the Hall algebra in the simplest example: the category of equi-characteristic restricted representations of $\sl_2(\FF_q)$. Surprisingly, we show that it receives a map from a deformation of an \emph{entire} (shifted) quantum affine algebra, and not just the positive half.

\subsection{Restricted representations of $\sl_2(\FF_q)$}

Let $\mathcal C$ be the category of (equi-characteristic, finite dimensional) restricted representations of $\sl_2(\FF_q)$, where $p = \mathrm{char}(\FF_q)$ is assumed to be odd in this section. Equivalently, $\mathcal C$ is the category of $\sl_2(\FF_q)$-representations with trivial $p$-character; explicitly, these are the representations which are annihilated by $e^p, f^p$, and $h^p-h$. In general, for an arbitrary reductive Lie algebra $\Gg$ over $\FF_q$ and a fixed $p$-character $\chi$, the category of $\Gg$-representations with $p$-character $\chi$ is wild. However when $\Gg=\sl_2$ and $\chi=0$,   the category $\cc$ turns out to be tame. In fact, it has an elementary quiver theoretic description (see \eqref{eq:introrud} below) which we will be using throughout this paper (also see Remark~\ref{rem:GP}.)

Just as in the characteristic 0 case, let $V_n := \FF_q[x,y]_n$ be the \emph{Weyl module} (where $\sl_2(\FF_q)$ acts in the standard way). Unlike the characteristic 0 case, the category $\mathcal{C}$ is not semisimple, making it interesting from the point of view of Hall algebras. By the works of \cite{Jac41, Pol68}, the category $\mathcal C$ splits up into blocks:
\begin{equation*}
\mathcal C \simeq Vect \oplus \bigoplus_{i=0}^{(p-3)/2} \mathcal C_i,
\end{equation*}
where block $\mathcal C_i$ contains exactly 2 simple Weyl modules ($V_i$ and $V_{p-i-2}$). In \cite{Rud82}, Rudakov showed that each block $\mathcal C_i$ is equivalent to the category of representations of the following quiver (modulo relations), which we refer to as the \emph{Rudakov quiver}:

\hfill
\begin{minipage}{0.4\textwidth} 
\begin{equation*}
   \qrud := \left[\begin{tikzcd}[arrow style=tikz,>=stealth,row sep=4em]
\bullet \arrow[rr, shift left = 3ex, "e"]
\arrow[rr,shift left = 1ex, "f"]
&& \arrow[ll, shift left = 1ex, swap, "e'"]
\arrow[ll,shift left = 3ex, swap, "f'"] \bullet
\end{tikzcd}\right], \qquad \textrm{modulo the relations}
\end{equation*}
\end{minipage}
\hfill 
\begin{minipage}{0.4\textwidth} 
    \begin{align} 
 ee'=e'e&=ff'=f'f=0,\notag\\
 ef'&=-fe', \label{eq:introrud}\\
 e'f&=-f'e.\notag
    \end{align}
\end{minipage}
\hfill


Rudakov also classified indecomposable representations of $\qrud$ (see Section \ref{sec:Kron}); except for two modules $M$ and $M'$, every other indecomposable $N$ comes from one of the two Kronecker subquivers (precisely, in $N$ either $e'=f'=0$ or $e=f=0$). Translated in terms of $\sl_2(\FF_q)$-representations, Rudakov's results say that every indecomposable representation in $\cc_i$ (except $M$ and $M'$) is a submodule (as well as a quotient) of a Weyl module. The modules $M$ and $M'$ are both projective and injective, and are central elements in the Hall algebra (see Section \ref{sec:hallcomputations}). In terms of $\sl_2(\FF_q)$-representations, the modules $M$ and $M'$ correspond to the principal indecomposable modules that cover the Weyl modules $V_i$ and $V_{p-i-2}$ respectively.


\subsection{Shifted affine algebras}

Morally, the category of $\qrud$-representations is a doubled version of the category of representations of the Kronecker quiver. Based on the theorems of Ringel, Green, and Bridgeland described above, one might guess that the Hall algebra of the category of $\qrud$-modules realizes the whole quantum affine algebra $\uu_v(\widehat{\mathfrak{sl}_2})$, since its positive half is realized by the Hall algebra of the Kronecker quiver.

While this guess is not exactly correct, we show that the Hall algebra of
$\qrud$-modules does give a realization of a closely related algebra. There exist shifted versions $\uu_{v}(\widehat{\mathfrak{sl}_2})_{m,n}^{FT}$ of the quantum affine algebra defined by Finkelberg and Tsymbaliuk \cite{FT19}, depending upon parameters $m,n\in \ZZ$ (or on a pair of co-weights $\mu_1$ and $\mu_2$ for an arbitrary simple Lie algebra $\Gg_0$). These algebras were introduced in the study of quantized $K$-theoretic Coulomb branches of 3D $N = 4$ SUSY quiver gauge theories. The term ``shifted" originates from the theory of shifted Yangians introduced in \cite{BK06, KWWY14, BFN}, 
and these algebras can be viewed as trigonometric counterparts of shifted Yangians. Shifted quantum affine algebras have become a subject of wide interest in geometric representation theory and have deep connections to quantum integrable systems, cyclotomic Hecke algebras, and cluster algebras (see \cite{FT19, FT20, Hern, GHL, SWSQ, VV25}).


On a very concrete level, the shifted quantum affine algebra $\uu_{v}(\widehat{\Gg_0})_{\mu_1,\mu_2}^{FT}$ can be seen as a variation of the new Drinfeld presentation of the quantum affine algebra $\uu_v(\widehat{\Gg_0})$: by a theorem due to Drinfeld and Beck (see Theorem~\ref{thm:BD}), the algebra $\uu_v(\widehat{\Gg_0})$ is isomorphic to the quantum loop algebra $\uu_v(\ll\Gg)$. The algebra $\uu_v(\ll\Gg)$ has a triangular decomposition, and the algebra $\uu_{v}(\widehat{\Gg_0})_{\mu_1,\mu_2}^{FT}$ is obtained by gluing together the pieces of this decomposition after ``translating'' them by the pair of co-weights $(\mu_1,\mu_2)$. This translation procedure has the effect of making the loop generators of the algebra $\uu_{v}(\widehat{\Gg_0})_{\mu_1,\mu_2}^{FT}$ more commutative or more non-commutative depending on the dominance behavior of $\mu_1+\mu_2$.
(See Section~\ref{sec:quant} for more precise details.) In the unshifted case $\mu_1=\mu_2=0$, the algebra $\uu_{v}(\widehat{\Gg_0})_{0,0}^{FT}$ is isomorphic to a central extension of the usual quantum affine algebra without central charge. 


\subsection{Main results}
Following the loose analogy with Bridgeland's result, we define $\hh$ to be the localization\footnote{Like Green, we use a twisted product Hall algebra, and we also adjoin formal $-1/4$ roots of $[M]$ and $[M']$.} of the Hall algebra $Hall(\qrud)$ at the (central) elements $[M]$ and $[M']$. In Section \ref{sec:quant} we define a $\mathbb{C}(v)$-algebra $\uu_v(\widehat{\sl_2})_{1,1}$ via explicit generators and relations. This algebra is a (flat) deformation of the adjoint version of the shifted quantum affine algebra $\uu_v^{ad}(\widehat{\sl_2})_{1,1}^{FT}$ defined in \cite{FT19} (see Proposition~\ref{prop:deform}). Our first main theorem relates $\uu_v(\widehat{\sl_2})_{1,1}$ to the Hall algebra $\hh$ as follows: 

\begin{Thm}[see Theorem \ref{thm:SQAA} and Corollary \ref{cor:basis}] \label{thm:mainmain}
    There exists an injective algebra map from the specialization $\uu_{v=q^{1/2}}(\widehat{\sl_2})_{1,1}$ to $\hh$, and the image is precisely the spherical subalgebra $\hh^\sph$ (defined below).
\end{Thm}

As a consequence of the above theorem, we define a $\CC[v,v^{-1}]$-algebra $\uu_v(\widehat{\sl_2})_{1,1}^{int}$, which is an integral form of the $\CC(v)$-algebra $\uu_v(\widehat{\sl_2})_{1,1}$ (see Definition~\ref{def:intform}). This integral form is a slight refinement of the one constructed in \cite{FT20} for the shifted quantum affine algebra.
 
As often happens with point-counting Hall algebras, the map $g$ is far from being surjective. Intuitively, the space of objects has positive dimension, so the number of objects depends on the order of the field, which is not a property that is desirable in representation theory. (Concretely, in our case, for every $[m:n] \in \PP^1$, we have a $\qrud $ representation where the vertices are assigned 1-dimensional spaces, $e'=f'=0$, and $e=m$ and $f=n$.) In the setting of the Ringel-Green theorem, the image is the \emph{composition subalgebra}, which is the subalgebra generated by the simple objects. However, as we explain in Remark \ref{rem:spherical}, the image of our map $g$ is \emph{not} contained in the composition subalgebra. To remedy this, for a finitary category $\cc$, we define the \emph{spherical subalgebra} $Hall^\sph(\cc) \subset Hall(\cc)$ to be the subalgebra generated by the elements
\begin{equation} \label{eq:sphsum}
\Theta_\beta := \sum_{Y\in\min(\beta)} \alpha(Y) [Y],
\end{equation}
where $\beta \in K_0(\cc)$, the coefficient $\alpha(Y)$ is the size of the automorphism group of $Y$, and the sum is over objects whose endomorphism ring has minimal dimension amongst all objects of class $\beta$. In the case of Hall algebras of coherent sheaves over a curve, our notion of spherical subalgebra agrees with the standard definition in the literature (see Remark \ref{rmk:sphericalthoughts}). 


Apart from the spherical subalgebra, the Hall algebra has several other interesting subalgebras that we also study. Let $\Sigma$ denote the set of all monic irreducible polynomials over $\FF_q$ (including the zero polynomial). 
For any $\phi\in\Sigma$, there exist two sequences of representations $R_{\phi}(n)$ and $R_{\phi}(n)'$ for $n\geq 1$ in $\Rep(\qrud)$, which correspond to regular representations of the Kronecker quiver with $e'=f'=0$ and $e=f=0$ respectively (see Section~\ref{sec:Kron}). Let $Hall_{\phi}$ be the subalgebra of $Hall(\qrud)$ that is generated by these two sequences of representations. 
\begin{Thm} \label{thm:mainHeis}
(see Corollary~\ref{cor:Heis}) The algebra $Hall_{\phi}$ is isomorphic to a quantum Heisenberg algebra. Furthermore, if $\phi, \pi\in \Sigma$ are polynomials of equal degree, the algebras $Hall_{\phi}$ and $Hall_{\pi}$ are canonically isomorphic.
\end{Thm}
Together, the spherical subalgebra $Hall^{\sph}$ and the algebras $Hall_{\phi}$ for $\phi\in\Sigma$ generate the entire Hall algebra $Hall(\qrud)$. 

While the regular representations aren't themselves in the spherical subalgebra, there are certain `averages' $R_d$ (resp. $R_d'$) of the elements $[R_{\phi}(n)]$ (resp. $[R_{\phi}(n)'])$ of dimension $(d,d)$ which indeed lie in $Hall^{\sph}$. (These averages are exactly the sums defined using equation~\eqref{eq:sphsum} above {and are already defined in literature, e.g.~in \cite{Sz06}).} Let $\hh^0$ be the subalgebra of $\hh$ generated by these elements $R_n,R_n',[M]^{\pm 1/4}$ and $[M']^{\pm 1/4}$. Then, analogously to the above theorem, we show that the algebra $\hh^0$ is a quantum Heisenberg algebra, which can be viewed as a quantization of the maximal commutative subalgebra of $\uu_v^{ad}(\widehat{\sl_2})_{1,1}^{FT}$ as a consequence of Proposition~\ref{prop:deform} (see Remark \ref{rem:Hayashi}). We use this observation and Theorem~\ref{thm:mainHeis} to give an independent proof of Proposition~\ref{prop:rel9}, which is a key step in the proof of Theorem~\ref{thm:mainmain}.

Finally, we establish some facts about the algebra $\uu_v(\widehat{\sl_2})_{1,1}$ that parallel those of $\uu_v^{ad}(\widehat{\sl_2})_{1,1}^{FT}$.
\begin{Thm}
The deformed quantum affine algebra $\uu_v(\widehat{\sl_2})_{1,1}$ has the following algebraic properties: 
\begin{enumerate} 
\item A $\CC(v)$-basis for the space $\uu_v(\widehat{\sl_2})_{1,1}$ is given by the set of ordered monomials in the defining generators of the algebra. (Corollary~\ref{cor:basis})
\item The center of the algebra $\uu_v(\widehat{\sl_2})_{1,1}$ is equal to the ring of Laurent polynomials in the central charge $C^{1/2}$. (Proposition~\ref{prop:center})
\item There exists a presentation for the algebra $\uu_v(\widehat{\sl_2})_{1,1}$ using four generators and finitely many relations. (Theorem~\ref{thm:finpres})
\end{enumerate}
\end{Thm}

The finite presentation for $\uu_v(\widehat{\sl_2})_{1,1}$ is an analog of the Levendorskii presentation of the Yangian \cite{Lev93} involving finitely many generators and relations, and is an extension of the finite presentation of the algebra $\uu_v(\widehat{\Gg_0})^{FT}_{\mu_1,\mu_2}$ in \cite{FT19} to the dominant co-weight case when $\Gg_0=\sl_2$.


\subsection{Future directions}




One natural way to generalize  our results  is to study Hall algebras of blocks in categories of restricted $\FF_q$-representations of a higher rank Lie algebra $\mathfrak{g}$. These blocks are typically wild, but Gordon and Premet \cite{GP02} have classified the blocks of finite and tame representation type. We intend to compute Hall algebras of these in future work. Categories of positive characteristic representations of finite groups of Lie type would also be natural candidates to study.

Another natural generalization is already suggested by our description of Bridgeland's results. The starting point of his construction is the category of 2-periodic complexes of $Q$-modules (where $Q$ is an acyclic quiver), but this is equivalent to the functor category $\mathrm{Fun}(Q_{\mathrm{Bri}},\Rep(Q))$. By analogy, in future work we will study the Hall algebra of the functor category $\mathrm{Fun}(\qrud, \Rep(Q))$ and relate it to higher rank shifted affine algebras. The precise details of this relationship will be subtle, but this Hall algebra already has some nice structural properties; for example, it receives an algebra map from $\uu_{\vv}(\widehat{\sl_2})_{1,1}$ for every vertex in $Q$. There is also an interesting model category structure on this functor category.

There are also a number of directions where results are harder to predict. For example, the spherical subalgebra strictly contains the composition subalgebra in many cases, but the Ringel-Green theorem only describes the composition subalgebra; describing the larger algebra explicitly in examples may be interesting. It would also be natural to study other versions of the Hall algebra construction for $\qrud$, such as the cohomological Hall algebra defined over $\CC$. 
Another natural question would be to study the theory of canonical bases for shifted quantum affine algebras: canonical basis constructions typically come from geometric/Hall frameworks, and our realization of $\uu_v(\widehat{\sl_2})_{1,1}$ inside an explicit Hall algebra is well-placed to be the starting point for a canonical basis theory in the shifted setting. Finally, one would also hope to find applications of our Hall-theoretic computations to finite characteristic representation theory.
Other natural questions could be to find applications of our Hall-theoretic computations to finite characteristic representation theory and to the theory of canonical bases for shifted quantum affine algebras. 

\subsection{Structure and notation}\label{sec:structurenotation}
In Section \ref{sec:background} we provide background and restate some results in the literature that we will use. We collect results of a homological nature in Section \ref{sec:homological}, and we use these to perform certain Hall algebra computations in Section \ref{sec:hallcomputations}. We construct the algebra map in the statement of Theorem~\ref{thm:mainmain} from the deformed quantum affine algebra in Section \ref{sec:themap} by showing our Hall algebra computations imply the necessary relations. Regular representations only appear as averages in our spherical algebra; we compute relations between specific regular representations in Section \ref{sect:regular}. We characterize the image of our map and prove injectivity in Section \ref{sec:imker}. Finally, in Section \ref{sec:flatpresentation} we provide a finite presentation of our algebra, show that it is a flat deformation of $\uu_v^{ad}(\widehat{\sl_2})^{FT}_{1,1}$, and compute its center.

Throughout the paper, we work over $\CC$ for concreteness, and all tensor products will be assumed to be over $\CC$, unless specified otherwise.
We also include a summary of the most important notation: 
\begin{enumerate} 
\item $\uu_v^{sc}(\widehat{\Gg_0})^{FT}_{\mu_1,\mu_2}$ -- Simply connected shifted quantum affine algebra (see Definition \ref{def:ftalg})
\\$\uu_v^{ad}(\widehat{\Gg_0})^{FT}_{\mu_1,\mu_2}$ -- Adjoint version of the shifted quantum affine algebra (see Definition \ref{def:ftalg})
\\$\uu_v(\widehat{\Gg_0})_{\mu_1,\mu_2}$ -- Deformed quantum affine algebra (see Definition \ref{def:shiftedloop}) 
\\$\uu_v(\widehat{\sl_2})_{1,1}^{int}$ -- Integral form of the deformed quantum affine algebra (see Definition~\ref{def:intform}) 
\item $\FF_q$ -- a finite field with $q$ elements
\\$v$ -- Formal parameter in quantum groups
\\$\vv$ -- Numerical value to which $v$ is specialized (usually chosen so that $\vv^2=q$)
\\$[n]$, $[n]!$, ${n \choose k}$ -- quantum numbers (see Section \ref{sec:quant})
\\$[x,y]_q := xy-qyx$ is the $q$ commutator (whose generalized Jacobi relation is in equation \eqref{eq:qjac})
\item $Q_K$ and $\qrud$ -- Kronecker quiver and Rudakov quiver (see Sections \ref{sec:Kron} and \ref{sec:Rud})\\
$P_k,I_k,R_\phi(k)$ -- indecomposable modules over $Q_K$ (and over $\qrud$, with $e'=f'=0$) (see Section \ref{sec:Kron})\\
$M,M'$ -- indecomposable $\qrud$ modules (but {not} $Q_K$ modules) (see Section \ref{sec:Rud})\\
$\Sigma$ -- $\{\text{Set of all monic irreducible polynomials over }\FF_q\} \cup \{0\}$ (see Definition \ref{def:sigma})
\item $Hall(\qrud)$ -- Hall algebra of the category of $\qrud$ representations over $\FF_q$
\\$Hall_\phi$ -- subalgebra of $Hall(\qrud)$ generated by $R_\phi(n)$ and $R_\phi(n)'$, for some $\phi \in \Sigma$ (see Definition \ref{def:hallphi})
\\$\hh$ -- The algebra $Hall(\qrud)$ with $\left( (q-1)[M]\right)^{-1/4}$ and $\left((q-1)[M']\right)^{-1/4}$ adjoined (see Definition \ref{def:hh})
\\$\hh^{tw}$ -- The algebra $\hh$ with twisted multiplication (see Definition \ref{def:twist})
\\$Hall^{\sph}$ -- Spherical subalgebra of $Hall(-)$ (see Definition \ref{def:maximal})
\\$\hh^{\sph}$ -- Spherical subalgebra of $\hh^{tw}$ 
\end{enumerate}

\noindent \textbf{Acknowledgements:} We would like to thank V.~Chari, C.~Grossack, A.~Negut, D.~Orr, R.~Rouquier, O.~Schiffmann,  O.~Tsymbaliuk, and A.~Weekes for helpful comments, discussions, and references. We would especially like to thank B. Tsvelikhovskiy for many helpful and enthusiastic discussions at the beginning of this project, and, in particular, for explaining the results of \cite{Rud82}. The work of the second author is partially supported by a Simons Travel Grant. 

\section{Background}\label{sec:background}

\subsection{Infinite dimensional quantum groups}\label{sec:quant}

We work over the field of complex numbers $\CC$ and fix an indeterminate $v$. Define the following $v$-integers and $v$-binomial coefficients:
\begin{align}
[n]&:=\frac{v^n-v^{-n}}{v-v^{-1}},\notag\\
[n]!&:=[n][n-1]\cdots[2][1],\notag\\
{n\choose k}&:=\frac{[n]!}{[k]![n-k]!}.\label{eq:qbinomial}
\end{align}
It is possible to specialize the above elements at $v=\vv$ for any $\vv\in\CC^\times$, since they are all Laurent polynomials in $v$. We will abuse the above notation and use it to refer to both the polynomials in $v$ as well as their specializations.

Fix a finite index set $I$. Let $A=(a_{ij})_{i,j\in I}$ be a symmetric generalized Cartan matrix and fix a realization $(\mathfrak{h},\mathfrak{h}^*,\Pi,\Pi^{\vee})$ for $A$. Let $\mathfrak{g}$ be the Kac-Moody algebra associated with the matrix $A$.

\begin{dfn}
Define the \emph{quantized Kac-Moody algebra} $\uu_v(\Gg)$ to be the $\CC(v)$-algebra generated by elements $\{E_i,F_i,K_i^{\pm1}: i\in I\}$ satisfying the following relations:
\begin{align*}
K_iK_j&=K_jK_i\\
K_iE_jK_i^{-1}&=v^{a_{ij}}E_j\\
K_iF_jK_i^{-1}&=v^{-a_{ij}}F_j\\
[E_i,F_j]&=\delta_{ij}\frac{K_i-K_i^{-1}}{v-v^{-1}}\\
\sum_{k=0}^{1-a_{ij}}(-1)^k{1-a_{ij} \choose k}& E_i^kE_jE_i^{1-a_{ij}-k}=0\\
\sum_{k=0}^{1-a_{ij}}(-1)^k{1-a_{ij} \choose k}& F_i^kF_jF_i^{1-a_{ij}-k}=0,
\end{align*}
(assuming $i\neq j$ in the last two relations).
\end{dfn}

In the above definition (and the rest of the paper), the binomial coefficients are the Gaussian ones defined in \eqref{eq:qbinomial}.

The algebra $\uu_v(\Gg)$ in the above definition, in fact, also has a Hopf structure defined as follows. The coproduct $\Delta$ is defined by the equations
\begin{align*}
\Delta(K_i)&=K_i\otimes K_i\\
\Delta(E_i)&=E_i\otimes 1+K_i\otimes E_i\\
\Delta(F_i)&= 1\otimes F_i + F_i\otimes K_i^{-1},
\end{align*}
the antipode $S$ is
\[S(K_i)=K_i^{-1},\quad  S(E_i)= -K_i^{-1}E_i,\quad  S(F_i)= -F_iK_i,\]
and the counit $\epsilon$ is
\[\epsilon(K_i)=1,\quad  \epsilon(E_i)=\epsilon(F_i)=0.\]

The algebra $\uu_v(\Gg)$ and its representation theory have been extensively studied, and are particularly well understood at least when $\Gg$ is a simple Lie algebra (see for example \cite{Jan,CP, Kas, LusQG, DCQG}). Furthermore, the algebra $\uu_v(\Gg)$ has an integral form, which allows us to specialize $v=\vv$ for any $\vv\in\CC^{\times}$ to get a $\CC$-Hopf algebra $\uu_v(\Gg)$. From the point of view of Hall algebras, the specialization at $v=q^{1/2}$ for a prime power $q$ is especially pertinent.

The next easiest case to study after simple Lie algebras, which will also be the example of interest to us, is when $\Gg$ is the affine Lie algebra $\widehat{\Gg_0}$ associated to a simple Lie algebra $\Gg_0$. The Lie algebra $\widehat{\Gg_0}$ can be identified with the universal central extension of the loop algebra $\Gg_0[t,t^{-1}]$. A quantization of this identification gives a new presentation of the quantum affine algebra $\uu_v(\widehat{\Gg_0})$, known as Drinfeld's new realization:

\begin{thm} \cite{Drin, Beck} \label{thm:BD}
When $\Gg_0$ is simple, the algebra $\uu_v(\widehat{\Gg_0})$ is isomorphic to the quantum loop algebra $\uu_v(\ll\Gg)$, which is the $\CC(v)$-algebra generated by elements $E_{i,l}, F_{i,l}, H_{i,n}, K_i^{\pm1}, C^{\pm 1/2}$ for $i\in I$ and $l\in\ZZ, n\in \ZZ\setminus\{0\}$ satisfying the following relations:
    \begin{align}
        C^{1/2} &\textrm{ is central} \label{rrel1}\\
        K_iE_{j,k} K_i^{-1} &= v^{a_{ij}} E_{j,k}\label{rrel2}\\
        K_iF_{j,k} K_i^{-1} &= v^{-a_{ij}} F_{j,k} \label{rrel3}\\
        K_iH_nK_i^{-1}&=H_n\label{rrel3.5}\\
        E_{i,k+1}E_{j,l}-v^{a_{ij}} E_{j,l} E_{i,k+1} &= v^{a_{ij}} E_{i,k}E_{j,l+1}-E_{j,l+1}E_{i,k}\label{rrel4}\\
        F_{i,k+1}F_{j,l}-v^{-a_{ij}} F_{j,l} F_{i,k+1} &= v^{-a_{ij}} F_{i,k}F_{j,l+1}-F_{j,l+1}F_{i,k}\label{rrel5}\\
        [E_{i,k},F_{j,l}] &= \delta_{i,j}\frac{C^{(k-l)/2} \Psi_{i,k+l} - C^{(l-k)/2}\Phi_{i,k+l}}{v-v^{-1}}\label{rrel6}\\
        [H_{i,l},E_{j,k}] &= \frac{[la_{ij}]}{l}C^{-\lvert l\rvert/2}E_{j,k+l}\label{rrel7}\\
        [H_{i,l},F_{j,k}] &= \frac{-[la_{ij}]}{l}C^{\lvert l\rvert/2}F_{j,k+l} \label{rrel8}\\
        [H_{i,l}, H_{j,k}] &= \delta_{i,j}\delta_{l,-k}\frac{[2l]}{l}\frac{C^l-C^{-l}}{v-v^{-1}}\label{rrel9}\\
        \textrm{Sym}_{k_1,k_2,\dots,k_n}\sum_{t=0}^{n:=1-a_{ij}}& (-1)^t{n\choose t} E_{i,k_1}E_{i,k_2}\dots E_{i,k_t} E_{j,l} E_{i,k_{t+1}}\dots E_{i,k_n} = 0\label{rrel10}\\
        \textrm{Sym}_{k_1,k_2,\dots,k_n}\sum_{t=0}^{n:=1-a_{ij}}& (-1)^t{n\choose t} F_{i,k_1}F_{i,k_2}\dots F_{i,k_t} F_{j,l} F_{i,k_{t+1}}\dots F_{i,k_n} = 0,\label{rrel11}
    \end{align}
(assuming $i\neq j$ in the last two relations),
    where the elements $\Psi_{i,k}$ and $\Phi_{i,k}$ are defined via the following generating series:
    \begin{align}
        \sum_{k\geq 0}\Psi_{i,k} u^k &= K_i \mathrm{exp}\left( (v-v^{-1}) \sum_{k=1}^\infty H_{i,k} u^k \right) \label{eq:untwistedseries} \\
        \sum_{k\geq 0}\Phi_{i,-k} u^k &= K_i^{-1} \mathrm{exp}\left( -(v-v^{-1}) \sum_{k=1}^\infty H_{i,-k} u^{k} \right). \notag
    \end{align}
\end{thm}

The central charge $C^{1/2}$ in the above algebra acts by $1$ on all finite dimensional representations of the algebra $\uu_v(\widehat{\Gg_0})$ (up to a twist), and thus we often consider the quotient of the above algebra by the relation $C^{1/2}=1$ (see next definition). 
When $\Gg_0 = \sl_2$, the index set $I$ consists of a single element, so we can (and will) suppress the first index from each generator in $\uu_v(\widehat{\sl_2})$ and write $E_l:=E_{i,l}$, etc. In this case, the isomorphism $\uu_v(\widehat{\sl_2}) \to \uu_v(\ll\sl_2)$ of the above theorem can be seen to be a consequence of the derived equivalence between the category of coherent sheaves over $\mathbb{P}^1$ and the category of representations of the Kronecker quiver (see \cite{BS12affine}).

The algebras that we are most interested in this paper are shifted versions of the quantum affine algebra $\uu_v(\widehat{\Gg_0})$ defined by Finkelberg and Tsymbaliuk \cite{FT19}, that depend upon the choice of a pair of co-weights $\mu_1,\mu_2\in \mathfrak{h}$. Given any co-weight $\mu$, we can define a sequence of integers $(b_i)_{i\in I}$ by evaluating $\mu$ at the simple roots, that is $b_i=\alpha_i^{\vee}(\mu)$.

\begin{defn}\cite{FT19}\label{def:ftalg} 
For any simple Lie algebra $\Gg_0$ and co-weights $\mu_1$ and $\mu_2$, the simply connected \emph{shifted quantum affine algebra} $\uu_v^{sc}(\widehat{\Gg_0})_{\mu_1,\mu_2}^{FT}$ is defined as the $\CC(v)$-algebra that is generated by elements $E_{i,l}, F_{i,l}, H_{i,n}, (K_i^{\pm})^{\pm1}$ for $i\in I$ and $l\in\ZZ, n\in \ZZ\setminus\{0\}$ satisfying the same relations as~\eqref{rrel1}-\eqref{rrel11} (with $C^{1/2}$ specialized to $1$),
where the elements $\Psi_{i,k}$ and $\Phi_{i,k}$ are defined via the following generating series (instead of the series in \eqref{eq:untwistedseries}):
    \begin{align*}
        \sum_{k\geq 0}\Psi_{i,k-b_{i,1}} u^k &= K_i^+ \mathrm{exp}\left( (v-v^{-1}) \sum_{k=1}^\infty H_{i,k} u^k \right) \notag \\
        \sum_{k\geq 0}\Phi_{i,-k+b_{i,2}} u^k &= (K_i^-)^{-1} \mathrm{exp}\left( -(v-v^{-1}) \sum_{k=1}^\infty H_{i,-k} u^{k} \right), \notag
    \end{align*}
    where $b_{i,j} = \alpha_i^{\vee}(\mu_j)$ for $i\in I$ and $j\in \{1,2\}$.

There is also an \emph{adjoint version} $\uu_v^{ad}(\widehat{\Gg_0})^{FT}_{\mu_1,\mu_2}$ of the shifted quantum affine algebra, which is obtained from $\uu_v^{sc}(\widehat{\Gg_0})_{\mu_1,\mu_2}^{FT}$ by introducing additional generators $\{(S_i^{\pm})^{\pm 1}: i\in I\}$ and relations:
\begin{align*}
(S_i^{\pm})^2\prod_{j\neq i}(S_j^{\pm})^{a_{ij}}&=K_i^{\pm},\\
[S_i^{\pm},S_j^{\pm}]=[S_i^{\pm},S_j^{\mp}]&=0,\\
[S_i^{\pm}, H_j]&=0,\\
S_i^{\pm}E_j(S_i^{\pm})^{-1} &= v^{\pm\delta_{ij}}E_j,\\
S_i^{\pm}F_j(S_i^{\pm})^{-1} &= v^{\mp\delta_{ij}}F_j.
\end{align*}
\end{defn}

For $\Gg_0=\sl_2$, the above relations simplify slightly and have been restated in Section $\ref{sec:defat}$. Furtheremore, in that case, the choice of a pair of co-weights $(\mu_1,\mu_2)$ is equivalent to a choice of a pair of integers $(b_1,b_2)$.

It is clear that for all $i\in I$, the elements $K_i^+K_i^-$ are central in the algebra $\uu_{v}^{sc}(\widehat{\Gg_0})^{FT}_{\mu_1,\mu_2}$. When $\mu_1=\mu_2=0$, we have the following isomorphism with the usual quantum affine algebra:
\[\frac{\uu_{v}^{sc}(\widehat{\Gg_0})^{FT}_{0,0}}{(K_i^+K_i^--1)_{i\in I}} \cong \frac{\uu_v(\widehat{\Gg_0})}{(C^{1/2}-1)}.\]
One way to think about these algebras is that the co-weights $\mu_1$ and $\mu_2$ dictate how non-commutative the loop generators of the algebra $\uu_{v}^{sc}(\widehat{\Gg_0})^{FT}_{\mu_1,\mu_2}$ are. For example, consider the relation~\eqref{rrel6}:
\[[E_{i,k},F_{i,l}] = \frac{C^{(k-l)/2} \Psi_{i,k+l} - C^{(l-k)/2}\Phi_{i,k+l}}{v-v^{-1}}.\]
Note that, by definition, $\Psi_{i,n}=0$ if $n<-b_{i,1}$ and $\Phi_{i,n}=0$ if $n>b_{i,2}$. Thus, if the co-weights $\mu_1$ and $\mu_2$ are chosen to be anti-dominant, the integers $b_{i,j}'s$ are all negative, and then, in the range $b_{i,2}<k+l<-b_{i,1}$, the right side of the above commutator vanishes. On the other hand, if $\mu_1$ and $\mu_2$ are dominant, the right side of the above commutator never vanishes, and therefore, fewer commutators vanish compared to the usual quantum affine algebra, as will be the case that is relevant to this paper. (Nevertheless, we do note here that the algebra $\uu_{v}^{sc}(\widehat{\Gg_0})^{FT}_{\mu_1,\mu_2}$ is actually a subalgebra of $\uu_{v}^{sc}(\widehat{\Gg_0})^{FT}_{0,0}$ when $\mu_1+\mu_2$ is dominant (see \cite[Proposition 10.9]{FT19}), but we won't be needing that fact in this paper.) 


We note a structural result about these algebras from \cite{FT19} here which will be useful for us. Let $\uu^{> 0}$ $\uu^{<0}$ and $\uu^{0}$ 
be the subalgebras of $\uu_{v}^{sc}(\widehat{\Gg_0})^{FT}_{\mu_1,\mu_2}$ generated by the sets $\{E_{i,l}: i\in I, l\in \ZZ\}$, $\{F_{i,l}: i\in I, l\in \ZZ\}$ and $\{H_{i,n}, (K_i^{\pm})^{\pm1}: i\in I, n\in \ZZ\setminus\{0\}\}$  respectively.

\begin{prop} [{\cite[Proposition 5.1]{FT19}}]
The multiplication map
\[\uu^{>0}\otimes \uu^0\otimes\uu^{< 0}\to \uu_{v}^{sc}(\widehat{\Gg_0})^{FT}_{\mu_1,\mu_2}\]
is an isomorphism of $\CC(v)$-vector spaces. Furthermore, the algebras $\uu^{> 0}$, $\uu^{< 0}$ and $\uu^0$ are respectively generated by the sets $\{E_{i,l}: i\in I, l\in \ZZ\}$, $\{F_{i,l}: i\in I, l\in \ZZ\}$ and $\{H_{i,n}, (K_i^{\pm})^{\pm1}: i\in I, n\in \ZZ\setminus\{0\}\}$ with defining relations given by $\{\eqref{rrel4}\}$, $\{\eqref{rrel5}\}$ and $\{\eqref{rrel3.5},\eqref{rrel9}\}$ (with $C^{1/2}$ specialized to $1$).
\end{prop}

\begin{rem}
When $\Gg_0=\sl_2$, a strengthening of the above proposition is given by the fact that ordered monomials in the defining generators form a $\CC(v)$-basis for the algebra $\uu_{v}^{sc}(\widehat{\Gg_0})^{FT}_{\mu_1,\mu_2}$. (For details, see \cite[Theorem E.2]{FT19} or the proof of Proposition~\ref{prop:deform}.)
\end{rem}

Finally, we define the deformed quantum affine algebra which will be our main algebra of interest.

\begin{defn} \label{def:shiftedloop}
For any simple Lie algebra $\Gg_0$ and co-weights $\mu_1$ and $\mu_2$, the \emph{deformed quantum affine algebra} $\uu_v(\widehat{\Gg_0})_{\mu_1,\mu_2}$ is defined as the $\CC(v)$-algebra that is generated by elements $E_{i,l}, F_{i,l}, H_{i,n}, K_i^{\pm1}, S_i^{\pm1}, C^{\pm1/2}$ for $i\in I$ and $l\in\ZZ, n\in \ZZ\setminus\{0\}$ satisfying the same relations as~\eqref{rrel1}-\eqref{rrel11}, where the elements $\Psi_{i,k}$ and $\Phi_{i,k}$ are defined via the following generating series (instead of the series in \eqref{eq:untwistedseries}):
    \begin{align*}
        \sum_{k\geq 0}\Psi_{i,k-b_{i,1}} u^k &= K_i \mathrm{exp}\left( (v-v^{-1}) \sum_{k=1}^\infty H_{i,k} u^k \right) \notag \\
        \sum_{k\geq 0}\Phi_{i,-k+b_{i,2}} u^k &= K_i^{-1} \mathrm{exp}\left( -(v-v^{-1}) \sum_{k=1}^\infty H_{i,-k} u^{k} \right), \notag
    \end{align*}
    where $b_{i,j} = \alpha_i^{\vee}(\mu_j)$ for $i\in I$ and $j\in \{1,2\}$.
Additionally, we also impose the relations:
\begin{align*}
  S_i^2\prod_{j\neq i}S_j^{a_{ij}}&=C^{1/2}K_i,\\
[S_i,S_j]&=0,\\
[S_i, H_j]&=0,\\
S_iE_jS_i^{-1} &= v^{\delta_{ij}}E_j,\\
S_iF_jS_i^{-1} &= v^{-\delta_{ij}}F_j.
\end{align*}
\end{defn}

\begin{rem}
The main difference between the above definition and that of the adjoint version of the shifted quantum affine algebra $\uu_v^{ad}(\widehat{\Gg_0})^{FT}_{\mu_1,\mu_2}$ is the presence of the central element $C^{1/2}$, which we do not specialize to 1 here. It is easy to see that the above algebra is actually defined over the ring $\CC[v,v^{-1},(v-v^{-1})^{-1}]$. Therefore, we can talk about the specialization of this algebra $\uu_{\vv}(\widehat{\Gg_0})_{\mu_1,\mu_2}$ at $v=\vv$ for any $\vv\in\CC\setminus\{-1,0,1\}$. We find an integral form for the algebra $\uu_{\vv}(\widehat{\sl_2})_{1,1}$ in Section~\ref{sec:intform}, which will in fact allow the specialization at any non-zero $\vv$. 
\end{rem}

In this paper, we will solely be interested in the case when $\Gg_0=\sl_2$. Therefore, we will  identify the co-weights $\mu_1$ and $\mu_2$ with integers $b_1$ and $b_2$. Furthermore, we will mostly be interested in the case $b_1=b_2=1$ through the paper (except in Section $\ref{sec:finpres}$).


\subsection{Hall algebras}\label{sec:hallbackground}
 We will use the following definition of the Hall algebra in the remainder of the paper.

\begin{dfn}\label{def:hallalg}
    Let $\mathcal C$ be an abelian category which is linear over $\FF_q$, and assume that $\mathcal C$ is \emph{finitary} (which means $\Hom(M,N)$ and $\Ext^1(M,N)$ are finite dimensional for all $M,N$). The \emph{Hall algebra} of $\mathcal C$ is the vector space $Hall(\mathcal C) := \mathbb{C}\{X:X \in \mathrm{Iso}(\mathcal C)\}$ with basis the set of isomorphism classes of objects. The structure constants are given by the following formulas:
    \begin{equation*}
    [M][N] := \sum_{[E]} c_{M,N}^E [E],\qquad \qquad c_{M,N}^E := \# \{ E' \subset E\mid E' \simeq N \textrm{ and } E/E' \simeq M\}.
    \end{equation*}
    (Both the sum and the structure constants are finite by the finitary condition.)
\end{dfn}
It is an enjoyable but elementary exercise to show the associativity of this product, and that the structure constants come from the convolution diagram \eqref{eq:hallconv}, where the space of objects is viewed as a discrete space. We also note that a rescaling of the generators $[X] \mapsto \widetilde{[X]} := \alpha(X)[X]$ (where $\alpha(X)$ is the size of the automorphism group of $X$) leads to the following structure constants which are sometimes more convenient to compute: 
\[
\widetilde{c}_{\widetilde{M},\widetilde{N}}^{\widetilde{E}} = \frac{\lvert \Ext^1_\cc(M,N)_E\rvert}{\lvert \Hom_\cc(M,N)\rvert},
\]
where $\Ext^1_\cc(M,N)_E$ is the subset of extensions whose middle term is isomorphic to $E$ (see, e.g.~\cite[Sec.~2.3]{Bri13}).

It is evident from the definition that $Hall(\cc)$ is graded by the Grothendeick group $K_0(\cc)$. This grading allows us to twist the multiplication in the Hall algebra. Precisely, when $\cc$ is hereditary, Ringel and Green defined the following twisted product:
\[
[X]\diamond [Y] := \frac{\lvert \Hom_\cc(X,Y)\rvert }{\lvert \Ext^1_\cc(X,Y)\rvert}[X][Y].
\]
We too will twist our Hall algebra before constructing a map from the (twisted affine) quantum group, but our twist is slightly different from Green's (see Definition \ref{def:twist}).

One of the first modern theorems about Hall algebras was proved by Ringel and extended by Green. Given a quiver $Q$, let $\mathcal{U}_v(\mathfrak{g}^+_Q)$ be the positive part of the associated quantum Kac-Moody algebra. 
\begin{thm}[{\cite{Rin90,Gre95}}]
    Let $Q$ be a quiver with no oriented cycles and vertex set $I$, and let $\mathcal C_Q$ be the category of representations of $Q$ over $\FF_q$. Then there is an injective algebra map 
    \[
    \mathcal{U}_v(\mathfrak{g}^+_Q) \hookrightarrow Hall^{tw}(\mathcal{C}_Q),\qquad E_i \mapsto [S_i]
    \]
    where $E_i$ is the standard generator associated to vertex $i$, and $S_i$ is the simple module at vertex $i$, and where the parameter $v$ is specialized to $q^{1/2}$. Furthermore, this map is an isomorphism precisely when $Q$ is some orientation of a simply-laced Dynkin diagram.
\end{thm}

In \cite{Bri13}, Bridgeland found a realization of the entire quantum group using periodic complexes as follows. Given a quiver $Q$ with no oriented cycles, let $C(Q)$ be the abelian category of $\mathbb{Z}_2$-graded complexes in $\Rep(Q)$, and let $C(P)$ be the subcategory of complexes of projective objects. Let $DH(Q)$ be the localization of the (twisted) Hall algebra of $C(P)$ at the set of acyclic complexes (so $[M_\bullet]$ becomes invertible if $H_\ast(M_\bullet) = 0$). Finally, let $DH_{\mathrm{red}}(Q)$ be the quotient by the relations $[M] = 1$ for all acyclic, shift-invariant complexes $M$.
\begin{thm}[{\cite{Bri13}}]\label{thm:bridgeland}
    There is an injective algebra map $\mathcal{U}_v(\mathfrak{g}_Q) \hookrightarrow DH_{\mathrm{red}}(Q)$ specializing $v$ to $q^{1/2}$. Again, this map is an isomorphism precisely when $Q$ is some orientation of a simply-laced Dynkin diagram.
\end{thm}



\subsection{The Hall algebra of the Kronecker quiver} \label{sec:Kron}


In this section we summarize some results in the literature about the Kronecker quiver:
\begin{equation}
    Q_K := \left[\begin{tikzcd}[arrow style=tikz,>=stealth,row sep=4em]
\bullet \arrow[rr, shift left = 1ex, "e"]
\arrow[rr,shift right = 1ex,swap, "f"]
&&
\bullet
\end{tikzcd}
\right].
\end{equation}

First, the classification of indecomposable objects in $\Rep(Q_K)$ is well known, and every indecomposable is in one of the following families:
\begin{itemize}
\item {\it Pre-projective representations}: For any $n\geq 0$, the pre-projective indecomposable representation $P_n$ is of dimension $(n,n+1)$ and is given by:
\[
P_n:= \begin{tikzcd}[arrow style=tikz,>=stealth,row sep=4em]
\FF_q^n \arrow[rr,shift left = 1ex, "e"]
\arrow[rr, shift right = 1ex, swap, "f"] && \FF_q^{n+1}
\end{tikzcd},
\]
where $e$ is the inclusion into the first $n$ coordinates of $\FF_q^{n+1}$ and $f$ is the inclusion into the last $n$.

\item {\it Pre-injective representations}: For any $n\geq 0$, the pre-injective indecomposable representation $I_n$ is of dimension $(n+1,n)$ and is given by:
\[
I_n:= \begin{tikzcd}[arrow style=tikz,>=stealth,row sep=4em]
\FF_q^{n+1} \arrow[rr,shift left = 1ex, "e"]
\arrow[rr, shift right = 1ex, swap, "f"] && \FF_q^{n}
\end{tikzcd},
\]
where $e$ is the projection onto the first $n$ coordinates of $\FF_q^{n+1}$ and $f$ is the projection onto the last $n$.

\item {\it Regular representations}: Fix any monic irreducible polynomial $\phi$ of degree $d$ over the field $\FF_q$. Let $M_{\phi}$ denote the companion matrix of $\phi$, which is a 
$d\times d$ matrix whose characteristic polynomial is $\phi$. Then, for any $n\geq 1$, we can construct the following $n\times n$ block matrix:
\[M_{\phi}(n):=\begin{bmatrix}
    M_{\phi}       & I_{d\times d} &  &  & \\
           & M_{\phi} & I_{d\times d} &  & \\
        & & .&.  & \\
         & & &.  &. \\
      &  & &  & M_{\phi}
\end{bmatrix},
\]
where $I_{d\times d}$ is the identity matrix and all the empty blocks are zero.
Then there exists an indecomposable representation $R_{\phi}(n)$ of dimension $(dn,dn)$ given by:
\[
R_{\phi}(n):= \begin{tikzcd}[arrow style=tikz,>=stealth,row sep=4em]
\FF_q^{dn} \arrow[rr,shift left = 1ex, "e = I_{dn\times dn}"]
\arrow[rr, shift right = 1ex, swap, "f = M_{\phi}(n)"] && \FF_q^{dn}
\end{tikzcd}.
\]
Apart from these, we have an additional family of regular representations given by:
\[
R_0(n):= \begin{tikzcd}[arrow style=tikz,>=stealth,row sep=4em]
\FF_q^{n} \arrow[rr,shift left = 1ex, "e = J_n"]
\arrow[rr, shift right = 1ex, swap, "f = I_{n\times n}"] && \FF_q^{n}
\end{tikzcd},
\]
where $J_n$ is a regular nilpotent matrix.
\end{itemize}

\begin{dfn} \label{def:sigma}
For later reference, define
\[\Sigma:= \{\text{monic irreducible polynomials over } \FF_q\} \cup \{0\}=: \bigcup_{d \geq 1} \Sigma_d,\]
where $\Sigma_d$ represents the subset of monic irreducible polynomials of degree $d$. Here, by convention, we put $0$ in the set $\Sigma_1$.
\end{dfn}

\begin{dfn}
We define the following sequence of elements $(R_n)$ in the Hall algebra:
\[R_n:= \frac{q^n}{q-1}\sum\left(1-\frac{1}{q^{l_1}}\right)\left(1-\frac{1}{q^{l_2}}\right)\cdots \left(1-\frac{1}{q^{l_r}}\right) [R_{\phi_1}(n_1)\oplus R_{\phi_2}(n_2)\oplus\cdots R_{\phi_r}(n_r)],\]
where the sum is over all finite sets $\{\phi_1, \phi_2,\cdots, \phi_r\}$ of monic irreducible polynomials such that $\deg(\phi_i)=l_i$ and $\sum l_in_i = n$. As a convention, we define $R_0 = \frac{1}{q-1}$.
\end{dfn}

\begin{ex}
    For the sake of concreteness, we note that 
    \[
R_1 = \sum_{(a,b) \in \mathbb{P}^1}  \begin{tikzcd}[arrow style=tikz,>=stealth,row sep=4em]
\FF_q \arrow[rr,shift left = 1ex, "e = a"]
\arrow[rr, shift right = 1ex, swap, "f = b"] && \FF_q
\end{tikzcd}.
\]
\end{ex}

It will be helpful in later sections to be able to write a more convenient expression for $R_n$. 
\begin{dfn}
To shorten notation, write $\underline{\phi} = (\phi_1,\cdots,\phi_r)$ and $\underline{n} = (n_1,\cdots,n_r)$, and write $\underline{\phi}\cdot \underline{n} = \sum_i (\deg(\phi_i)) n_i$. We also write $R_{\underline{\phi}}(\underline{n}) := R_{\phi_1}(n_1)\oplus \cdots \oplus R_{\phi_r}(n_r)$, and we will say $\underline{\phi}$ is \emph{minimal} if all the $\phi_i$ are distinct. 
\end{dfn}

This definition of minimality is generalized in Definition \ref{def:maximal}.
\begin{lem}\label{lem:rnexpression} 
    We have the identities 
    \begin{align*}
    R_n &= \frac{1}{q-1} \sum_{\underline{\phi} \cdot \underline{n}=m} {\alpha\left(R_{\underline{\phi}}(\underline{n})\right)} [R_{\underline{\phi}}(\underline{n})]\\
    [R_{\underline{\phi}}(\underline{n})] &= [R_{\phi_1}(n_1)]\cdots [R_{\phi_r}(n_r)],
    \end{align*}
    where the sum is taken over minimal $\underline{\phi}$ and  the product is taken inside the Hall algebra. 
\end{lem}
\begin{proof}
    By \cite[Chapter II, (1.6)]{Mac}, if $\phi \in \Sigma_d$ we have that $\alpha(R_\phi(k)) = |Z(M_{\phi}(k))|=q^{dk}(1-q^{-d})$ (where $Z(-)$ is the centralizer in $GL_{k\deg{\phi}}(\FF_q)$). Note that if $\underline{\phi} $ is minimal, then \cite[Lem.~1.1(g)]{Sz06} (recalled in Theorem \ref{thm:szanto} below) implies $\Hom(R_{\phi_i},R_{\phi_j})=0$ when $i\not= j$. Then bilinearity of the $\Hom(-,-)$ functors implies $\alpha(R_{\underline{\phi}})(\underline{n}) = \prod_i \alpha(R_{\phi_i}(n_i))$. We then compute
\begin{align*}
R_n & := \frac{q^n}{q-1}\sum\left(1-\frac{1}{q^{l_1}}\right)\left(1-\frac{1}{q^{l_2}}\right)\cdots \left(1-\frac{1}{q^{l_r}}\right) [R_{\phi_1}(n_1)\oplus R_{\phi_2}(n_2)\oplus\cdots R_{\phi_r}(n_r)]\\
&= \frac{1}{q-1}\sum \big\lvert Z(M_{\phi_1}(n_1))\big\rvert \cdots \big\lvert Z(M_{\phi_r}(n_r))\big\rvert [R_{\phi_1}(n_1)\oplus R_{\phi_2}(n_2)\oplus\cdots R_{\phi_r}(n_r)]\\
&= \frac{1}{q-1}\sum  \alpha\left((R_{\phi_1}(n_1)\oplus R_{\phi_2}(n_2)\oplus\cdots R_{\phi_r}(n_r)\right)  [R_{\phi_1}(n_1)\oplus R_{\phi_2}(n_2)\oplus\cdots R_{\phi_r}(n_r)]\\
&= \frac{1}{q-1}\sum  \alpha\left(R_{\underline{\phi},\underline{n}}\right)  [R_{\underline{\phi},\underline{n}}].
\end{align*}
The product identity follows from the fact that for minimal $\underline{\phi}$, there are no extensions between any of the summands (again see Theorem \ref{thm:szanto} below), which means the only term appearing in the Hall product is the direct sum. The coefficient is 1 by Lemma \ref{lem:sumcoeff}, since there are no nonzero homomorphisms between the summands (again by Theorem \ref{thm:szanto}).
\end{proof}

\begin{rem} \label{rem:irred}
For any point $\phi=[\phi_0:\phi_1:\cdots:\phi_n]\in \PP^n(\FF_q)$ with $\phi_0\neq 0$, we can view $\phi$ as a (not necessarily irreducible) monic polynomial of degree $n$ whose coefficients are $\phi_i/\phi_0$ for $1\leq i\leq n$. Let $M_{\phi}$ denote the unique (up to conjugation) $n\times n$ matrix with characteristic polynomial equal to $\phi$, that has a centralizer of minimal size amongst all such matrices. (That is, $M_{\phi}$ is a regular element in the Lie algebra $\mathfrak{gl}_n$.) Define the representation $R_{\phi}$ as
\[
R_{\phi}:= \begin{tikzcd}[arrow style=tikz,>=stealth,row sep=4em]
\FF_q^{n} \arrow[rr,shift left = 1ex, "e = I_{n\times n}"]
\arrow[rr, shift right = 1ex, swap, "f = M_{\phi}"] && \FF_q^{n}
\end{tikzcd}.
\]
Similarly, if $\phi_i=0$ for $0\leq i\leq k-1$, we can view $\phi$ as a polynomial $\phi'$ of degree $n-k$ and define
\[R_{\phi}:=R_{\phi'}\oplus R_0(k).\]
Then, a reformulation of the previous lemma says that the element $R_n$ can be rewritten as:
\[R_n=\frac{1}{q-1}\sum_{\phi\in\PP^n(\FF_q)} \alpha(R_{\phi}) [R_{\phi}].\]
\end{rem}

The Hall algebra of the Kronecker quiver (and its relation to the Hall algebra of coherent sheaves over $\mathbb{P}^1$) has been studied by a number of authors (see, e.g. \cite{Kap97, BK01, BS12affine, Sz06,ref1,ref2,ref3,ref4,ref5}). For the convenience of the reader, we now recall some of these results that we will use throughout the paper, translated into our notation.


\begin{thm}\label{thm:szanto}
    The following statements are proved in \cite{Sz06} Lemmas 1.1, 4.1 and Theorems 4.2, 4.3.
    \begin{enumerate}
        \item[(1.1a)] If $R$ is regular, $I$ is pre-injective, and $P$ is pre-projective, then 
        \[
        \Hom(R,P) = \Hom(I,P) = \Hom(I,R) = \Ext^1(P,R) = \Ext^1(P,I) = \Ext^1(R,I) = 0.
        \]
        \item[(1.1b)] If $\phi \not= \pi$, then \[ \Hom(R_\phi(n),R_\pi(m)) = 0 \,\textrm{ and } \,\Ext^1(R_\phi(n),R_\pi(m)) = 0.\]
        \item[(1.1c)] For $n \leq m$, we have $\dim \Hom(P_n,P_m)= m - n + 1$ and $\Ext^1(P_n,P_m)= 0$;
        otherwise $\Hom(P_n,P_m)= 0$ and $\dim \Ext^1(P_n,P_m)= n-m-1$. In particular
        $\End(P_n)\cong\FF_q$ and $\Ext^1(P_n,P_n)= 0$.
        \item[(1.1f)] If $\phi \in \Sigma_d$ and $r \geq 1$, then 
        \[ 
        \dim \Hom(P_n,R_\phi(r)) = \dim \Hom(R_\phi(r),I_n) = dr = \dim \Ext^1(R_\phi(r),P_n) = \dim \Ext^1(I_n,R_\phi(r)).
        \]
        \item[(1.1g)] If $\phi \in \Sigma_d$, then 
        \[\dim \Hom(R_\phi(m), R_\phi(n)) = \dim \Ext^1(R_\phi(m),R_{\phi}(n)) = d \min(m,n).\]
        \item[(4.1b)] If $X = R_\phi(1)$ with $\phi \in \Sigma_1$, then
        \[
        [X][P_n] = [P_{n+1}] + q[P_n][X],\qquad [I_n][X] = [I_{n+1}] + q[X][I_n].
        \]
        \item[(4.2)] If $i < j$, we have the identities
        \begin{align*}
            [P_i][P_j] &= [P_i \oplus P_j]\\
            [P_i^{\oplus n}][P_i^{\oplus m}] &=  \binom{m+n}{m }_q [P_i^{\oplus m+n}]\\
            [P_j][P_i] &= q^{j-i+1}[P_j\oplus P_i] + q^{j-i}(q-q^{-1})\sum_{\ell=1}^{\lfloor(j-i)/2\rfloor} [P_{j-\ell}\oplus P_{i+\ell}].
        \end{align*}
        \item[(4.3)] We have the identity
        \[
        [I_{n-1-i}][P_i] = R_n + q^{n-1} [P_i][I_{n-1-i}].
        \]
        \item[(4.9)] 
        If $\phi \in \Sigma_d$, we have the identity
        \[
        [I_m][R_\phi(k)] = q^{dk}[R_\phi(k)][I_m] + I_{m+dk} + \sum_{i=1}^{k-1}\left(q^{d(k-i)}-q^{d(k-i-1)}\right)[R_\phi(k-i)][I_{m+di}].
        \]
        \item[(4.12)] For $n \geq 1$ and $m \geq 0$ we have the identity
        \[R_n[P_m] = q^n[P_m]R_n + \sum_{i=1}^n (q^{n+i}-q^{n+i-2})[P_{m+i}]R_{n-i}.\]
    \end{enumerate}
\end{thm}

\subsection{The Rudakov quiver} \label{sec:Rud}
In this section we describe the quiver (with relations) $\qrud$ that Rudakov used in \cite{Rud82} to describe blocks in the category of equi-characteristic restricted representations of $\sl_2(\FF_q)$. 

\begin{dfn} The \emph{Rudakov quiver} is the following quiver-with-relations:
\begin{equation}\label{eq:qrud}
\begin{tikzcd}[arrow style=tikz,>=stealth,row sep=4em]
\bullet \arrow[rr, shift left = 3ex, "e"]
\arrow[rr,shift left = 1ex, "f"]
&& \arrow[ll, shift left = 1ex, swap, "e'"]
\arrow[ll,shift left = 3ex, swap, "f'"] \bullet
\end{tikzcd},\qquad \qquad ee'=e'e=ff'=f'f=0,\quad ef'=-fe', \quad e'f=-f'e.
\end{equation}
Throughout the paper we will refer to this quiver modulo relations as $\qrud$.
\end{dfn}
Rudakov classified indecomposable representations of this quiver in the following theorem, and we use this classification extensively throughout the paper.\footnote{Technically, Rudakov works with the quiver where the relations are $ef'=fe'$ and $e'f=f'e$, but it is clear that those relations give an equivalent category of representations, the equivalence being given by negating $f'$.} 
For the reader's convenience, we provide a self-contained proof of this classification theorem at the end of this section. 
\begin{thm}[{\cite{Rud82}}]\label{prop:class}
All but two of the indecomposable representations of $\qrud$ satisfy either $e=f=0$ or $e'=f'=0$. The other two indecomposables are the  objects $M$ and $M'$, described explicitly in equations \eqref{eq:mdef} and \eqref{eq:mpdef} below; they are both projective and both injective.
\end{thm}

A corollary of the above theorem is that all indecomposable representations of $\qrud$ except $M$, $M'$ and the two simples have Loewy length two.

\begin{rem}
    Because of the previous theorem, it is most efficient to reuse notation from the previous section to describe $\qrud$ modules; more precisely, if $Y$ is a module over the Kronecker quiver, we will write $Y$ for the $\qrud$ module with the same vector spaces and maps $e,f$, and with $e'=f'=0$. (See Example \ref{ex:mexact} for an instance of this notation.)
\end{rem}

Let $A$ be the path algebra of $\qrud$ and let $s_1,s_2 \in A$ be the idempotents projecting onto vertex 1 and 2, respectively.


\begin{defn}
For any $V\in\Rep(\qrud)$, we define the vector spaces:
\[V_1:=s_1 \cdot V, \qquad \qquad V_2:=s_2\cdot V.\]
We will often also use the shorthand $V=(V_1,V_2)$. Furthermore, we will use $e_V: V_1 \to V_2$ to refer to the linear map in $V$ corresponding to the action of $e\in\qrud$, etc. 
\end{defn}

The two representations $M$ and $M'$ of $\qrud$ in the above theorem both have dimension $(2,2)$. If we write $M_1=\langle x_1, x_2\rangle$ and $M_2=\langle y_1, y_2\rangle$, we can represent the maps in $M$ explicitly via the diagram
\begin{equation}\label{eq:mdef}
M := As_1 \cong \begin{tikzcd}
	{x_1} && {y_2} \\
	\\
	{x_2} && {y_1}
	\arrow["f"{pos=0.5}, from=1-1, to=1-3]
	\arrow["e"{pos=0.15}, from=1-1, to=3-3, swap]
	\arrow["{e'}"{pos=0.1}, from=1-3, to=3-1]
	\arrow["{-f'}"{pos=0.5}, from=3-3, to=3-1]
\end{tikzcd},
\end{equation}
where $x_1=s_1$. All the arrows that have not been drawn are equal to zero, and the arrow notation means $e\cdot x_1 = y_1$, $f'\cdot y_1 = -x_2$ etc.
Similarly, $M'$ is described by the following diagram:
\begin{equation}\label{eq:mpdef}
M' := As_2 \cong 
\begin{tikzcd}
	{x_2} && {y_1} \\
	\\
	{x_1} && {y_2}
	\arrow["{e}"'{pos=0.15}, from=1-1, to=3-3]
	\arrow["{-f'}"'{pos=0.5}, from=1-3, to=1-1]
	\arrow["{e'}"'{pos=0.1}, from=1-3, to=3-1, swap]
	\arrow["{f}"'{pos=0.5}, from=3-1, to=3-3]
\end{tikzcd}.
\end{equation}

(We use the above bases for $M$ and $M'$ throughout the paper, whenever there is no ambiguity.)




\begin{ex}\label{ex:mexact}
    Let us illustrate an explicit example of an exact sequence involving $M$:
    \[
    0 \to I_0 \to M \to P_1 \to 0.
    \]
    The image of the inclusion map is spanned by $x_2 \in M$, and the quotient is obviously isomorphic to $P_1$. This shows that the subcategory of $\Rep(\qrud)$ consisting of objects where $e'=f'=0$ is not closed under extensions, which is an important fact throughout the paper.
\end{ex}

There exist two types of dualities in our Hall algebra that will be quite helpful in simplifying some of the computations. We first define these on the level of representations. Fix a representation:
\[
\begin{tikzcd}[arrow style=tikz,>=stealth,row sep=4em]
V \arrow[rr, shift left = 3ex, "e"]
\arrow[rr,shift left = 1ex, "f"]
&& \arrow[ll, shift left = 1ex, swap, "e'"]
\arrow[ll,shift left = 3ex, swap, "f'"] W
\end{tikzcd},
\]
in $\Rep(\qrud)$, where $V$ and $W$ are $\FF_q$-vector spaces and $e,f,e'$ and $f'$ are linear maps that satisfy the relations in equation \eqref{eq:qrud}.

\begin{dfn}
We define two commuting endofunctors $\sigma,\tau:\Rep(\qrud)\to \Rep(\qrud)$, with $\sigma$ covariant and $\tau$ contravariant. The functor $\sigma$  swaps the vector spaces $V$ and $W$ and swaps the linear map $e$ with $e'$ and $f$ with $f'$. Next, let $\tau$ be the map that  replaces $V$ and $W$ by their duals and sends 
\[
\tau(e) = (e')^\ast,\quad \tau(f) = (f')^\ast,\quad \tau(e') = e^\ast,\quad \tau(f') = f^\ast.
\]
If $Y$ is a representation of $\qrud$, we often denote $\sigma(Y)$ by $Y'$.
\end{dfn}

It is clear that $\sigma$ and $\tau$ are both involutions that commute with each other. Also, we note/define: 
\begin{align*} 
\sigma(P_n) &=: P_n', & \sigma(I_n) &=: I_n',& \sigma(M) &= M',& \sigma(M') &= M\\
\tau(P_n)&=I_n',& \tau(I_n)&=P_n',& \tau(M) &= M,& \tau(M') &= M'\\
&&\sigma(R_{\phi}(n)) &=: R_{\phi}(n)',& \tau(R_{\phi}(n)) &= R_{\phi}(n)',
\end{align*}
for all $\phi\in \Sigma$.

\begin{rem} \label{rem:syzygy}
For non-projective modules, the functor $\sigma$ can also be expressed in terms of the syzygy functor: for an arbitrary ring $R$, let $\cc$ denote the category of left $R$-modules and $\underline{\cc}$ denote the stable module category obtained by factoring out the projectives from $\cc$. We assume that $R$ is perfect, so that every object in $\cc$ has a projective cover. Then, there exists a syzygy functor $\Omega: \cc\to \cc$ which takes an object $X$ to the kernel of its projective cover. The restriction of $\Omega$ to $\underline\cc$ is an equivalence of categories.

When $R$ is the path algebra of $\qrud$, we have the following set of equalities (by Lemma~\ref{lem:projcover}) for $n\geq 1$ and $\phi\in\Sigma_d$:
\begin{equation*}
\begin{aligned}
\Omega(P_n)&= \sigma(P_{n-1}), \\
\Omega(I_{n-1}) &= \sigma(I_{n}),\\
\Omega(R_{\phi}(n)) &= \sigma(R_{\phi}(n)),
\end{aligned}
\qquad 
\begin{aligned}
    \Omega(P_n')&= \sigma(P_{n-1}'),\\
    \Omega(I_{n-1}')&= \sigma(I_{n}'),\\
    \Omega(R_{\phi}(n)') &= \sigma(R_{\phi}(n)').
\end{aligned}
\end{equation*}
The discrepancy in indices in the first two lines can be viewed as another manifestation of the presence of the ``shifted'' algebra in our main theorem (also see Remark~\ref{rem:shifted}). This discrepancy is related to the classical reflection functors of \cite{BGP}. For example, if $X\in\Rep(\qrud)$ is $(e,f)$-saturated (see Definition~\ref{def:sat}), then we have the equality $\Omega (X) = \sigma\circ \mathbb{S}^+(X)$, where $\mathbb{S}^+$ is a reflection functor associated with the Kronecker subquiver of $\qrud$ with maps $e$ and $f$ (see \cite[Definition 3.1]{BS12affine}).
\end{rem}

\begin{lem}
The map $\sigma$ induces an algebra automorphism and the map $\tau$ induces an algebra anti-automorphism on the Hall algebra.
\end{lem}
\begin{proof}
    Since both these functors are equivalences they induce invertible maps on the Hall algebra, and since $\tau$ is contravariant, it induces an anti-automorphism.
\end{proof}



\begin{dfn}\label{dfn:slope}
We will say that a module with dimension vector $(d_1,d_2)$ has \emph{positive slope} if $d_2-d_1 > 0$, and we similarly define \emph{zero slope} and \emph{negative slope}. 
\end{dfn}
We note that the Hall algebra is graded by the abelian group of dimension vectors, so the subspace of $Hall(\qrud)$ spanned by positive slope modules is a subalgebra.

Let $A$ be the path algebra of $\qrud$, considered as a $\qrud$ module. 
\begin{lem}\label{lem:qrudpath}
    The path algebra $A$ splits as a sum $A \cong M \oplus M'$, and the summands are also injective.
\end{lem}
\begin{proof}
    Over the field $\FF_q$, the path algebra of $\qrud$ is an $8$-dimensional vector space and has a basis given the following paths:
\[\{s_1,s_2,e,e',f,f',ef',f'e\},\]
where $s_1$ and $s_2$ are paths of length zero with source and end at vertices 1 and 2 respectively, and the identity element is $1 = s_1+s_2$. As a left $\qrud$-module, the path algebra $A$ splits into a sum of two modules as follows:
\begin{align*}
V&:= \text{span}_{\FF_q}\{s_1, e, f, f'e\} \cong As_1,\\
W&:= \text{span}_{\FF_q}\{s_2, e', f', ef'\} \cong As_2.
\end{align*}
It is clear that $V\cong M$ and $W\cong M'$, which implies both $M$ and $M'$ are projective modules. (In fact, this shows that they are the only projective indecomposables.) 
As the contravariant functor $\tau$ preserves both $M$ and $M'$, we conclude that they are both injective modules as well.
\end{proof}

\begin{cor}\label{cor:mhom}
    If $Y$ is any $\qrud$ representation, we have the following isomorphisms of vector spaces:
    \begin{align*}
    \Hom_{\qrud}(M,Y) \cong Y_1 \quad &\textrm{and} \quad  \Hom_{\qrud}(M',Y) \cong Y_2\\
    \Hom_{\qrud}(Y,M) \cong Y^\vee_1 \quad &\textrm{and} \quad  \Hom_{\qrud}(Y,M') \cong Y^\vee_2.
    \end{align*}
\end{cor}
\begin{proof}
    Recall that for an algebra $A$ and a left $A$-module $X$, the space $\Hom_{A\mod}(A,X)$ inherits a left $A$-module structure from the right action of $A$ on itself. If $A$ is unital, then $\Hom_{A\mod}(A,X)$ is canonically isomorphic to $X$ via the map $\left[ f:A \to X\right] \mapsto f(1_A)$. Then we have the following equalities of vector spaces:
    \[
    \Hom_{\qrud}(M,Y) = \Hom_{\qrud}(As_1,Y) = s_1\cdot \Hom_{\qrud}(A,Y) = s_1\cdot Y = Y_1
    \]
    where the first equality uses Lemma \ref{lem:qrudpath} and the second and third equalities follow from the above isomorphism $\Hom_{A\mod}(A,Y) = Y$. The second claim follows from the computation
    \[
    \Hom(Y,M) = \Hom(\tau(M),\tau(Y)) = \Hom(M,\tau(Y)) = Y^\vee_1
    \]
    where $Y_1^\vee$ is the linear dual.
\end{proof}
\begin{proof}[Proof of Theorem~\ref{prop:class}]
Let $X=(X_1,X_2)$ be any indecomposable representation of $\qrud$. We first assume that the composition $f'e$ is not zero in $X$, so there exists $x\in X_1$ such that $0 \neq  f'e(x) = -e'f(x)$. Then it is easy to see that the space spanned by the vectors
\[\{x,e(x),f(x),f'e(x)\}\]
is a submodule of $X$ that is isomorphic to $M$. Since $M$ is injective and $X$ is indecomposable, this implies that $X=M$.

By a similar argument, we can show that if the composition $ef'$ is not zero in $X$, then $X$ is isomorphic to $M'$. So, we assume henceforth that $ef'=0$ and $f'e=0$. This implies that the composition of any two of $e,e',f$ and $f'$ on $X$ is zero. Consider the following vector space decompositions:
\begin{align*}
    X_1&=(\ker(e)\cap \ker(f))\oplus X_1',\\
    X_2&=(\im(e)+ \im(f))\oplus X_2'
\end{align*}
for some $X_1'\subseteq X_1, X_2'\subseteq X_2$. Then, it is clear that:
\begin{align*}
V:&=(\im(e)+ \im(f))\oplus X_1',\\
W:&=(\ker(e)\cap \ker(f))\oplus X_2'
\end{align*}
are submodules of $X$ and that $X=V\oplus W$. Since $X$ was assumed to be indecomposable, we must have that either $V={0}$ or $W=0$. In the first case we have that $e=f=0$ and the second we have that $e'=f'=0$, completing the proof of the theorem.
\end{proof}

\subsection{Rudakov's theorem and $\sl_2(\FF_q)$}\label{sec:rudsl2}

In this section we discuss the precise relation between Rudakov's quiver and the representation theory of $\sl_2(\FF_q)$. Let $F$ be a field of characteristic $p>0$.

\begin{defn}
A \emph{restricted Lie algebra} $(\Gg,[p])$ over $F$ is a Lie algebra $\Gg$ over $F$ equipped with a linear map $(\cdot)^{[p]}:\Gg\to\Gg$, called the $p$-operator, satisfying the following conditions:
\begin{enumerate}
\item $\mathrm{ad}\text{ } x^{[p]} = (\mathrm{ad}\text{ } x)^p$ for all $x\in\Gg$
\item $(ax)^{[p]}=a^px^{[p]}$ for all $a\in F$ and $x\in\Gg$
\item $(x+y)^{[p]} = x^{[p]} + y^{[p]} + \sum_{i=1}^{p-1} s_i(x,y)$ for all $x,y\in \Gg$
\end{enumerate}
where $s_i(x,y)$ is $1/i$ times the coefficient of $t^i$ in the formal expansion of $(\mathrm{ad}\text{ }(tx+y))^{p-1}(x)$.
\end{defn}
We suppress $[p]$ from the definition and simply refer to $\Gg$ as a restricted Lie algebra.
\begin{ex}
Let $A$ be an associative algebra over $F$. Consider the Lie algebra $\Gg=\mathrm{Der}(A)$ of $F$-linear derivations on $A$ and fix $D\in\Gg$. By an elementary computation involving the Leibniz rule, it can be shown that the linear map $D^p=D\circ D\circ \cdots \circ D$ ($p$ times) is itself a derivation, and thus, $D^p\in \Gg$. Then, the linear map $D\mapsto D^p$ provides a restricted Lie algebra structure on $\Gg$.

In the special case when $A$ is the algebra of regular functions on an algebraic group $G$ defined over $F$, the Lie algebra $\Gg$ is the space of left invariant vector fields on $G$ and is equivalent to the Zariski tangent space to $G$ at the identity. The above construction equips $\Gg$ with a restricted structure (which is the unique restricted structure on $\Gg$ when it is a simple Lie algebra).
\end{ex}

With the above restricted structure, the element $x^p-x^{[p]}$ is central in the universal enveloping algebra $\uu(\Gg)$ for all $x\in\Gg$. The commutative subalgebra of $\uu(\Gg)$ generated by such elements is known as the $p$-center of $\uu(\Gg)$. A representation $V$ of $\Gg$ is called a \emph{restricted representation} if the $p$-center acts trivially on $V$. Thus, a restricted representation is the same as a module over the \emph{restricted enveloping algebra} $\uu_0(\Gg)$ defined as:
\[\uu_0(\Gg):=\uu(\Gg)/(x^p-x^{[p]}: x\in \Gg).\]
We will be interested in the case when $\Gg=\sl_2(\FF_q) =\FF_q\langle e, f, h\rangle$, where $q$ is some power of $p$. In that case, the construction in the above example gives the following explicit restricted structure on $\Gg$:
\[e^{[p]} = 0, \quad f^{[p]}=0, \quad h^{[p]}=h.\]

Henceforth, we will be assuming that $p>2$ throughout this section. 
Let $\mathcal C$ be the category of (equi-characteristic, finite dimensional) \emph{restricted} representations of $\sl_2(\FF_q)$. By the discussion above, these are the representations which are annihilated by $e^p, f^p$, and $h^p-h$. Just as in the characteristic 0 case, let $V_n := \FF_q[x,y]_n$ be the \emph{Weyl module}. Explicitly, this is the space of degree $n$ polynomials in 2 variables, which is a module over $\sl_2(\FF_q)$ via the standard action
\[
e \mapsto y\partial_x,\quad f \mapsto x\partial_y,\quad h \mapsto y\partial_y - x \partial_x.
\]


\begin{thm}[\cite{Jac41,Pol68}] \label{thm:Pol}
    The Weyl modules $V_n$ are indecomposable for all $n\geq 0$ and are simple if and only if $n < p$. This is a complete list of the simples in the category $\mathcal C$. For $0 \leq i,j \leq p-2$, we have
    \[
    \dim(\Ext^1(V_i,V_j)) = \begin{cases} 2& \textrm{if } i+j=p-2\\ 0 & \textrm{otherwise} \end{cases}.
    \]
    On the other hand, $V_{p-1}$ is simple, projective, and injective and is known as the Steinberg representation. Furthermore, as a left module over itself, the algebra $\uu_0(\Gg)$ splits as a sum of indecomposable modules:
    \[\uu_0(\Gg) = V_{p-1}^{\oplus p} \oplus\bigoplus_{i=0}^{p-2} Q_i^{\oplus i+1},\]
    where $Q_i$ is the principal indecomposable module surjecting onto $V_i$.
\end{thm}
These results imply that the category $\mathcal C$ splits up into the following blocks:
\[
\mathcal C \simeq Vect \oplus \bigoplus_{i=0}^{(p-3)/2} \mathcal C_i
\]
where the block $\mathcal C_i$ contains exactly 2 simple Weyl modules ($V_i$ and $V_{p-i-2}$), and $Vect$ is generated by the Steinberg module $V_{p-1}$. As a result, the Hall algebras of these categories are related as:
\[Hall(\cc) \cong Hall(Vect) \otimes \bigotimes_{i=0}^{(p-3)/2} Hall(\cc_i) .\]
Since $Vect$ is semisimple and generated by the unique simple object, the algebra $Hall(Vect)$ is isomorphic to the polynomial ring in one variable. Thus, in order to study $Hall(\cc)$, it suffices to study $Hall(\cc_i)$ for each $i$. Fix $0\leq i\leq (p-3)/2$ and let $j=p-i-2$.

\begin{thm}[{\cite{Rud82}}] \label{thm:Rudback}
Each block $\mathcal C_i$ is equivalent to the category $\Rep_{\FF_q}(\qrud)$ of representations of the Rudakov quiver. Under this equivalence, the simples $I_0$ and $P_0$ are sent to $V_{i}$ and $V_{j}$ respectively.
\end{thm}

\begin{ex}
Let $V$ be an extension of $V_i^m$ by $V_j^n$ for some $m,n\geq 0$. By Theorem~\ref{thm:Pol}, we know that
\[\Ext^1(V_i^n,V_j^m) \cong \Ext^1(V_i,V_j)^{mn} \cong (\FF_q^2)^{mn} \cong (\FF_q^{mn})^2.\]
Therefore, the data of any such extension can be thought of as being given by two $m\times n$ matrices over $\FF_q$. This gives a representation of $Q_{Rud}$ having dimension $(m,n)$: The maps $e$ and $f$ can be defined as being given by these matrices, whereas the maps $e'$ and $f'$ are taken to be zero. Similarly, representations of $Q_{Rud}$ with $e=f=0$ correspond to extensions of $V_j^m$ by $V_i^n$.
\end{ex}

Here, we create a table which translates between representations 
in the categories $\cc_i$ and $\Rep_{\FF_q}(\qrud)$.


\renewcommand{\arraystretch}{1.2}
\setlength{\tabcolsep}{1.2em}
\begin{center}
\begin{tabular}{|c|c|}
\hline
$\Rep_{\FF_q}(\qrud)$ & $\cc_i$ \\ \hline \hline
$P_n$ &$V_{np+j}$\\ \hline
$P_n'$ &$V_{np+i}$\\\hline
$I_n$ &$(V_{np+i})^{\vee}$\\\hline
$I_n'$ &$(V_{np+j})^{\vee}$\\\hline
$R_{\phi}(n)$ & Indecomposable maximal proper submodules of $(V_{np+i})^{\vee}$\\&(Equivalently, indecomposable minimal proper quotients of $V_{np+j}$)\\\hline
$R_{\phi}(n)'$ & Indecomposable maximal proper submodules of $(V_{np+j})^{\vee}$\\&(Equivalently, indecomposable minimal proper quotients of $V_{np+i}$)\\\hline
$M$ & $Q_i$\\\hline
$M'$ & $Q_j$\\\hline
\end{tabular}
\end{center}


\begin{remark}
    For $n\geq 0$, our map from the (deformed) quantum affine algebra to $Hall(\mathcal C_i)$ sends the generators $E_n$ and $F_n$ to the Weyl modules (up to central coefficients involving $[Q_i]$ and $[Q_j]$). Furthermore, the spherical subalgebra of $Hall(\cc_i)$ (see Definition~\ref{def:maximal}) is generated by $[V_i]$, $[V_j]$, $[V_i\oplus V_j]$ and $[Q_i]$ by the remark following Corollary~\ref{cor:4gens}.
\end{remark}

\begin{remark} \label{rem:GP}
In general, we can consider representations of a restricted Lie algebra $\Gg$ on which the $p$-center doesn't necessarily act trivially. That is, we can consider the category of modules over the reduced enveloping algebra
\[\uu_{\chi}(\Gg):=\uu(\Gg)/(x^p-x^{[p]}-\chi(x)^p: x\in\Gg),\]
where $\chi\in\Gg^*$ is an arbitrary linear function known as the $p$-character. These categories have been quite extensively studied in geometric representation theory (see, for example, \cite{Hum71, Fried87, Prem97, Gor02, BMR, BMR2}).

Rudakov's theorem (Theorem~\ref{thm:Rudback}) was generalized in Section 9.14 of \cite{GP02}. More generally, for good primes $p$, they classify $p$-characters $\chi$ and the blocks of $\uu_\chi(\mathfrak{g})$ that have finite and tame representation type, and find a quiver theoretic description in the tame cases. In particular, tame blocks exist precisely when $\chi$ is a regular or a subregular element of $\Gg^*$.
\end{remark}

\section{Homological Statements}\label{sec:homological}


In this section we collect and prove various homological statements involving representations of the Kronecker and Rudakov quivers. Roughly, these results will be used to compute whether certain Hall algebra structure constants $c_{M,N}^E$ are zero or nonzero. The precise computation of the actual values of (certain) structure constants will be performed in the next section.

\subsection{The Kronecker quiver}
The main goal of this section is to understand the precise relationship between representations of the Kronecker quiver $Q_K$ and of the Rudakov quiver $\qrud$. In particular, we will show that a ``large'' subcategory of $Q_K$-representations is equivalent to ``half'' of the category of $\qrud$-representations. This induces a map from a
subalgebra of the Hall algebra of the Kronecker quiver to the Hall algebra of the Rudakov quiver, and this allows us to use Hall algebra identities from the literature that were discussed in the previous section.




\begin{definition}
The \emph{extension by 0} functor 
\[
Func:\Rep(Q_K) \to \Rep(\qrud)
\]
outputs the same vector spaces and maps, with the maps $e',f'$ defined to be $0$.    
\end{definition}
\begin{lem}\label{lem:exact}
    The extension by 0 functor is exact and fully faithful.
\end{lem}
\begin{proof}
    This is obvious from the definition, or from the fact that the extension by 0 is isomorphic to the restriction of scalars along the surjection $\qrud \to Q_K$ of path algebras taking $e'$ and $f'$ to 0.
\end{proof}
Even though it is exact and fully faithful, $Func$ does not induce a map on Hall algebras because its image  is not closed under extensions. For example, in the following short exact sequence, $I_0$ and $P_1$ are in the image of $Func$, but $M$ isn't:
\[
0 \to I_0 \to M \to P_1 \to 0
\]
(see Example \ref{ex:mexact}). However,  this short exact sequence is essentially the only obstruction to the existence of a map between Hall algebras. To make this statement precise, we provide the following definitions.

\begin{definition} \label{def:sat}
    Let $Y$ be a representation of the Kronecker quiver $Q_K$ or the Rudakov quiver $\qrud$. 
    \begin{enumerate} 
    \item $Y$ is \emph{$(e,f)$-injective} if $\ker(e_Y) \cap \ker(f_Y)= 0$.
    \item $Y$ is \emph{$(e,f)$-surjective} if $\mathrm{coker} (e_Y \oplus f_Y) = 0$. 
    \item $Y$ is \emph{$(e,f)$-saturated} if it is both $(e,f)$-injective and $(e,f)$-surjective.
    \item For $Q \in \{Q_K,\qrud\}$, we denote by $Sat(Q)\subset \Rep(Q)$ be the full subcategory of $(e,f)$-saturated objects. 
    \end{enumerate}
    For $\qrud$, we modify the above definitions by replacing ``$(e,f)$'' with ``$(e',f')$'' if $e'$ and $f'$ satisfy the analogous conditions. For example, $Y$ is \emph{$(e',f')$-surjective} if $\mathrm{coker}(e'\oplus f')= 0$, etc.
\end{definition}
It is easy to see which $\qrud$ indecomposables satisfy these properties (assuming $n\geq 1$):
\renewcommand{\arraystretch}{1.2}
\setlength{\tabcolsep}{1.2em}
\begin{center}
\begin{tabular}{|c|c|c|c|c|}
\hline
&$(e,f)$-inj.&$(e,f)$-surj.&$(e',f')$-inj.&$(e',f')$-surj.\\
\hline
$I_0$, $M$&&\checkmark&\checkmark&\\
\hline
$P_0$, $M'$&\checkmark&&&\checkmark\\
\hline
$P_n,I_n,R_\phi(n)$&\checkmark&\checkmark&&\\
\hline
$P'_n,I'_n,R_\phi(n)'$&&&\checkmark&\checkmark\\
\hline
\end{tabular}.
\end{center}
Note from the above table that an indecomposable is $(e,f)$-saturated or $(e',f')$-saturated precisely when its Loewy length is two.
\begin{rem} \label{rem:sat}
    These definitions do not seem to be widely used in the quiver literature, so we provide some elementary remarks.
    \begin{enumerate} 
    \item For the Kronecker quiver, the saturated objects are precisely those without a simple summand.
    \item 
    The functor $\sigma$ swaps the $(e,f)$-versions of the above definition with the respective $(e',f')$-counterparts. The functor $\sigma\circ\tau$ swaps $(e,f)$-injective (respectively $(e',f')$-injective) representations with $(e,f)$-surjective (respectively $(e',f')$-surjective) representations.
    \item Finally, the full subcategory of $(e,f)$-saturated representations is not an abelian subcategory: for any regular $(1,1)$ $(e,f)$-saturated representation $X$, there exists a short exact sequence 
    \[
    0 \to X \to P_1 \to P_0 \to 0.
    \]
    Then $X$ and $P_1$ are $(e,f)$-saturated but the quotient $P_0$ is not.
    \end{enumerate}
\end{rem}

The next lemmas show that these definitions interact well with short exact sequences.

\begin{lem}\label{lem:saturated}
    Suppose $0 \to X \stackrel{\iota}{\to} Y \stackrel{\pi}{\to} Z \to 0$ is a short exact sequence of $\qrud$-representations. If $X$ and $Z$ are $(e,f)$-injective, then $Y$ is too. Similarly, if $X$ and $Z$ are $(e,f)$-surjective, then $Y$ is too.
\end{lem}
\begin{proof}
    To prove the first claim, suppose $y \in \ker(e_Y) \cap \ker(f_Y)$. Since $Z$ is $(e,f)$-injective, this implies $\pi(y) = 0$, so there exists $x \in X$ with $\iota(x) = y$. Then $\iota(e_X(x)) = 0 = \iota(f_X(x))$, and since $\iota$ is injective and $X$ is $(e,f)$-injective, this implies $x=0$.

    For the second claim, suppose $y \in Y_2$. Since $Z$ is $(e,f)$-surjective, there exist $z_e, z_f \in Z_1$ with $\pi(y) = e_Z(z_e) +f_Z(z_f)$. Pick $\pi$-lifts $y_e,y_f \in Y_1$ of $z_e,z_f$ respectively. Then $\pi(y-e_Y(y_e) - f_Y(y_f)) = 0$, so there exists $x \in X_2$ with $\iota (x) = y-e_Y(y_e) - f_Y(y_f)$. Since $X$ is $(e,f)$-surjective, there exist $x_e, x_f \in X_1$ with $x = e_X(x_e) + f_X(x_f)$. Then 
    \begin{align*}
        e_Y(y_e + \iota (x_e)) + f_Y(y_f + \iota (x_f)) &= e_Y(y_e) + \iota(e_X(x_e)) + f_Y(y_f) + \iota(f_X(x_f))\\
        &= e_Y(y_e) + f_Y(y_f) + \iota(x)\\
        &= e_Y(y_e) + f_Y(y_f) + (y - e_Y(y_e) - f_Y(y_f))\\
        &= y.
    \end{align*}
\end{proof}

\begin{lem}\label{lem:satesssur}
    If $Y$ is $(e,f)$-surjective, then 
    \[ 
    \im(e') + \im(f') \subset \ker(e) + \ker(f).
    \]
    If $Y$ is $(e,f)$-saturated, then $e'_Y = 0 = f'_Y$. 
\end{lem}
\begin{proof}
    Since $Y$ is $(e,f)$-surjective, if $y \in Y_2$, then we can find $a,b \in Y_1$ with $y = e(a) + f(b)$. Then 
    \begin{align*}
        e'(y) &= e'(e(a)) + e'(f(b))\\
        &= 0 - f'(e(b))\\
        &\in \ker(f).
    \end{align*}
    Similarly, $\im(f') \subset \ker(e) + \ker(f)$. If $Y$ is also $(e,f)$-injective, this implies $e'(y) = 0 = f'(y)$.
\end{proof}
\begin{remark} The converse of Lemma \ref{lem:saturated} is false: take $X$ and $Z$ to be the two simples and $Y$ to be a regular representation, which is $(e,f)$-saturated. Also, Lemma \ref{lem:satesssur} uses the relation $e'f = -f'e$ in an essential way in the second line of the displayed equality.
\end{remark}


\begin{prop}\label{prop:satequiv}
    The extension by 0 functor $Func$ induces both an equivalence $Sat(Q_K) \to Sat(\qrud)$ and an algebra map $Hall(Sat(Q_K)) \to Hall(\qrud)$.
\end{prop}
\begin{proof} 
Lemma \ref{lem:exact} shows $Func$ is fully faithful and exact, which implies its restriction to $Sat(Q_K)$ is also fully faithful and exact. Then the first claim follows from Lemma \ref{lem:satesssur}, which shows the restriction  $Func: Sat(Q_K) \to Sat(\qrud)$  is essentially surjective. The second claim follows from this and Lemma \ref{lem:saturated}, which shows that $Sat(\qrud)$ is closed under extensions.
\end{proof}

This proposition shows that except when $A$ is pre-projective and $B$ is pre-injective, any extensions of $Func(A)$ by $Func(B)$ in $\Rep(\qrud)$ already come from extensions of $A$ and $B$ in $\Rep(Q_K)$.

\begin{rem}
Technically, our definition of the Hall algebra starts with an abelian category, and $Sat(\qrud)$ is not abelian. However, since it is an extension-closed full subcategory category of an abelian category, $Sat(\qrud)$ is an \emph{exact} category. In \cite[Sec. 2]{Hub06}, Hubery provides a definition of the Hall algebra of an exact category -- his results immediately imply that there is an injective algebra map $Hall(Sat(\qrud)) \hookrightarrow Hall(\qrud)$.
\end{rem}


\subsection{Computations with Hom and Ext}
In this section we provide some homological computations that will be necessary to prove the Hall algebra identities we need.
\begin{lem}\label{lem:sathom}
    Suppose $X$ is $(e,f)$-surjective and $Y$ is $(e',f')$-saturated. Then 
    \[
    \Hom_{\qrud}(X,Y) = \Hom_{Vect}(X_1,Y_1).
    \]
    If $X$ is $(e',f')$-surjective and $Y$ is $(e,f)$-saturated, then
    \[
    \Hom_{\qrud}(X,Y) = \Hom_{Vect}(X_2,Y_2).
    \]    
\end{lem}
\begin{proof}
    Since $Y$ is $(e',f')$-saturated, $e_Y=0=f_Y$, so any map $\alpha:X \to Y$ must satisfy $\alpha \circ e_X = 0 = \alpha \circ f_X$. Since $X$ is $(e,f)$-surjective, any $x_2 \in X_2$ is in the image of $e_X\oplus f_X$. Combining these facts implies $\alpha_2:X_2\to Y_2$ must be 0. Conversely, since $e_Y = 0 = f_Y$, any linear map $\alpha_1:X_1\to Y_1$ can be extended to a map $(\alpha_1,0=\alpha_2):X \to Y$ of $\qrud$ modules. The second claim follows from the first by applying the covariant functor $\sigma$ (which switches $Y_1 \leftrightarrow Y_2$ and switches the pairs $(e,f) \leftrightarrow (e',f')$).
\end{proof}



For the next two propositions, fix a point $[m:n]\in\mathbb{P}^1(\FF_q)$ and let $X$ be the $(e,f)$-saturated regular representation of dimension $(1,1)$ such that the maps $e$ and $f$ are multiplication by $m$ and $n$ respectively, and the maps $e'$ and $f'$ are zero. Without loss of generality, we will assume that $m\neq 0$. Thus, we have that $X=R_{\phi}(1)$, where $\phi(t)=t-n/m$. 

\begin{lem}\label{lem:extxp0}
    We have $\Ext^1(X,P_{0}') = 0$ and $\Ext^1(P_{0}',X)=\kk$, and the nontrivial extension is $I_1$.
\end{lem}
\begin{proof}


Since $P_0'=I_0$ is $(e,f)$-surjective, any extension of $X$ by $I_0$ (or vice-versa) must be $(e,f)$-surjective and have dimension $(2,1)$. Since $M$ has dimension $(2,2)$, it can't be a summand of any such extension, so all extensions arise from those in the Kronecker quiver. Then this lemma is a consequence of \cite[Lemma 1.1]{Sz06}.
\end{proof}

\begin{lem}\label{lem:extxp}
For $k \geq 0$ we have $\Ext^1(X,P_{k+1}')=0$ and $\Ext^1(P_{k+1}',X)=\kk$, and the nontrivial extension is $M'\oplus P'_k$. 
\end{lem}
\begin{proof}
We start with a linear map $\FF_q^{k+1}\to \FF_q^{k+2}$, which is given by the inclusion into the first $k+1$ coordinates. This map at vertex 1 of the quiver induces an inclusion of representations $P_k'\xrightarrow{\iota} P_{k+1}' . $ Then we have a short exact sequence
\[0\to P_k' \xrightarrow{\iota}P_{k+1}' \to X' \to 0,\]
where the quotient $X'$ is isomorphic to the $(e',f')$-saturated regular representation such that the maps $e,f$ and $e'$ are zero, whereas the map $f'$ is an isomorphism. Thus, we have that $X'=R_{\pi}(1)'$, where $\pi(t)=0$.

Applying $\Ext^\bullet(X,-)$ to this sequence gives the following piece of the long exact sequence:
\[
\Ext^1(X,P_k') \to \Ext^1(X,P_{k+1}') \to \Ext^1(X,X').
\]
The first term in this sequence is 0 by induction, and the last term is 0 by Lemma \ref{lem:regularext}.

For the second claim, we apply $\Ext^\bullet(-,X)$ to the same short exact sequence to obtain the exact sequence
\[
\Ext^1(X',X) \to \Ext^1(P'_{k+1}) \to \Ext^1(P'_k,X).
\]
The first term in this sequence is 0 by Lemma \ref{lem:regularext}, and the last term is $\kk$ by induction. Finally, to show the middle term is nonzero, we construct a nonsplit extension of $P'_k$ by $X$ as follows:
\[
0 \to X \stackrel{\iota}{\to} M' \oplus P_k' \to P_{k+1}' \to 0.
\]
The inclusion is defined by $\iota(x) = mx_2+nx_1+v$, where $x\in X_1$ is a generator and $v$ is a basis vector for the first coordinate at vertex $1$ of $P_k'$. (We're using notation from \eqref{eq:mpdef} for $x_1,x_2 \in M'$.) Then, it is easy to check that the cokernel of the above sequence is indeed isomorphic to $P_{k+1}'$, proving the claim.
\end{proof}

\begin{lem} \label{lem:m0annoy}
Let $R$ be an arbitrary $(e,f)$-saturated regular representation of $\qrud$ and $S=P_0$ or $I_0$. Then, any extensions of $R$ by $S$, and vice-versa, arise from those in the Kronecker quiver.
\end{lem}

\begin{proof}
We will only prove the claim when $S=I_0$, since the claim for $P_0$ follows by applying the functor $\sigma\circ\tau$.

Since $R$ and $I_0$ are both $(e,f)$-surjective, so must be any extension between them. Therefore, the only possible indecomposable summand of such an extension that doesn't arise from the Kronecker quiver is $M$. Suppose we have a short exact sequence
\[0\to I_0 \xrightarrow{\Theta} M\oplus V\to R\to0,\]
for some representation $V$. Since $e'=f'=0$ in $R$, we must have that $\langle x_2\rangle\oplus 0 \subseteq \Theta(I_0)$ (using notation from \eqref{eq:mdef} for $x_2 \in M$), so that the maps $e'_M$ and $f'_M$ become zero in the quotient. But then, the quotient $(M\oplus V)/\Theta(I_0) \cong P_1\oplus V$ can not be isomorphic to $R$.

Next, suppose we have a short exact sequence
\[0\to R \xrightarrow{\Theta} M\oplus V\to I_0\to0,\]
for some representation $V$. Since $e'=f'=0$ in $R$, we must have that $\Theta(R)\subseteq \langle x_2\rangle\oplus V$. But then, the maps $e_M$ and $f_M$ won't be zero in the quotient, so the quotient can't be isomorphic to $I_0$.
\end{proof}

\subsection{Regular representations}

Recall that the set $\Sigma$ can be partitioned into subsets $\Sigma_d$ consisting of polynomials of degree $d$. The subset $\Sigma_1$ is in natural bijection with the set of points $\mathbb{P}^1(\FF_q)$ (see Remark~\ref{rem:irred}) as follows:
\begin{align*}\phi(x)=x-c \in \Sigma_1 &\longleftrightarrow [1:c] \in \mathbb{P}^1(\FF_q),\\
0\in\Sigma_1 &\longleftrightarrow [0:1] \in \mathbb{P}^1(\FF_q).
\end{align*}


The following lemma describes the most fundamental homological properties of regular representations:

\begin{lem}\label{lem:regularext}
Let $\phi,\pi \in \Sigma$.
\begin{enumerate} [label=(\alph*)]
\item If $\phi \in \Sigma_1$, then 
\[\Ext^1(R_{\phi}(1),R_{\phi}(1)') \cong \FF_q,
\] 
and the nontrivial extension is isomorphic to $M$. Similarly, 
\[
\Ext^1(R_{\phi}(1)',R_{\phi}(1))\cong\FF_q,
\]
and the nontrivial extension is $M'$. 
\item If $\phi\neq {\pi}$, then $\Ext^1(R_{\phi}(m),R_{\pi}(n)')=\Ext^1(R_{\pi}(n)',R_{\phi}(m))=0$. 
\end{enumerate}
\end{lem}

\begin{proof}
Suppose we have a non-split extension of $R_{\phi}(1)$ by $R_{\pi}(1)'$. Such an extension will be a representation of dimension $(2,2)$ such that both pairs $(e,f)$ and $(e',f')$ of maps are nonzero. It is clear that the direct sum can not occur as a non-split extension. Therefore, the only possible candidate of such a representation is $M$. (It is not possible to have $M'$ as it does not have a subrepresentation isomorphic to $R_{\pi}(1)'$ for any $\pi\in \Sigma_1$.) So, we assume that we have the following short exact sequence:
\[0\to R_{\pi}(1)'\to M\to R_{\phi}(1)\to 0.\]

Suppose the polynomials $\phi$ and $\pi$ respectively correspond to the points $[m:n]$ and $[m':n']$ in $\mathbb{P}^1(\FF_q)$ under the above bijection. Explicitly, this means that
\[
R_{\phi}(1)\cong \begin{tikzcd}[arrow style=tikz,>=stealth,row sep=4em]
\FF_q \arrow[rr,shift left = 1ex, "e = m(I_{1\times 1})"]
\arrow[rr, shift right = 1ex, swap, "f = n(I_{1\times 1})"] && \FF_q
\end{tikzcd},
\]
\[
R_{\pi}(1)'\cong \begin{tikzcd}[arrow style=tikz,>=stealth,row sep=4em]
\FF_q &&\arrow[ll,shift left = 1ex, "f' = n'(I_{1\times 1})"]
\arrow[ll, shift right = 1ex, swap, "e' = m'(I_{1\times 1})"] \FF_q
\end{tikzcd}.
\]
Fix the basis $\{x_1,x_2,y_1,y_2\}$ for $M$ as described in equation \eqref{eq:mdef}. Then, the only subrepresentation $W$ of $M$ that is isomorphic to $R_{\pi}(1)'$ is the one spanned by $\{x_2, -n'y_1+m'y_2\}$. For the quotient $M/W$ to be isomorphic to $R_{\phi}(1)$, we must have that $ne(x_1)=mf(x_1)$ or $ny_1=my_2$ in $M/W$. This is possible if and only if $ny_1-my_2$ lies in the span of $\{x_2, -n'y_1+m'y_2\}$. This implies that $nm'-mn'=0$, which is equivalent to the condition that $\phi={\pi}$. This implies the first claims in (a) and the $m=n=1$ case of (b) when $\pi,\phi \in \Sigma_1$. The second claims in (a) and the $m=n=1$ case of (b) follow by a similar argument.

We next prove part (b) when $\pi,\phi \in \Sigma_1$ and when at least one of $m,n$ is greater than $1$. In this case, we have short exact sequences
\begin{align*}
    0 \to R_{\phi}(n-1) &\to R_{\phi}(n) \to R_{\phi}(1)\to 0\\
 0 \to R_{\pi}(m-1)' &\to R_{\pi}(m)' \to R_{\pi}(1)'\to 0.
\end{align*}
We can construct long exact sequences of Ext functors from these short exact sequences, which can be used to prove by induction that $\Ext^1(R_{\phi}(n),R_{\pi}(m)')=\Ext^1(R_{\pi}(m)',R_{\phi}(n))=0$ for all $m,n$. 

Finally, we prove part (b) when $\phi,\pi \in \Sigma$ are arbitrary. Pick a large enough finite extension $\FF_{q'}$ of $\FF_q$ in which both $\phi$ and $\pi$ split. Then, in the  category of modules for $\qrud$ over $\FF_{q'}$, the objects $R_{\phi}(n)$ and $R_{\pi}(m)'$ both split into sums of representations of dimension $(1,1)$. Furthermore, by assumption, none of the roots of $\phi$ are equal to any of the roots of $\pi$. So, using part (a), we conclude that there aren't any non-trivial extensions between $[R_{\phi}(n)]$ and $[R_{\pi}(m)']$ over $\FF_{q'}$. But this implies that there can't be any non-trivial extensions between $[R_{\phi}(n)]$ and $[R_{\pi}(m)']$ over $\FF_{q}$ either. This is because if such an extension  $V$ existed, we would have that $V$ splits as a sum of regular representations of dimension $(1,1)$ over $\FF_{q'}$. However, this implies that $V$ must be a regular representation over $\FF_q$ (as no other representations split upon extension). Furthermore, any regular direct summand of $V$ should correspond to an irreducible polynomial whose roots coincide with those of either $\phi$ or $\pi$, which proves that it is a trivial extension. 
\end{proof}

We then have the following immediate corollary of the previous lemma:


\begin{cor}\label{lem:regcomm}
Let $\phi,\pi \in \Sigma_1$.
\begin{enumerate} [label=(\alph*)]
\item If $\phi={\pi}$, then 
\begin{align*}
[R_{\phi}(1)][R_{\pi}(1)'] &= q[R_{\phi}(1)\oplus R_{\pi}(1)'] + [M],\\
[R_{\pi}(1)'][R_{\phi}(1)] &= q[R_{\phi}(1)\oplus R_{\pi}(1)'] + [M'].
\end{align*}
\item If $\phi\neq {\pi}$, then 
\begin{equation*}
[R_{\phi}(m)][R_{\pi}(n)'] = q^{mn}[R_{\phi}(m)\oplus R_{\pi}(n)']=[R_{\pi}(n)'][R_{\phi}(m)].
\end{equation*}
\end{enumerate}
\end{cor}

\begin{proof}
For all $\phi,\pi$, Lemma \ref{lem:sathom} shows that $\dim(\Hom(R_{\phi}(1),R_{\pi}(1)'))=\dim(\Hom(R_{\pi}(1)',R_{\phi}(1)))=1$. If $\phi = {\pi}$, Lemma \ref{lem:regularext} shows that the only non-trivial extension of $R_{\phi}(1)$ by $R_{\pi}(1)'$ is $M$, which proves the first equality. The second equality follows by a similar argument.

Conversely, if $\phi \not= {\pi}$, Lemma \ref{lem:regularext} shows $\Ext^1(R_{\phi}(n),R_{\pi}(m)')=0 = \Ext^1(R_{\pi}(m)',R_{\phi}(n))$. Lemma \ref{lem:sathom} then implies $\dim(\Hom(R_{\phi}(m),R_{\pi}(n)') = mn = \dim(\Hom(R_{\pi}(n)',R_{\phi}(m))$; combining this with Lemma \ref{lem:sumcoeff}  implies the final equality.
\end{proof}

\subsection{Projective resolutions of indecomposables}


In this section, we record projective resolutions of all indecomposable objects in $\Rep(\qrud)$. While we won't be needing these in the rest of the paper, they may be of independent interest. (In fact, we use certain Hall product computations to find these resolutions.) In particular, we show that the category $\Rep(\qrud)$ has infinite global dimension.

Recall that the only indecomposable projective objects in $\Rep(\qrud)$ are $M$ and $M'$. In the following lemma, we construct surjections from projective objects to every other indecomposable object:

\begin{lem} \label{lem:projcover}
\begin{enumerate} [label=(\alph*)]
\item For all $n\geq 1$, there exist short exact sequences:
\begin{align*}
0\to P'_{n-1} \to M^n \to P_n \to 0&,\\
0\to P_{n-1} \to M'^n \to P_n' \to 0&,\\
0\to I'_{n} \to M^n \to I_{n-1} \to 0&,\\
0\to I_{n} \to M'^n \to I_{n-1}' \to 0&.
\end{align*}
\item For all $\phi\in \Sigma_d$ and all $n\geq 1$, there exist short exact sequences:
\begin{align*}0 \to R_{{\phi}}(n)'\to M^{dn}\to R_{\phi}(n)\to 0,&\\
0 \to R_{{\phi}}(n)\to M'^{dn}\to R_{\phi}(n)'\to 0.&
\end{align*}
\end{enumerate}
\end{lem}

\begin{proof}
\begin{enumerate} [label=(\alph*)]
\item Taking $n\geq 1$ and $m=1-n$ in Theorem~\ref{thm:mainrel}, we get that $M^n$ occurs as an extension of $P_n$ by $P_{n-1}'$ (since it occurs in their Hall product), proving the existence of the first sequence. The others can be seen to follow similarly from Theorem~\ref{thm:mainrel}, or by applying $\sigma, \tau$ and $\sigma\circ\tau$ to the first sequence respectively.
\item By Lemma~\ref{lem:ckconst}, we conclude that $M^{dn}$ occurs in the Hall product $[R_{\phi}(n)][R_{\phi}(n)']$ with non-zero coefficient, proving the existence of the first short exact sequence. The second one follows by applying $\sigma$ to the first sequence.
\end{enumerate}
\end{proof}

\begin{cor}
We have projective resolutions of the following form for all $n\geq 0$: 

\begin{align*}
\cdots \to M'^{n+4}\to M^{n+3} \to M'^{n+2} \to M^{n+1}\to I_n\to0&,\\
\cdots \to M^{n+4}\to M'^{n+3} \to M^{n+2} \to M'^{n+1}\to I_n'\to0&,\\
\cdots \to M'^3\to M^2 \to M'\to M'\to M^2\to \cdots \to M^{n-1} \to M'^{n} \to M^{n+1}\to P_{n+1}\to0& \text{ (}n+1\text{ even),}\\
\cdots \to M^3\to M'^2 \to M\to M\to M'^2\to \cdots \to M^{n-1} \to M'^{n} \to M^{n+1}\to P_{n+1}\to0& \text{ (}n+1\text{ odd),}\\
\cdots \to M^3\to M'^2 \to M\to M\to M'^2\to \cdots \to M'^{n-1} \to M^{n} \to M'^{n+1}\to P_{n+1}'\to0& \text{ (}n+1\text{ even),}\\
\cdots \to M'^3\to M^2 \to M'\to M'\to M^2\to \cdots \to M'^{n-1} \to M^{n} \to M'^{n+1}\to P_{n+1}'\to0& \text{ (}n+1\text{ odd).}
\end{align*}
For all $\phi\in\Sigma_d$ and $n\geq 1$, we have projective resolutions of the following form:
\begin{align*}
\cdots \to M'^{dn}\to M^{dn} \to M'^{dn} \to M^{dn}\to R_{\phi}(n)\to0&,\\
\cdots \to M^{dn}\to M'^{dn} \to M^{dn} \to M'^{dn}\to R_{\phi}(n)'\to0&.
\end{align*}
\end{cor}

\begin{remark}
Applying $\tau$ to the above sequences, we get injective resolutions  for all the indecomposables.
\end{remark}

\begin{lem}
If $X$ is any non-projective $\qrud$-representation, then $X$ has infinite projective dimension. In particular, the category $\Rep(\qrud)$ has infinite global dimension.
\end{lem}

\begin{proof}
Without loss of generality, we will assume that $X$ is indecomposable. First, suppose $X$ is a preprojective or a preinjective. Thus, the object $X$ doesn't have slope zero. However, as $M$ and $M'$ are the only projective indecomposables in our category, both of which have slope zero, an Euler characteristic argument shows that $X$ can not have a finite projective resolution.

Next, suppose $X$ is a regular indecomposable. Without loss of generality, let $X=R_{\phi}(n)$ for some $\phi\in \Sigma$. The proof is by induction on $n$. Suppose $X$ has projective dimension $k$. Thus, there must be a short exact sequence:
\[0\to Y\to P \to X\to 0,\]
where $P$ is projective, and $Y$ has projective dimension $k-1$. By the first part of the proof, we get that $Y$ can not have any preprojective or preinjective summands, and so $Y$ must be a direct sum of regular representations. Then, by Lemma~\ref{lem:regularext}, we must have that $Y$ is a direct sum of representations of the form $R_{\phi}(m)'$, for varying $m$. By induction, if $m<n$, we get that $R_{\phi}(m)'$ has infinite projective dimension. So, we can assume that $Y=R_{\phi}(n)'$. However, by the duality $\sigma$, the projective dimensions of $R_{\phi}(n)$ and $R_{\phi}(n)'$ should be equal, which gives us the required contradiction, showing that their projective dimension must be infinite.
\end{proof}

\section{Hall algebra computations}\label{sec:hallcomputations}
In this section we prove a number of identities in the (untwisted) Hall algebra $Hall(\qrud)$. In the next section we will use these identities to prove our main theorem -- the existence of an algebra map from the deformed quantum affine algebra $\uu_v(\widehat{\sl_2})_{1,1}$ to our (twisted, localized) Hall algebra. We start by  proving several elementary lemmas that we use frequently throughout this section.



\begin{lem} \label{lem:radical}
Suppose $X$ and $Y$ are modules with no isomorphic indecomposable summands. Then 
\[\mathrm{Aut}(X\oplus Y) = \mathrm{Aut}(X)\times \mathrm{Aut}(Y)\times \Hom(X,Y)\times \Hom(Y,X).\]
\end{lem}
\begin{proof}
Note that
\[\mathrm{End}(X\oplus Y)=\begin{bmatrix}\mathrm{End}(X)& \mathrm{Hom}(Y,X)\\\mathrm{Hom}(X,Y)& \mathrm{End}(Y)\end{bmatrix}.\]
Pick a block matrix $\begin{bmatrix}A &B\\C& D\end{bmatrix}$ in the above space. This block matrix can be written as the sum
\[\begin{bmatrix}A &B\\C& D\end{bmatrix} = \begin{bmatrix}A &0\\0& D\end{bmatrix}+\begin{bmatrix}0 &B\\C& 0\end{bmatrix},\]
where the second summand is nilpotent by Fitting's Lemma. Therefore, the matrix $\begin{bmatrix}A &B\\C& D\end{bmatrix}$ is invertible if and only if $\begin{bmatrix}A &0\\0& D\end{bmatrix}$ is invertible, proving the claim.
\end{proof}

We combine this with \cite[Proposition 4]{Ring} to obtain the following useful technical lemma: 

\begin{lem}\label{lem:sumcoeff}
    If $\Ext^1(X,Y) = 0$, then the coefficient $c$ in the product $[X][Y] = c[X\oplus Y]$ is 
    \[
    c = \frac{\alpha({X\oplus Y)}}{\alpha(X) \alpha(Y) \big\lvert \Hom(X,Y) \big\rvert } 
    \]
    where $\alpha(P) := |\mathrm{Aut}(P)|$. If $X$ and $Y$ have no isomorphic indecomposable summands, then $c = \lvert \Hom(Y,X) \rvert$.
\end{lem}
\begin{proof}
    The first claim is \cite[Proposition 4]{Ring}, and the second claim follows directly from Lemma \ref{lem:radical}. 
\end{proof}



We start by studying the Hall products involving the projective representations $M$ and $M'$.

\begin{lem} \label{lem:mcentral}
The elements $[M]$ and $[M']$ are central in the Hall algebra.
\end{lem}
\begin{proof}
Without loss of generality, we only need to show that $[M]$ is central. Since the Hall algebra is generated as an algebra by the indecomposable representations by \cite{BG16} (see Theorem \ref{thm:monobasis}), we only need to show that $[M]$ commutes with $[Y]$ for every indecomposable $Y\neq M$.
We know $\Ext^1(Y,M)=\Ext^1(M,Y)=0$ since $M$ is projective and injective, so Lemma \ref{lem:sumcoeff} implies
\begin{align*}
[Y]\cdot[M]&=q^{\dim(\Hom(M,Y))}[Y\oplus M],\\
[M]\cdot[Y]&=q^{\dim(\Hom(Y,M))}[Y\oplus M].
\end{align*}
Then Corollary \ref{cor:mhom} shows the equality of the powers of $q$ in these two products.
\end{proof}

Recall that in Bridgeland's construction of the quantum group $\mathcal{U}_v(\mathfrak{g})$, acyclic complexes are inverted and shift-invariant cyclic complexes are set equal to 1. Our construction requires a localization of a similar flavor -- it is well-defined because $[M]$ and $[M']$ are central and hence satisfy the Ore conditions.
\begin{dfn}\label{def:hh} 
    Let $\hh$ be the localization of $Hall(\qrud)$ at the elements $[M]$ and $[M']$ with formal fourth roots of $(q-1)[M]$ and $(q-1)[M']$ adjoined
    \[
    \hh := Hall(\qrud)\langle a,b,a',b'\rangle/R
    \]
    where the relations $R$ we impose are that $a,b,a',b'$ are central, and that
    \[
    a^4=(q-1)[M],\quad ab=1,\quad (a')^4=(q-1)[M'],\quad a'b'=1.
    \]
    From now on we will very mildly abuse notation and write $(q-1)^{1/4}[M]^{1/4} := a$, and so on.
\end{dfn}

Most of the computations in the rest of this section have two flavors: computing relations inside the Hall algebra $Hall(\qrud)$ between specific modules using Definition \ref{def:hallalg}, or performing formal algebraic manipulations in $\hh$ to show our Hall algebra relations imply the relations in the deformed quantum affine algebra $\uu_v(\widehat{\sl_2})_{1,1}$ (see Definition~\ref{def:shiftedloop}). For the convenience of the reader, we collect all identities of the first type into the following Theorem.

\begin{thm}\label{thm:relsummary}
    The following identities hold in the (untwisted) Hall algebra of the category of $\qrud$-modules and are proved in Sections \ref{sec:eachhalf} and \ref{sec:bothhalves}: 
    \begin{align*}
        [[X],[P_n]]_q  &= [P_{n+1}]\\
        [[I_n],[X]]_q  &= [I_{n+1}]\\
        [[P_{n+1}'],[X]]_q  &=  (q-1)[M'][P_{n}']\\
        [[X],[I_{n+1}']]_q  &= (q-1)[M][I_{n}']\\
        [[I_0],[P_0]] &= R_1-R_1'\\
        [[I_0], [P_1]]_q &= R_2 - q[M]\\
        [[X],R_1']&=[M]-[M']\\
        [[P_1],[P_1']]&=(q-1)([M]R_1'-[M']R_1) \\
        [M],[M'] \,&\textrm{are central}.
    \end{align*}
    The following identities between regular representations are proved in Section \ref{sect:regular}: 
    \begin{align*}
        [[R_{\phi}(n)],[R_{\pi}(m)]]&=0\\
        [[R_{\phi}(n)'],[R_{\pi}(m)']]&=0\\
        [[R_{\phi}(m)],[R_{\pi}(n)']] &= 0 \qquad \mathrm{if} \quad \phi\neq {\pi}.
    \end{align*}
\item If $\phi \in \Sigma_1$, then 
\begin{align*}
[R_{\phi}(1)][R_{\phi}(1)'] &= q[R_{\phi}(1)\oplus R_{\phi}(1)'] + [M],\\
[R_{\phi}(1)'][R_{\phi}(1)] &= q[R_{\phi}(1)\oplus R_{\phi}(1)'] + [M'].
\end{align*}
More generally, when $\phi  \in \Sigma_d$ we shorten notation by writing $J_m:= R_{\phi}(m)$ and $J_n':=R_{{\phi}}(n)'$, and prove
\[[J_m][J_n'] = \sum_{k=0}^{\min\{m,n\}}c_k[J_{m-k}\oplus J_{n-k}'\oplus M^{dk}]\]
where the constants $c_k$ are computed explicitly in Lemma \ref{lem:ckconst}.
\end{thm}

\begin{rem}
The identities from Section \ref{sect:regular} in the above theorem aren't strictly needed in the proof of our main result Theorem~\ref{thm:SQAA} (although we do use them to provide an independent proof of Proposition~\ref{prop:rel9} in the form of Proposition~\ref{prop:rel9''}). 
We state them here so that this is a complete list of the  identities that we prove Hall-theoretically (instead of by formal algebraic manipulation).
\end{rem}

\subsection{Slope subalgebras} \label{sec:eachhalf}


Recall that a $\qrud$-module $Y$ has positive slope if $\dim(Y_1) < \dim(Y_2)$. This condition is preserved under extensions, so positive slope modules span a subalgebra of the Hall algebra. In this section we study relations between indecomposables with positive (resp.~negative) slope (see Propositions \ref{prop:braid} and \ref{prop:braidbb}).

It will be helpful to first prove some relations involving ``small'' regular representations; to simplify notation, let $X$ be an indecomposable regular representation $R_{\phi}(1)$ with $\phi \in \Sigma_1$. We will view $\phi$ as a point $[a:b]\in \PP^1(\FF_q)$, so that $e$ and $f$ are multiplication by $a$ and $b$ respectively. Next, fix $X'=R_{\pi}(1)'$ such that $\pi\in \Sigma_1$ and $\pi\neq \phi$. Thus, by Corollary~\ref{lem:regcomm}, the elements $[X]$ and $[X']$ commute.


First, we consider the commutators of $[X]$ and $[X']$ with non-regular indecomposables in the ``same direction.'' The following relations can thus be seen as relations in a Kronecker subquiver of the $\qrud$.

\begin{lem}\label{lem:xpcom}
For all $n\geq 0$, we have:
\begin{align*}
q[X][I_n] - [I_n][X] = - [I_{n+1}],\\
[X][P_n] - q[P_n] [X] = [P_{n+1}],\\
q[X'][I_n'] - [I_n'][X'] = - [I_{n+1}'],\\
[X'][P_n'] - q[P_n'] [X'] = [P_{n+1}'].
\end{align*}
\end{lem}

\begin{proof}
We only need to prove the first equality, since the others follows from it by applying $\sigma\circ\tau$, $\sigma$ and $\tau$ respectively. When $n\geq 1$, we have that $I_n$ is $(e,f)$-saturated, and so the statement is a consequence of Proposition~\ref{prop:satequiv} and \cite[Lemma 4.1(b)]{Sz06}. When $n=0$, we use Lemma~\ref{lem:extxp0} and \cite[Lemma 4.1(b)]{Sz06} to conclude the result. 
\end{proof}


Next, we consider commutators with non-regular indecomposables in the ``opposite direction.''

\begin{lem}\label{lem:xppcom}
For all $n\geq 0$, we have:
\begin{align*}
q[X][P_{n+1}'] - [P_{n+1}'][X] &=  -(q-1)[M'][P_{n}'],\\
[X][I_{n+1}'] - q[I_{n+1}'][X] &= (q-1)[M][I_{n}'],\\
q[X'][P_{n+1}] - [P_{n+1}][X'] &=  -(q-1)[M][P_{n}],\\
[X'][I_{n+1}] - q[I_{n+1}][X'] &= (q-1)[M'][I_{n}].
\end{align*}
\end{lem}


\begin{proof}
Again, we only need to prove the first equality, since the others follows from it by applying $\sigma\circ\tau$, $\sigma$ and $\tau$ respectively. The proof is by induction on $n$. For $n=0$, the required equality is a consequence of the following Hall product computations:
\begin{align}
    [X][P_{1}'] &= q[X\oplus P_{1}'],\label{eq:h1}\\
    [P_{1}'][X] &= q^{2}[X\oplus P_{1}'] + (q-1)[M'\oplus P_{0}']. \label{eq:h1p1}
\end{align}

To compute these products, Lemma \ref{lem:extxp} shows $\Ext^1(X,P_{1}')=0$ and $\Ext^1(P_{1}',X)=\FF_q$, and also shows that in the latter case, the nontrivial extension is isomorphic to $M' \oplus P'_0$. To complete the proof of \eqref{eq:h1}, all that remains is the computation of the coefficients.

To compute the coefficients  in \eqref{eq:h1}, we note that $\dim(\Hom(P_{1}',X))=1$ (see Lemma \ref{lem:sathom}). For \eqref{eq:h1p1}, the coefficient of $[X\oplus P_{1}']$ follows from $\dim(\Hom(X,P_{1}'))=2$ (see Lemma \ref{lem:sathom}). In order to compute the coefficient of $[M'\oplus P_0']$, we need to compute the number of subrepresentations $W$ of $M'\oplus P_0'$ isomorphic to $X$ such that $(M'\oplus P_0')/W \cong P_{1}'$.

Fix a basis vector $x$ for the simple representation $P_0'$ and basis $\{x_1,x_2,y_1,y_2\}$ for $M'$ as in equation \eqref{eq:mpdef}. Then, the only subrepresentations $W$ of $P_0'\oplus M'$ that are isomorphic to $X$ are the ones that are spanned by $\{ax_2+bx_1+cx, y_2\}$ for some $c\in \FF_q$. If $c\neq 0$, the quotient $(M'\oplus P_0')/W$ is isomorphic to $M'/X \oplus P_0'$, which is decomposable, and thus, not isomorphic to $P_1'$. On the other hand, when $c\neq 0$, it is clear that the quotient $(M'\oplus P_0')/W$ is indeed isomorphic to $P_1'$. Thus, there are $q-1$ valid choices for $c$, which completes the proof of \eqref{eq:h1p1}. This proves the base case $n=0$ for the statement of the lemma.

For the inductive step, suppose the equality
\begin{equation} \label{eq:tempo}
q[X][P_{n+1}'] - [P_{n+1}'][X] =  -(q-1)[M'][P_{n}']
\end{equation}
is known for some $n\geq 0$. We take the $q$-commutator of both sides with $[X']$: 
\begin{align*}
[[X'],LHS\eqref{eq:tempo}]_{{q}}&= q[X][[X'],[P_{n+1}']]_q - [[X'],[P_{n+1}']]_q[X]\\
&=q[X][P_{n+2}'] - [P_{n+2}'][X],
\end{align*}
where the first equality uses the fact that $[X']$ commutes with $[X]$, and the second equality follows from Lemma~\ref{lem:xpcom}.  Similarly, using Lemma~\ref{lem:xpcom} and the fact that $[M']$ is central, we compute that $[[X'],RHS\eqref{eq:tempo}]_q=-(q-1)[M'][P_{n+1}']$. Equating the $q$-commutators of both sides completes the induction.
\end{proof}

Next, to simplify notation we define the following sequences of elements in the localization $\hh$ of the Hall algebra $Hall(\qrud)$:

 \begin{align*}
     A_n&:=\begin{cases*}
      [P_n] & if $n\geq 0$ \\
      (q-1)^n [M]^n [I_{-n}'] & if $n\leq 0$
    \end{cases*},\\
B_{n}&:=\begin{cases*}
      [I_{n}] & if $n\geq 0$ \\
      (q-1)^{n} [M']^{n} [P_{-n}'] & if $n\leq 0$
    \end{cases*}.
\end{align*}
Note that
\begin{align*} \tau(\sigma(A_n))&=B_{n},\\
\sigma(B_n)=\tau(A_n) &= (q-1)^{n}[M]^{n} A_{-n},\\
\sigma(A_n)=\tau(B_n) &= (q-1)^{n}[M']^{n} B_{-n}.
\end{align*}
Using this notation, the above lemmas can be reformulated as follows:
\begin{cor} \label{cor:base_h}
For all $n\in \ZZ$ we have the equalities:
\begin{align*}
[[X],A_n]_q &= A_{n+1}\\
[A_n,[X']]_q&=  (q-1)[M]A_{n-1}\\
[B_n,[X]]_q &= B_{n+1}\\
[[X'],B_n]_{q}&= (q-1)[M']B_{n-1}.
\end{align*}
\end{cor}

\begin{prop} \label{prop:braid}
For any $m, n\in \ZZ$ such that $m\geq n$:
\begin{enumerate} [label=(\alph*)]
\item If $m-n$ is odd,
\[A_mA_n= q^{m-n+1}A_nA_m + (q^{m-n+1} - q^{m-n-1})\sum_{k=1}^{\frac{m-n-1}{2}}A_{n+k}A_{m-k}.\]
\item If $m-n$ is even,
\[A_mA_n= q^{m-n+1}A_nA_m + (q^{m-n+1} - q^{m-n-1})\sum_{k=1}^{\frac{m-n}{2}}A_{n+k}A_{m-k} + (q^{m-n}-q^{m-n+1})A_{\frac{m+n}{2}}^2.\]
\end{enumerate}
\end{prop}

\begin{proof}
Given $m,n \in \ZZ$ with $m\geq n$, let $R_{m,n}$ be the statement that the above proposition is true. By \cite[Theorem 4.2]{Sz06}, we know that $R_{m,n}$ is true whenever $m\geq n\geq 0$. We'll prove the proposition by using negative induction on $n$. That is, assuming that $R_{m,n}$ is true for all $m\geq n\geq t$, we will show that $R_{m,n}$ is true for all $m\geq n\geq t-1$.

Suppose $R_{m,n}$ is known to be true for some $m\geq n$ such that $m-n$ is odd. By induction, we know that
\begin{equation}\label{eq:aa} 
A_mA_n= q^{m-n+1}A_nA_m + (q^{m-n+1} - q^{m-n-1})\sum_{k=1}^{\frac{m-n-1}{2}}A_{n+k}A_{m-k}.
\end{equation}
We compute the $q^{-2}$-commutator of both sides of equation \eqref{eq:aa} with the element $[X']$. The left side becomes
\begin{align}
[[X'],A_mA_n]_{q^{-2}} &= [[X'],A_m]_{q^{-1}}A_n + q^{-1}A_m[[X'],A_n]_{q^{-1}}\notag\\
&=(q^{-1}-1)[M](A_{m-1}A_n + q^{-1}A_mA_{n-1}).\label{eq:aal}
\end{align}
Similarly, the $q^{-2}$-commutator of the right side with $[X']$ is
\begin{align}
[[X'],RHS \eqref{eq:aa}] &=(q^{-1}-1)[M]\Big(q^{m-n+1}A_{n-1}A_m + (q^{m-n+1} - q^{m-n-1})\sum_{k=1}^{\frac{m-n-1}{2}}A_{n+k-1}A_{m-k}\Big)\label{eq:aar}\\
&\phantom{=} +(q^{-1}-1)[M]\Big(q^{m-n}A_{n}A_{m-1} + (q^{m-n} - q^{m-n-2})\sum_{k=1}^{\frac{m-n-1}{2}}A_{n+k}A_{m-k-1}\Big).\notag
\end{align}

By induction, we know that
\[A_{m-1}A_n= q^{m-n}A_nA_{m-1} + (q^{m-n} - q^{m-n-2})\sum_{k=1}^{\frac{m-n-1}{2}}A_{n+k}A_{m-k-1} + (q^{m-n-1}-q^{m-n})A_{\frac{m+n-1}{2}}^2.\]
Equating the $q^{-2}$-commutators \eqref{eq:aal} and \eqref{eq:aar} and using the above inductive hypothesis, we see that
\begin{align*}
A_mA_{n-1}&= q^{m-n+2}A_{n-1}A_m + (q^{m-n+2} - q^{m-n})\sum_{k=1}^{\frac{m-n-1}{2}}A_{n+k-1}A_{m-k} - (q^{m-n}-q^{m-n+1})A_{\frac{m+n-1}{2}}^2\\
&= q^{m-n+2}A_{n-1}A_m + (q^{m-n+2} - q^{m-n})\sum_{k=1}^{\frac{m-n+1}{2}}A_{n+k-1}A_{m-k} - (q^{m-n+2} - q^{m-n})A_{\frac{m+n-1}{2}}^2\\
&\phantom{=}- (q^{m-n}-q^{m-n+1})A_{\frac{m+n-1}{2}}^2\\
&=q^{m-n+2}A_{n-1}A_m + (q^{m-n+2} - q^{m-n})\sum_{k=1}^{\frac{m-n+1}{2}}A_{n+k-1}A_{m-k} +(q^{m-n+1}-q^{m-n+2})A_{\frac{m+n-1}{2}}^2,
\end{align*}
which proves $R_{m,n-1}$. In case $m-n$ is even, we can use a computation similar to the one above to show that $R_{m,n}$ implies $R_{m,n-1}$. This completes the proof.
\end{proof}

Applying $\tau\circ\sigma$, we get a dual of the above proposition involving the sequence $(B_n)$:

\begin{prop}\label{prop:braidbb}
For any $m, n\in \ZZ$ such that $m\leq n$:
\begin{enumerate} [label=(\alph*)]
\item If $n-m$ is odd,
\[B_mB_n= q^{n-m+1}B_nB_m + (q^{n-m+1} - q^{n-m-1})\sum_{k=1}^{\frac{n-m-1}{2}}B_{n+k}B_{m-k}.\]
\item If $n-m$ is even,
\[B_mB_n= q^{n-m+1}B_nB_m + (q^{n-m+1} - q^{n-m-1})\sum_{k=1}^{\frac{n-m}{2}}B_{n+k}B_{m-k} + (q^{n-m}-q^{n-m+1})B_{\frac{m+n}{2}}^2.\]
\end{enumerate}
\end{prop}

\subsection{Relations between slope subalgebras}\label{sec:bothhalves}

In this section, we prove relations between the sequences $(A_n)$ and $(B_n)$, which we state as follows: 
\begin{thm} \label{thm:mainrel}
For all $m,n\in \ZZ$,
\[[B_m,A_n]_{q^{m+n}} = R_{m+n+1}-q^{m+n}(q-1)^{m+n}[M']^m[M]^nR_{-m-n+1}'.\]
(Here, by convention, we assume that $R_d=R_d'=0$ for $d<0$.)
\end{thm}
\begin{proof}
When $m,n\geq 1$, the theorem is a consequence of \cite[Theorem 4.3]{Sz06} and Proposition \ref{prop:satequiv}. Applying $\sigma$ 
to that, we get that the theorem is also true when $m,n\leq -1$.
When $n\geq 0\geq m$, this is a reformulation of Proposition \ref{prop:halves} below. Applying $\sigma$ to this, we get that the theorem is true when $m\geq 0\geq n$.  That covers all the cases and proves the theorem.
\end{proof}
\begin{rem} \label{rem:shifted}
    At a very concrete level, one reason we see the shifted affine algebra inside our Hall algebra is apparent in the above formula -- on the right hand side we see $R_{m+n+1}$, whereas in the unshifted case we would expect to see $R_{m+n}$.
\end{rem}

The identities we need can be proved inductively with the help of some commutator identities involving ``small'' regular representations -- the indecomposable representation $X=R_{\phi}(1)$ of dimension $(1,1)$ from the previous section. Also, recall that we have defined
\[
R_1 =  \sum_{[a:b] \in \mathbb{P}^1}  \begin{tikzcd}[arrow style=tikz,>=stealth,row sep=4em]
\FF_q \arrow[rr,shift left = 1ex, "e = a"]
\arrow[rr, shift right = 1ex, swap, "f = b"] && \FF_q
\end{tikzcd} =: \sum_{[a,b] \in \mathbb{P}^1} R(a,b)
\]
where, to ease notation, we have defined $R(a,b)$ using the sum above.

The following 3-parameter version of the Jacobi relation  holds in any algebra where $a,b,c$ are central and invertible, and it will be useful in some of the computations below.
\begin{equation}\label{eq:qjac}
[x,[y,z]_{a}]_{bc^{-1}} + ac^{-1}[z,[x,y]_{b}]_{a^{-1}c}+bc^{-1}[y,[z,x]_c]_{ab^{-1}} = 0.
\end{equation}

\begin{lem}\label{lem:comm0}
We have the following commutator identities:
\begin{align*}
    [B_0,A_0] &= R_1-R_1'\\
    [B_0, A_1]_q &= R_2 - q[M].
\end{align*}
\end{lem}


\begin{proof}
 Note that for any $(a,b)$ we have a (non-split) short exact sequence
\[
0 \to P_0 \to R(a,b) \to I_0 \to 0.
\]
Any extension of $I_0$ by $P_0$ must have dimension vector $(1,1)$, and hence must be either $I_0 \oplus P_0$ or $R(a,b)$ for a unique $(a,b) \in \mathbb{P}^1$. Furthermore, each $R(a,b)$ has a unique proper submodule, and this submodule is isomorphic to $P_0$. These observations imply
\[B_0A_0 = [I_0][P_0] = [P_0\oplus I_0] + R_1.\]
A similar argument shows the identity
\[A_0B_0 = [P_0][I_0] = [P_0\oplus I_0] + R_1',\]
and combining these proves our first claim.

The second claim follows from the following Hall products:
\begin{align*}
    B_0A_1 &= [I_0][P_1] = q[I_0\oplus P_1] + R_2,\\
A_1B_0 &= [P_1][I_0] = [I_0\oplus P_1] + [M].
\end{align*}
To see these, we note that $I_0$ and $P_1$ are both $(e,f)$-surjective, and so, by Lemma~\ref{lem:saturated}, any extensions between them should also be $(e,f)$-surjective. Therefore, the only possible extension that doesn't come from the Kronecker quiver is $M$. Then, the above equalities follow from \cite[Theorem 4.3]{Sz06}  after observing that $[M]$ has exactly one submodule isomorphic to $I_0$ and none isomorphic to $P_1$. 
\end{proof}

\begin{lem}\label{lem:somanycommutators}
We have the following commutator identities:
\begin{align*}
    [[X],R_1']&=[M]-[M']\\
    [P_1,P_1']&=(q-1)\left([M]R_1'-[M']R_1\right)\\
    [[X],R_2']&=(q-1)([M]-[M'])R_1' .
    \end{align*}
\end{lem}
\begin{proof}
Recall that $X=R_{\phi}(1)$ for some $\phi\in\Sigma_1$. Thus, by Lemma~\ref{lem:regcomm}, we have that $[X,[R_{\pi}(n)]]=0$ for any $n$, if $\pi\neq\phi$. By definition,
\[
R_1':= \sum_{\pi\in\Sigma_1} [R_{\pi}(1)'].
\]
So, by Lemma~\ref{lem:mainreg}, we have
\[[[X],R_1'] = [R_{\phi}(1),R_{\phi}(1)'] = [M]-[M'].\]

The second identity follows from the following Hall algebra computation (and its dual obtained by applying $\sigma$):
\begin{equation}\label{eq:incommproof}
[P_1][P_1'] = q^2[P_1\oplus P_1'] + \sum_{\phi\in\Sigma_1} (q^2-q)[M\oplus R_{\phi}(1)']=q^2[P_1\oplus P_1']+(q-1)[M]R_1'.
\end{equation}
The second step in \eqref{eq:incommproof} follows from Lemma~\ref{lem:sumcoeff}. 
To prove the first step, we note that in the representation $P_1$, we have $\mathrm{rank}(e\oplus f)=2$, and similarly, in $P_1'$ we have $\mathrm{rank}(e'\oplus f')=2$. Thus, in an extension of $P_1$ by $P_1'$, we must have both $\mathrm{rank}(e\oplus f)\geq 2$ and $\mathrm{rank}(e'\oplus f') \geq 2$. The only representations of dimension $(3,3)$ that satisfy both of these conditions are $P_1\oplus P_1'$, $M\oplus R_{\phi}(1)'$ and $M'\oplus R_{\phi}(1)$ for some $\phi\in\Sigma_1$. It is clear that $M'\oplus R_{\phi}(1)$ does not have any submodule isomorphic to $P_1'$. The $q^2$ coefficient of $P_1\oplus P_1'$ occurs because the image of any injective map $P_1' \to P_1\oplus P_1'$ is uniquely determined by an element in the vector space $(P_1)_2$, which can be chosen arbitrarily by Lemma \ref{lem:sathom}.


To compute the $q^2-q$ coefficient in \eqref{eq:incommproof}, note that any map $g: P_1' \to M \oplus R_\phi(1)'$ is determined by the image of a generator $p \in (P_1')_2$, which must map to an element of the form $ax + by_1 + cy_2$ (where $x$ is a generator of $(R_\phi(1)')_2$).  If the maps $e',f'$ in $R_\phi(1)$ are multiplication by $m,n\in \FF_q$, respectively, then 
\begin{align*}
e'\cdot g(p) &= (bx_2, am) \in M\oplus R_\phi(1)'\\
f'\cdot g(p) &= (-cx_2, an) \in M\oplus R_\phi(1)'.
\end{align*}
In order to have $\mathrm{rank}(e'\oplus f')\geq 2$, we must have $bn+cm\neq 0$. This removes a hyperplane $H$ from the space $\FF_q^2 \ni (b,c)$. Therefore, the set of parameters $(a,b,c)$ lives in $(a,b,c) \in \FF_1^\times \times (\FF_q^2\setminus H) =: X$, and the number of submodules which are the image of such a $g$ is $\lvert X / \FF_q \rvert = q^2-q$. It is easy to check that the cokernel is indeed $P_1'$ for each such submodule.



For the third claimed equality, note that while the commutator $[[X],R_2']$ can be directly computed using Lemma~\ref{lem:mainreg}, we provide a different proof here that is independent of the results in Section~\ref{sect:regular}. We have
\begin{align*}
    [[X],R_2'] &= [X,[P_0,P_1']_q+q[M']]\\
    &= -[P_1',[X,P_0]_q] - [P_0,[P_1',X]_q]\\
    &= -[P_1',P_1] - [P_0,(q-1)[M']P'_0]\\
    &= -[P_1',P_1] - (q-1)[M'][P_0,P'_0]\\
    &= -[P_1',P_1] - (q-1)[M'](R_1'-R_1)\\
    &= (q-1)\left( [M] - [M']\right)R_1'.
\end{align*}
where the first equality uses Lemma~\ref{lem:comm0} (up to an application of $\tau$), the second equality follows by the Jacobi relation \eqref{eq:qjac} with $a=b=c=q$  and the fact that $[M']$ is central, 
the third equality is a consequence of Corollary~\ref{cor:base_h}, the second to last equality again follows from Lemma~\ref{lem:comm0}, and the final equality follows from our computation of $[P_1,P_1']$ above.


\end{proof}

We now find the commutator of $B_0$ with $A_n$ for arbitrary $n$. 
\begin{lem} \label{lem:ef-comm}
For all $n\geq 2$, we have
\[[B_0,A_n]_{q^n} = R_{n+1}.\]
\end{lem}

\begin{proof}
The proof is by induction on $n$, for the base case $n=2$, we start with the equality in Lemma \ref{lem:somanycommutators}: 
\[[B_0, A_1]_q = R_2 - q[M].\]
Note that any two regular representations with $e=f=0$ commute in the Hall algebra as a consequence of Proposition~\ref{prop:satequiv} and \cite[$\mathsection 3$]{Sz06}. Therefore, since $[X]$ commutes with the right hand side of the above identity and the identity we want to prove, we can prove the base case and inductive step simultaneously by taking the commutator of both identities with $[X]$. When $n \geq 1$ we then compute
\begin{align}
0&=[[X], [B_0, A_n]_{q^n}]\label{eq:qjacobiex}\\
&= -q^{n+1}[A_n,[[X],B_0]_{q^{-1}}]_{q^{-n-1}}-[B_0,[A_n,X]_{q^{-1}}]_{q^{n+1}}\notag\\
&= -q^{n+1}[A_n,-qB_1]_{q^{-n-1}}-[B_0,-qA_{n+1}]_{q^{n+1}}\notag
\end{align}
where in the first step we used the Jacobi relation \eqref{eq:qjac} with $a=q^n$ and $b=c=q^{-1}$, and in the second step we used Corollary \ref{cor:base_h}. Rearranging this equation and applying the identity $[x,y]_a = -a[y,x]_{a^{-1}}$, we obtain
\[
[B_0,A_{n+1}]_{q^{n+1}}=[B_1,A_n]_{q^{n+1}} = R_{n+2},
\]
where the final step follows from \cite[Theorem 4.3]{Sz06} and Proposition \ref{prop:satequiv}.
\end{proof}

\begin{cor}\label{cor:comms1}
For any $n\in \ZZ$, we have the following identities:
 \[[B_0,A_n]_{q^n}:=\begin{cases*}
      R_{n+1} & if $n\geq 2$ \\
      R_2 - q[M] & if $n=1$\\
    R_1 - R_1' & if $n=0$\\
    -q^{-1}(q-1)^{-1}[M]^{-1}(R_2' - q[M]) & if $n=-1$\\
          -q^n(q-1)^n[M]^n R_{-n+1}' & if $n\leq -2$\\
    \end{cases*}.\]
\end{cor}

\begin{proof}
The first three cases are proved above, and the last two follow from the first two after applying $\tau$.
\end{proof}

By applying $\tau\circ\sigma$ to Corollary \ref{cor:comms1}, we can compute the commutators $[B_n,A_0]_{q^n}$ for all $n$:

\begin{cor}\label{cor:comms2}
For any $m\in \ZZ$, we have the following identities:
 \[[B_m,A_0]_{q^m}:=\begin{cases*}
        R_{m+1} & if $m\geq 2$ \\
     R_2 - q[M'] & if $m=1$\\
    R_1 - R_1' & if $m=0$\\
    -q^{-1}(q-1)^{-1}[M]^{-1}(R_2' - q[M']) & if $m=-1$\\
          -q^m(q-1)^m[M']^mR_{-m+1}' & if $m\leq -2$\\
    \end{cases*}.\]
\end{cor}



The following proposition partly generalizes the two corollaries above:

\begin{prop}\label{prop:halves}
Suppose $n\geq 0\geq m$. Then,
 \[[B_m,A_n]_{q^{m+n}}:=\begin{cases*}
    R_{n+m+1} & if $n+m\geq 2$ \\
    R_2 - q[M']^m[M]^n & if $n+m=1$\\
    R_1 - [M']^m[M]^n R_1' & if $n+m=0$\\
    -q^{-1}(q-1)^{-1}[M']^m[M]^n(R_2' - q[M]^{-n}[M']^{-m})& if $n+m=-1$\\
      -q^{m+n}(q-1)^{m+n}[M']^m[M]^nR'_{-n-m+1}     & if $n+m\leq -2$\\
    \end{cases*}.\]
\end{prop}

\begin{proof}
The proof will by induction on $\min(n,|m|)$. The base case has been covered by Corollaries \ref{cor:comms1} and \ref{cor:comms2}. We note that we only need to prove the first three cases, as the last two cases follow from the first two after applying $\tau$.

Suppose the claim is true for whenever $min(n,|m|)\leq k$. Thus, we know that
\[[B_{-k}, A_{k+1}]_q = R_2 - q[M']^{-k}[M]^{k+1}.\]
Applying $\sigma$ to the above relation, we get
\begin{equation}\label{eq:indstep} 
[A_{k}, B_{-k-1}]_q = (q-1)^{-1}[M']^{-k-1}[M]^k(R_2' - q[M]^{-k}[M']^{k+1}).
\end{equation}
We take commutator of both sides of  equation \eqref{eq:indstep} with $[X]$. On the left, we get
\begin{align*}
[X,LHS\eqref{eq:indstep}] &= [[X],[A_{k}, B_{-k-1}]_q]\\
&=-[B_{-k-1},[[X],A_{k}]_q] - [A_k,[B_{-k-1},[X]]_q]\\
&=-[B_{-k-1},[[X],A_{k+1}] - [A_k,B_{-k}]\\
&=-[B_{-k-1},A_{k+1}] + \Big(R_1 - [M']^{-k}[M]^kR_1'\Big),
\end{align*}
where we used identity \eqref{eq:qjac} with $a=b=c=q$ in the first step, Corollary \ref{cor:base_h} in the second step and induction in the final step. 


Next, using the fact that $[M]$ and $[M']$ are central and  equality $[[X],R_2']=(q-1)([M]-[M'])R_1'$ from Lemma \ref{lem:comm0}  we get that the commutator of the right side of equation \eqref{eq:indstep} with $X$ is
\[[X,RHS\eqref{eq:indstep}] = [M']^{-k-1}[M]^k([M]-[M'])R_1'.\]
Equating the commutators of both the sides, we get the required equality
\begin{equation}\label{eq:commid3}
[B_{-k-1},A_{k+1}] = R_1 - [M']^{-k-1}[M]^{k+1}R_1'.
\end{equation}

To proceed further, we take the commutator of equation \eqref{eq:commid3} with $[X]$. Using a similar argument as above, we get that the commutator of the left side is equal to
\begin{align*}
[[X],LHS\eqref{eq:commid3}] &= [[X],[B_{-k-1}, A_{k+1}]]\\
&=q^{-1}[B_{-k-1},A_{k+2}]_q-q^{-1}[B_{-k},A_{k+1}]_{q}\\
&=q^{-1}[B_{-k-1},A_{k+2}]_q-q^{-1}\Big(R_2 - q[M']^{-k}[M]^{k+1}\Big),
\end{align*}
where the last equality follows by induction. To compute the commutator of the right side, we use the fact that $[X]$ commutes with $R_1$ and the identity $[[X],R_1']=[M]-[M']$ from Lemma \ref{lem:comm0} to get
\[[[X],RHS\eqref{eq:commid3}] = - [M']^{-k-1}[M]^{k+1}([M]-[M']).\]
Equating the two commutators of $[X]$ with equation \eqref{eq:commid3}, we get the required equality
\[[B_{-k-1},A_{k+2}]_q=R_2-q[M']^{-k-1}[M]^{k+2}.\]
To complete the proof, we need to compute $[B_{-k-1},A_{k+t}]_{q^{t-1}}$ for all $t\geq 2$. For this, we proceed by induction on $t$, using the same argument as we did in the proof of Lemma~\ref{lem:ef-comm}, by taking the commutator of the equality at the $t^{th}$ step with $[X]$ to obtain  the $(t+1)^{th}$ step.
\end{proof}

\section{Map from the quantum affine algebra to the Hall algebra}\label{sec:themap}

In this section, we construct a map from the deformed quantum affine algebra to a twisted version of our Hall algebra.  We fix a square root $\vv\in \CC$ of $q$ that will be used through this section. We also fix a square root of $\vv$ which is needed in the definition of the twist below (although it won't appear in any of our computations).

First, we twist the multiplication on the Hall algebra and define a new $*$-product. Note that the Grothendieck group $K_0(Hall(\Rep(\qrud)))=\ZZ^2$ is generated by the classes of the two simple representations. More concretely, for any representation $X$, its class $\overline X$ in the Grothendieck group is given by $\dim(X)=(\dim(X_1),\dim(X_2))$. 
The group $K_0(\hh)$ is obtained by adjoining $\left(-\frac{1}{2}, -\frac{1}{2}\right)$ to $\ZZ^2$ to account for $((q-1)[M])^{-1/4}$ and $(q-1)[M'])^{-1/4}$. 
\begin{defn} \label{def:twist}
For  elements $X, Y \in \hh$ that are homogeneous with respect to the $K_0(\hh)$-grading, 
 define
\[[X] * [Y]:= \vv^{-det(\overline X, \overline Y)}[X]\cdot[Y],\]
where `$\cdot$' is the usual product in $\hh$.
More generally, the above formula can be extended additively to define a product of any two elements of $\hh$. The space $\hh$ with this product will be referred to as the \emph{twisted Hall algebra} $\hh^{tw}$. 
\end{defn}


\begin{rem}
The composition of functors $\tau\circ\sigma$ defines an anti-automorphism on the twisted Hall algebra. 
The equivalences $\tau$ and $\sigma$ themselves do not induce algebra maps on the twisted algebra; instead, for any representations $X$ and $Y$, it is true that
\begin{align*}
\sigma([X]*[Y])&=q^{-det(\dim(X), \dim(Y))}\sigma([X])*\sigma([Y]),\\
\tau([X]*[Y])&=q^{-det(\dim(X), \dim(Y))}\tau([Y])*\tau([X]).
\end{align*}
\end{rem}

Under this new product, the elements $[M]$ and $[M]'$ aren't central anymore: given a representation $V$ of dimension $(m,n)$, they satisfy
\begin{align*}
[M]*[V] &= q^{2(n-m)}[V]*[M],\\
[M']*[V] &= q^{2(n-m)}[V]*[M'].
\end{align*}
Nevertheless, the above equalities imply that that element $[M']*[M]^{-1}$ is still central under the twisted product. A related observation is the fact that given two representations $X$ and $Y$ of slope zero, the twisted product $[X]*[Y]$ coincides with the usual product $[X]\cdot[Y]$.

We recall the following definition: 
\begin{defn}
    The specialization $\uu_{\vv}(\widehat{\sl_2})_{1,1}$ of the deformed quantum affine algebra $\uu_{v}(\widehat{\sl_2})_{1,1}$ at $v=\vv$ is the $\mathbb{C}$-algebra generated by elements $E_l, F_l, H_n, S^{\pm1}, K^{\pm1}, C^{\pm 1/2}$ for $l\in\ZZ, n\in \ZZ\setminus\{0\}$ satisfying the following relations:
    \begin{align}
        C^{1/2} &\textrm{ is central} \label{sela1}\\
        S^2&=KC^{1/2}\label{sela3.6}\\
        SE_k S^{-1} &= \vv E_k\label{sela2}\\
        SF_k S^{-1} &= \vv^{-1} F_k \label{sela3}\\
        SH_nS^{-1}&=H_n\label{sela3.5}\\
        E_{k+1}E_l-\vv^2 E_l E_{k+1} &= \vv^2 E_kE_{l+1}-E_{l+1}E_k\label{sela4}\\
        F_{k+1}F_l-\vv^{-2} F_l F_{k+1} &= \vv^{-2} F_kF_{l+1}-F_{l+1}F_k\label{sela5}\\
        [E_k,F_l] &= \frac{C^{(k-l)/2} \Psi_{k+l} - C^{(l-k)/2}\Phi_{k+l}}{\vv-\vv^{-1}}\label{sela6}\\
        [H_l,E_k] &= \frac{[2l]}{l}C^{-\lvert l\rvert/2}E_{k+l}\label{sela7}\\
        [H_l,F_k] &= \frac{-[2l]}{l}C^{\lvert l\rvert/2}F_{k+l} \label{sela8}\\
        [H_l, H_k] &= \delta_{l,-k}\frac{[2l]}{l}\frac{C^l-C^{-l}}{\vv-\vv^{-1}}\label{sela9}
    \end{align}
    where the elements $\Psi_k$ and $\Phi_k$ are defined via the following generating series:
    \begin{align}
        \sum_{k\geq 0}\Psi_{k-1} u^k &= K \mathrm{exp}\left( (\vv-\vv^{-1}) \sum_{k=1}^\infty H_k u^k \right) \notag \\
        \sum_{k\geq 0}\Phi_{-k+1} u^k &= K^{-1} \mathrm{exp}\left( -(\vv-\vv^{-1}) \sum_{k=1}^\infty H_{-k} u^{k} \right). \notag
    \end{align}
\end{defn}


In order to state our main result, we define some elements in $\hh$. First of all, we define:
\begin{align*}
s^{\pm1}&:=((\vv^2-1)[M])^{\mp1/4}, &s'^{\pm1}&:=((\vv^2-1)[M'])^{\mp1/4},\\
\kappa^{\pm1}&:=(s*s')^{\pm1},&c^{\pm1/2}&:=(s*s'^{-1})^{\pm1}.
\end{align*} 
For all $l\in \ZZ$, define:
\begin{align*}
e_l&:=s*c^{l/2}*\begin{cases*}
      [P_l] & if $l\geq 0$ \\
      \vv^{2l}s^{-4l} * [I_{-l}'] & if $l\leq 0$
    \end{cases*},\\
f_l&:=-s'*c^{-l/2}*\begin{cases*}
      [I_{l}] & if $l\geq 0$ \\
      \vv^{2l}s'^{-4l} * [P_{-l}'] & if $l\leq 0$
    \end{cases*}.
\end{align*}
Next, we define the sequences for $d\geq -1$:
\begin{align*}
\psi_d&:= \vv^{-d}(\vv-\vv^{-1})\kappa*R_{d+1},\\
\varphi_{-d}&:=\vv^{-d}(\vv-\vv^{-1})\kappa^{2d+1}*R'_{d+1}.
\end{align*}
Furthermore, let $\psi_d=\varphi_{-d}=0$ if $d<-1$. Finally, define the sequence $(h_n)$ for all $n\in\ZZ\setminus\{0\}$ as follows:
\begin{align*}        
\sum_{k\geq 0}\psi_{k-1} u^k &= \kappa* \mathrm{exp}\left( (\vv-\vv^{-1}) \sum_{k=1}^\infty h_k u^k \right),\\
\sum_{k\geq 0}\varphi_{-k+1} u^k &= \kappa^{-1} *\mathrm{exp}\left( -(\vv-\vv^{-1}) \sum_{k=1}^\infty h_{-k} u^{-k} \right).\end{align*}

\begin{thm} \label{thm:SQAA}
We have an algebra homomorphism $\theta_{\vv}:\uu_{\vv}(\widehat{\sl_2})_{1,1}\to \hh^{tw}$ from the deformed quantum affine algebra to the twisted Hall algebra that is defined on the generators as follows: 
\begin{align*}
C^{\pm 1/2} &\mapsto c^{\pm1/2}\\
K^{\pm1}&\mapsto \kappa^{\pm 1}\\
S^{\pm 1}&\mapsto s^{\pm1}\\
E_l\mapsto e_l, F_l&\mapsto f_l, H_n\mapsto h_n
\end{align*}
for all $l\in \ZZ$ and $n\in\ZZ\setminus\{0\}$.
\end{thm}

\begin{proof}
We need to show that the elements $c^{\pm1/2}, \kappa^{\pm1},e_l, f_l, h_n$ satisfy the relations of the deformed quantum {affine} algebra. Relations~\eqref{sela1},~\eqref{sela3.6},~\eqref{sela2},~\eqref{sela3} and \eqref{sela3.5} are clear. Relations~\eqref{sela4} and \eqref{sela5} are proved in Proposition~\ref{prop:rel45}. Relation~\eqref{sela6} is proved in Proposition~\ref{prop:rel6}. Relations~\eqref{sela7} and \eqref{sela8} are proved in Proposition~\ref{prop:rel78}. Relation~\eqref{sela9} is proved in Proposition~\ref{prop:rel9}.
\end{proof}

\begin{cor} \label{cor:4gens}
The image $\mathrm{Im}(\theta_{\vv})$ is generated by the elements $[I_0],[P_0], R_1, ((q-1)[M])^{\pm 1/4}$ and $((q-1)[M'])^{\pm1/4}$. Furthermore, the intersection $\mathrm{Im}(\theta_{\vv})\cap Hall(\qrud)$ is generated by  $[I_0],[P_0], R_1$ and $[M].$
\end{cor}

\begin{proof}
We only prove the second claim, since the first one follows consequently. Let $\mathcal{A}$ be the subalgebra of $Hall(\qrud)$ generated by $[I_0],[P_0], R_1$ and $[M].$ It is clear from the map $\theta_{\vv}$ defined above that the intersection $\mathrm{Im}(\theta_{\vv})\cap Hall(\qrud)$ is generated by $[I_n], [P_n], [I_n'], [P_n'], R_n, R_n', [M]$ and $[M']$. By Lemma~\ref{lem:comm0}, we have:
\[[[I_0], [P_0]] = R_1-R_1',\]
which shows that $R_1'\in \mathcal{A}$. Next, by Lemma~\ref{lem:xpcom}, we have that $[I_n], [P_n], [I_n'], [P_n'] \in \mathcal{A}$ for all $n\geq 0$. (We use the fact that $R_1$ and $R_1'$ are sums of all regular $(e,f)$-saturated (resp. $(e',f')$-saturated representations) of dimension $(1,1)$.)
Next, Lemma~\ref{lem:somanycommutators} shows that:
\[[R_1,R_1'] = (q+1)([M]-[M']),\]
showing that $[M']\in \mathcal{A}$. Finally, the equalities 
\begin{align*}
[[I_n],[P_0]]_q &= R_{n+1} - \delta_{n,1}q[M],\\
[[I_n'],[P_0']]_q &= R_{n+1}' - \delta_{n,1}q[M'],
\end{align*}
from Corollaries~\ref{cor:comms1} and \ref{cor:comms2} for $n\geq 1$ show that $R_n,R_n' \in \mathcal{A}$, which completes the proof.
\end{proof}
\begin{rem}
\begin{enumerate}
\item The equality
\[[I_0][P_0] = [I_0\oplus P_0] + R_1\]
implies that $[I_0\oplus P_0]$ can replace $R_1$ in the generating set.
\item While we work with the untwisted algebra $Hall(\qrud)$ in the above corollary, as the specified generators are all homogeneous elements in the Grothendieck group, the claim is also true in the twisted Hall algebra.
\end{enumerate}
\end{rem}

In the rest of this section, we prove the technical results that are used in the proof of Theorem~\ref{thm:SQAA}.

\begin{prop} \label{prop:rel45}
The sequences $(e_n)$ and $(f_n)$ satisfy the following relations for all $k,l\in \ZZ$:
\begin{align*}e_{k+1}*e_l - \vv^2e_l*e_{k+1} &= \vv^2e_k*e_{l+1} - e_{l+1}*e_k,\\
f_{k+1}*f_l - \vv^{-2}f_l*f_{k+1} &= \vv^{-2}f_k*f_{l+1} - f_{l+1}*f_k.
\end{align*}
\end{prop}

\begin{proof}
Note that:
\begin{align*}
e_k&=s*c^{k/2}* A_k,\\
f_k&=-s'*c^{-k/2} *B_k,
\end{align*}
where the sequences $(A_k)$ and $(B_k)$ were defined in Section~\ref{sec:eachhalf}. The relations that we are required to prove are homogeneous in $k+l$, and thus, to prove these relations, it suffices to show the following:
\begin{align*}
A_{k+1}*A_l - \vv^2A_l*A_{k+1} &= \vv^2A_k*A_{l+1} - A_{l+1}*A_k,\\
B_{k+1}*B_l - \vv^{-2}B_l*B_{k+1} &= \vv^{-2}B_k*B_{l+1} - B_{l+1}*B_k.
\end{align*}
We only prove the former relation, since the latter follows from the former upon applying $\tau\circ\sigma$. 
Without loss of generality, we assume that $k\geq l$. First, we suppose that $k-l+1$ is odd. Then, by Proposition~\ref{prop:braid}, we have the following identity involving the untwisted Hall product:
\[A_{k+1}A_l= q^{k-l+1}A_lA_{k+1} + (q^{k-l+1} - q^{k-l-1})\sum_{t=1}^{\frac{k-l}{2}}A_{l+t}A_{k+1-t}.\]
In terms of the twisted product `$*$', the above relation reads
\[\vv^{k-l+1}A_{k+1}*A_l= q^{k-l+1}\vv^{l-k+1}A_l*A_{k+1} + (q^{k-l+1} - q^{k-l})\sum_{t=1}^{\frac{k-l-1}{2}}\vv^{l-k+2t+1}A_{l+t}*A_{k+1-t},\]
which is equivalent to
\[A_{k+1}*A_l - \vv^2A_l*A_{k+1} = (q - q^{-1})\sum_{t=1}^{\frac{k-l}{2}}\vv^{2t}A_{l+t}*A_{k+1-t}.\]
If $k=l$, the sum on the right is empty, and we get the relation that we want to prove. If $k>l$, again by Proposition~\ref{prop:braid} we have
\[A_{k}*A_{l+1} = \vv^2A_{l+1}*A_{k} + (q - q^{-1})\sum_{t=1}^{\frac{k-l-2}{2}}\vv^{2t}A_{l+1+t}*A_{k-t}.\]
Then, we can compute that
\begin{align*}
\vv^2A_k*A_{l+1} - A_{l+1}*A_k &= (\vv^4-1)A_{l+1}*A_{k} + (q - q^{-1})\sum_{t=1}^{\frac{k-l-2}{2}}\vv^{2t+2}A_{l+1+t}*A_{k-t}\\
&=(q - q^{-1})\sum_{t=1}^{\frac{k-l}{2}}\vv^{2t}A_{l+t}*A_{k+1-t}\\
&=A_{k+1}*A_l - \vv^2A_l*A_{k+1},
\end{align*}
completing the proof. The case when $k-l+1$ is even follows similarly.
\end{proof}

\begin{prop} \label{prop:rel6}
The elements $e_k$ and $f_l$ for $k,l\in \ZZ$ satisfy the following relation:
\[e_k*f_l - f_l*e_k = \frac{c^{(k-l)/2}*\psi_{k+l}-c^{(l-k)/2}*\varphi_{k+l}}{\vv-\vv^{-1}}.\]
\end{prop}
\begin{proof}
By Theorem~\ref{thm:mainrel}, in the untwisted Hall product we have
\[[B_l,A_k]_{q^{k+l}} = R_{k+l+1}-q^{k+l}(q-1)^{k+l}[M']^l[M]^kR_{-k-l+1}'.\]
Using this, we can compute
\begin{align*}
e_k*f_l - f_l*e_k &= -\kappa*c^{(k-l)/2}*(\vv^{-1}A_k*B_l - \vv B_l*A_k)\\
&=-\kappa*c^{(k-l)/2}*(\vv^{k+l}A_kB_l-\vv^{-k-l}B_lA_k)\\
&=q^{-(k+l)/2}\kappa*c^{(k-l)/2}*(B_lA_k-q^{k+l}A_kB_l).
\end{align*}
Using the formula for $[B_l,A_k]_{q^{k+l}}$ from Theorem \ref{thm:mainrel}, we get that the above expression is a sum of two terms. The first term $T_1$ is equal to
\begin{align*}
T_1 &=q^{-(k+l)/2}\kappa*c^{(k-l)/2}*R_{k+l+1}\\
&=\frac{c^{(k-l)/2}*\psi_{k+l}}{\vv-\vv^{-1}}.
\end{align*}
Similarly, the second term $T_2$ is equal to
\begin{align*}
T_2 &=q^{-(k+l)/2}\kappa*c^{(k-l)/2}*q^{k+l}(q-1)^{k+l}*[M']^l*[M]^k*R_{-k-l+1}'\\
&=\frac{c^{(l-k)/2}*\varphi_{k+l}}{\vv-\vv^{-1}},
\end{align*}
We then obtain
\[
e_k*f_l - f_l*e_k = T_1 - T_2
\]
completing the proof.
\end{proof}
Recall that we use the notation $[n]$ for the quantum number $\dfrac{\vv^{n}-\vv^{-n}}{\vv-\vv^{-1}}$.
\begin{lem} \label{lem:base_h}
For all $k\in\ZZ$, we have
\begin{align*}[h_{\pm1},e_k] &= [2]c^{-1/2}*e_{k\pm1},\\
[h_{\pm1},f_k]&= -[2]c^{1/2}*f_{k\pm1}.
\end{align*}
\end{lem}
\begin{proof}
By chasing definitions of $\psi_1$ and $h_1$, we see that $h_1= R_1.$
The element $R_1$ is a sum of $q+1$ regular representations $R_{\pi}(1)$ for $\pi\in\Sigma_1$, each having dimension $(1,1)$. Thus, by Corollary~\ref{cor:base_h}, we have
\begin{align*}
[h_1, e_k] &= [R_1, e_k]\\
&=\sum_{\pi\in\Sigma_1}[[R_{\pi}(1)], e_k]\\
&=\vv^{-1}c^{-1/2}*(q+1){e_{k+1}}\\
&=[2]c^{-1/2}*e_{k+1}.
\end{align*}
The proofs of the other assertions are similar. 
\end{proof}


\begin{prop} \label{prop:rel78}
For all $l\in \ZZ\setminus\{0\}$ and $k\in\ZZ$, we have the following identities:
\begin{align*}[h_l,e_k] &= \frac{[2l]}{l}c^{-\lvert l\rvert/2}*e_{k+l},\\
[h_l,f_k]&= \frac{-[2l]}{l}c^{\lvert l\rvert/2}*f_{k+l}.\end{align*}
\end{prop}

\begin{proof}
By \cite[Theorem 4.12]{Sz06}, Proposition~\ref{prop:satequiv} and Lemma~\ref{lem:m0annoy}, 
we know that for $n\geq 1$ and $m\geq 0$, we have the following identity in the untwisted Hall algebra: 
\begin{equation*} 
R_n[P_m] = q^n[P_m]R_n + \sum_{i=1}^n (q^{n+i}-q^{n+i-2})[P_{m+i}]R_{n-i}.
\end{equation*}
In terms of the twisted product, this equality is equivalent to
\[R_n*[P_m] = [P_m]*R_n + (1-q^{-2})\sum_{i=1}^n \vv^iq^i[P_{m+i}]*R_{n-i}.\]
Recalling the definitions of $\psi_n$ and $e_m$, for all $m,n\geq 0$, the previous equality can be restated as
\begin{equation}\label{eq:psie}
\psi_{n-1}*e_m = qe_m*\psi_{n-1} + (1-q^{-2})\sum_{i=1}^n q^{2i+1}c^{-i/2}*e_{m+i}*\psi_{n-i-1}.
\end{equation}
Define
\begin{align} 
P(t)&:=\kappa^{-1}*\sum_{n\geq 0} \psi_{n-1}t^n, \label{eq:ptdef}\\
E(s)&:=\sum_{m\geq 0} e_ms^m,\\
F(s,t)&:=1+(1-q^{-2})\sum_{i\geq 1}q^ic^{-i/2} t^is^{-i} = \frac{1-q^{-1}c^{-1/2}ts^{-1}}{1-qc^{-1/2}ts^{-1}}.
\end{align}
Note that $F(s,t)$ commutes with both $P(t)$ and $E(s)$. Then we can rewrite identity \eqref{eq:psie} 
to get the following equality for all $m,n\geq 0$:
\[[t^ms^n]P(t)*E(s) = [t^ms^n]E(s)*P(t)*F(s,t),\]
where the notation $[t^ms^n]$ refers to the coefficient of $t^ms^n$ in the following power series.
By induction, it follows that for all $k\geq 1$, we have
\[[t^ms^n]P(t)^k*E(s) = [t^ms^n]E(s)*P(t)^k*F(s,t)^k.\]
This implies that
\begin{align*}
[t^ms^n]\log(P(t))*E(s)&=[t^ms^n]\sum_{k\geq 1}\frac{(-1)^k}{k}P(t)^k*E(s)\\
&=[t^ms^n] \sum_{k\geq 1}\frac{(-1)^k}{k}E(s)*P(t)^k*F(s,t)^k \\
&=[t^ms^n] E(s)*\log(P(t)*F(s,t))\\
&=[t^ms^n] E(s)*\big(\log(P(t))+\log(F(s,t))\big)\\
&=[t^ms^n]E(s)*\log(P(t))+ [t^ms^n]E(s)*\sum_{k\geq1}\frac{q^k-q^{-k}}{k}c^{-k/2}t^ks^{-k}.
\end{align*}
However, by definition, we know that $\log(P(t))=(\vv-\vv^{-1})\sum_{k\geq 1}h_l t^l$. Therefore, for all $l\geq 1$ and $k\geq 0,$ comparing the coefficients of $t^ls^k$ on both sides of the above computation gives us that
\[[h_l,e_k]=\frac{q^l-q^{-l}}{l(\vv-\vv^{-1})}c^{-l/2}*e_{k+l}=\frac{[2l]}{l}c^{- l/2}*e_{k+l}.\]
Next, if $l\neq 1$, we take the commutator of both sides of the above equality by $h_{-1}$. As $h_{-1}$ commutes with $h_l$ whenever $l>1$ (by Proposition~\ref{prop:rel9'}), we get by Lemma~\ref{lem:base_h} that
\[[h_l,e_{k-1}]=\frac{[2l]}{l}c^{- l/2}*e_{k+l-1}.\]
Repeating this, we can shift $k$ further and conclude that the required identity is true for all $k\in \ZZ$ assuming $l\geq 1$. (When $l=1$, the identity was proved in Lemma \ref{lem:base_h}.)

Next, in order to prove the required identity for the $f_k$'s, we apply $\tau\circ\sigma$ to the identity for the $e_k$'s. From their definitions, we note that $\tau\circ\sigma$ preserves the $\psi_d$'s and the $\varphi_d$'s, which implies that $\tau\circ\sigma(h_l) = h_l$ for all $l$. Further, we note that $\tau\circ\sigma(e_k) = \vv f_k$ whereas $\tau\circ\sigma(c)=c^{-1}$.
Applying this to the above equality, we get
\[[h_l,f_k]=-\frac{[2l]}{l}c^{l/2}*f_{k+l}.\]

Finally, we need to deal with the case when $l\leq -1$. Note that
\[\sigma(\varphi_d) = \kappa^{-2d} *\psi_{-d},\]
and it follows that for all $l\geq 1$:
\[\sigma(h_l) = - \kappa^{-2l} *h_{-l}.\]
Next, we see that
\begin{align*}
    \sigma(e_k) &= -\vv^{-2k+1}s'^{-4k}*c^{-k}*f_{-k}\\
    \sigma(f_k)  &= -\vv^{2k-1}s^{-4k}*c^{k}*e_{-k}.
\end{align*}
Then, applying $\sigma$ to the identities for $l\geq 1$ exactly gives us the required identities for $l\leq -1$, completing the proof.
\end{proof}

\begin{prop} \label{prop:rel9'}
For all $n\geq 1$, we have the following identities:
\[h_n*h_{-1} - h_n*h_{-1}= \delta_{n,1}[2]\frac{c-c^{-1}}{\vv-\vv^{-1}},\]
\[h_1*h_{-n} - h_{-n}*h_{1}= \delta_{n,1}[2]\frac{c-c^{-1}}{\vv-\vv^{-1}}.\]
\end{prop}

\begin{proof}
We will only prove the first equality, since the second follows from it by applying $\tau$. 
Recall that
\[1+\psi(u):=\sum_{k\geq 0}\kappa^{-1}\psi_{k-1} u^k = \mathrm{exp}\left( (\vv-\vv^{-1}) \sum_{k=1}^\infty h_k u^k \right).\]
Using Proposition~\ref{prop:rel6} and Lemma~\ref{lem:base_h}, we can compute the commutator for $k\geq 1$
\begin{align*}
[h_{-1},\psi_{k-1}] &= (\vv-\vv^{-1})c^{(k-3)/2}[h_{-1},[e_1,f_{k-2}]]\\
&=(\vv-\vv^{-1})c^{(k-3)/2}\left([[h_{-1},e_1],f_{k-2}] + [e_1,[h_{-1},f_{k-2}]]\right)\\
&=(\vv-\vv^{-1})c^{(k-3)/2}[2]\left(c^{-1/2}[e_0,f_{k-2}] - c^{1/2}[e_1,f_{k-3}]\right)\\
&= [2](c^{-1}-c)\psi_{k-2}.
\end{align*}
This implies that
\[[h_{-1},\psi(u)] = [2](c^{-1}-c)u (1+\psi(u)).\]
Note that the right hand side of the above equality commutes with $\psi(u)$.

In any associative algebra with elements $A$ and $B$, if $[A,B]$ commutes with $B$, then it is clear that
\[[A,f(B)] = [A,B]\cdot \frac{\partial}{\partial B}f(B),\]
where $f$ is a polynomial (or power series) function. Using this fact with $A=h_{-1}$ and $B=\psi(u)$, we see that
\begin{align*}
\left[h_{-1}, (\vv-\vv^{-1}) \sum_{k=1}^\infty h_k u^k\right] &= \left[h_{-1}, \log(1+\psi(u))\right]\\
&=\left[h_{-1},\psi(u)\right]\frac{\partial}{\partial \psi(u)}(\log(1+\psi(u))) \\
&=[2](c^{-1}-c)u (1+\psi(u))\frac{1}{1+\psi(u)}\\
&=[2](c^{-1}-c)u.
\end{align*}
In particular, the coefficient of $u^n$ in the RHS is zero when $n\geq 2$, showing that $[h_{-1},h_n]=0$ for $n\geq 2$. Furthermore, comparing the coefficients of $u$ on both sides gives precisely the claimed equality when $n=1$.
\end{proof}

\begin{prop} \label{prop:rel9}
For all $k,l\in\ZZ\setminus\{0\}$, we have the following identity:
\[h_l*h_k - h_k*h_l= \delta_{l,-k}\frac{[2l]}{l}\frac{c^l-c^{-l}}{\vv-\vv^{-1}}.\]
\end{prop}

\begin{proof}
If $k,l>0$ or $k,l<0$, it is clear that $h_k$ and $h_l$ commute, since the regular part of the Hall algebra of the Kronecker quiver is commutative (see \cite[Sec.~3]{Sz06}). Hence, without loss of generality, we suppose that $l>0$ and $k<0$. The proof will be by induction on $\min\{|k|, l\}$. The base case was resolved by the previous proposition.

In general, without loss of generality, we suppose $|k|\geq l>1$. Note that, by definition, we have that:
\[h_l=\frac{1}{\vv-\vv^{-1}}\kappa^{-1}\psi_{l-1}+P,\]
where $P$ is some polynomial in the $\psi_n$'s for $n<l-1$. Furthermore, each $\psi_n$ is a polynomial in the $h_m$'s with $1\leq m\leq n$. Therefore, by induction, we can assume that $[P,h_k]=0$.
Then, we can compute:
\begin{align*}
[h_{l},h_{k}] &= \frac{1}{(\vv-\vv^{-1})}\kappa^{-1}[\psi_{l-1},h_{k}]\\
&=\kappa^{-1}c^{(l-1)/2}[[e_0,f_{l-1}],h_{k}]\\
&=\kappa^{-1}c^{(l-1)/2}\Big(-\frac{[2k]}{k}c^{k/2}[e_{k},f_{l-1}] + \frac{[2k]}{k} c^{-k/2}[e_0, f_{k+l-1}]\Big) \text{ (By Proposition}~\ref{prop:rel78})\\
&=\kappa^{-1}c^{(l-1)/2}\frac{[2k]}{k}\frac{-c^{(2k-l+1)/2}\psi_{k+l-1}+c^{(l-1)/2}\varphi_{k+l-1} + c^{-(2k+l-1)/2}\psi_{k+l-1}-c^{(l-1)/2}\varphi_{k+l-1}}{\vv-\vv^{-1}}\\
&\phantom{ = }\text{(By Proposition}~\ref{prop:rel6})\\
&=\kappa^{-1}\frac{[2k]}{k}\frac{(-c^{k} + c^{-k})\psi_{k+l-1}}{\vv-\vv^{-1}}\\
&=\delta_{l,-k}\frac{[2l]}{l}\frac{c^{l} - c^{-l}}{\vv-\vv^{-1}},
\end{align*}
where the last equality follows since $\psi_n=0$ for $n<-1$ and $\psi_{-1}=\kappa$.
\end{proof}

\section{Regular representations and Heisenberg subalgebras} \label{sect:regular}
In the Kronecker quiver the regular representations have the following property: the subalgebra of the Hall algebra $Hall(Q_K)$ generated by elements of the form $[R_{\phi}(n)]$, for fixed $\phi\in\Sigma$ and varying $n \geq 1$, is canonically isomorphic to the classical Hall algebra studied by Hall and Steinitz. This classical Hall algebra is in turn isomorphic to the ring $Sym$ 
of symmetric functions in infinitely many variables. 

In this section (see Corollary~\ref{cor:Heis}), we prove that a similar phenomenon occurs for the Rudakov quiver: for fixed $\phi\in \Sigma$ and varying $n \geq 1$, the elements $[R_{\phi}(n)]$ and $[R_{{\phi}}(n)']$ canonically generate a (quantum) Heisenberg algebra inside $Hall(\qrud)$. Furthermore, we observe that up to isomorphism, this algebra only depends on the positive integer $d$ such that $\phi\in\Sigma_d \subset \Sigma$, and not on $\phi$ itself. Finally, as a corollary of the results of this section, we also give an independent Hall theoretic proof of Proposition~\ref{prop:rel9} (see Proposition~\ref{prop:rel9''}).



By Corollary~\ref{lem:regcomm}, we know that regular representations $R_{\phi}(m)$ and $R_{\pi}(n)'$ commute in the Hall algebra, as long as $\phi\neq\pi$. Henceforth, we fix an element $\phi\in \Sigma_d$. Define the sequences
\[J_m:= R_{\phi}(m),\qquad  J_n':=R_{{\phi}}(n)'\]
for all $m,n\geq 1$. As a convention, we take $J_0$ and $J_0'$ to be the zero representation. 


\begin{defn}\label{def:hallphi}
We define $Hall_{\phi}$ to be the subalgebra of $Hall(\qrud)$ generated by the elements $[J_m]$ and $[J_n']$ for $m,n \geq 1$.
\end{defn}

To determine the structure of the algebra $Hall_{\phi}$, we start with the following lemma which describes all extensions of the representation $J_m$ by $J_n'$.

\begin{lem} \label{lem:ext_regular}
Any extension of $J_m$ by $J_n'$ is of the form $J_{m-k}\oplus J_{n-k}'\oplus M^k$ for some $0\leq k \leq \min(m,n)$. 
\end{lem}
\begin{proof}
We break down the proof into six steps as follows:

\textbf{Step 1:} We first show that any pre-projective module $P_t'$ or any pre-injective module $I_t$ can not occur as a direct summand of any extension of $J_m$ by $J_n'$. Suppose $I_t$ is a summand of an extension of $J_m$ by $J_n'$, that is, there is an exact sequence
\[0\to J_n' \xrightarrow{\Theta} I_t\oplus V\to J_m\to0,\]
for some representation $V$. As $J_n'$ is $(e',f')$-injective, so is $\Theta(J_n')$. But $e'$ and $f'$ are both zero on $I_t\oplus \{0\}\subseteq I_t\oplus V$, and so we conclude that $\Theta(J_n') \cap (I_t\oplus \{0\})=0$. This implies that we have an injective homomorphism
\[I_t \cong I_t\oplus \{0\} \to (I_t\oplus V)/\Theta(J_n') \cong J_m.\]
However, this isn't possible since $\Hom(I_t,J_m)=0$ by \cite[Lemma 1.1(a)]{Sz06}. This shows that $I_t$ can never occur as a summand of an extension of $J_m$ by $J_n'$. Applying $\tau$, we conclude that pre-projective modules $P_t'$ can't occur as such a direct summand either.

\textbf{Step 2:} 
We claim that any pre-projective module $P_t$ or any pre-injective module $I_t'$ can not occur as a direct summand of any extension of $J_m$ by $J_n'$. This is a consequence of the previous step using slope arguments: 
as $J_m$ and $J_n'$ are both representations with zero slope, so are all of their extensions. Thus, if any such extension has either $P_t$ or $I_t'$ (both of which are modules with positive slope) as a direct summand, it must also have a direct summand with negative slope, which is not possible because of the previous step. We conclude that $P_t$ and $I_t'$ can't occur as direct summands of the extensions either.

\textbf{Step 3:} We show that the projective module $M'$ can not occur as a direct summand of any extension of $J_m$ by $J_n'$. Suppose $M'$ is a summand of an extension of $J_m$ by $J_n'$, that is, there is an exact sequence
\[0\to J_n' \xrightarrow{\Theta} M'\oplus V\to J_m\to0,\]
for some representation $V$. As in Step 1, we try to compute $\Theta(J_n') \cap (M'\oplus \{0\})$ in $M'\oplus V$. As $e=f=0$ in $J_n'$, we conclude that $\Theta(J_n') \cap (M'\oplus \{0\}) \subseteq (\langle y_2\rangle\oplus \{0\})$ (using notation from equation \eqref{eq:mpdef} for a basis of $M'$). But then, as $J_n'$ is $(e',f')$-injective, we conclude that $\Theta(J_n') \cap (M'\oplus \{0\})=0$. Hence, we must have an injective homomorphism
\[M'\cong M'\oplus\{0\} \to (M'\oplus V)/\Theta(J_n') \cong J_m.\]
This isn't possible since $e'=f'=0$ in $J_m$, whereas both of these maps are non-zero in $M'$. Hence, we conclude that $M'$ can not occur as a summand of an extension of $J_m$ by $J_n'$.

\textbf{Step 4:} We show that any regular representation other than $J_k$ and $J_k'$ can't occur as a direct summand of an extension of $J_m$ by $J_n'$. Let $R_{\pi}(k)$ be a regular representation, where $\pi\neq \phi$. Suppose we have an exact sequence:
\[0\to J_n' \xrightarrow{\Theta} R_{\pi}(k) \oplus V\to J_m\to0,\]
for some representation $V$. By exactly the same argument as the one used in Step 1, we conclude that $\Theta(J_n')\cap (R_{\pi}(k)\oplus \{0\})$ is trivial. This implies that $R_{\pi}(k)$ is a submodule of $J_m$. However, by \cite[Lemma 1.1(b)]{Sz06}, we have that $\Hom(R_{\pi}(k),J_m)=0$. This gives us the required contradiction, showing that $R_{\pi}(k)$ can not occur as a direct summand of an extension of $J_m$ by $J_n'$ if $\pi\neq\phi$. Applying $\tau$, we conclude that $R_{\pi}(k)'$ can not occur as a direct summand either if $\pi\neq {\phi}$.

\textbf{Step 5:} We now show that $J_a\oplus J_b$ and $J_a'\oplus J_b'$ can not occur as a direct summand of an extension of $J_m$ by $J_n'$ for some $a,b\geq 1$. (That is, any extension contains at most one of the $J_a$'s and at most one of the $J_b'$ as a direct summand.) Suppose we have a short exact sequence
\[0\to J_n' \xrightarrow{\Theta} J_a\oplus J_b \oplus V\to J_m\to0,\]
for some representation $V$. By the same argument as used in Steps 1 and 4, we can show that this sequence implies that $J_a\oplus J_b$ is a submodule of $J_m$. This implies that $J_1\oplus J_1$ is a submodule of $J_m$, which is not possible since $J_m$ has a unique submodule isomorphic to $J_1$ (\cite[Lemma 1.1(g)]{Sz06})  
This is thus a contradiction, and so $J_a\oplus J_b$ can not occur as a direct summand of such an extension. Applying $\sigma$, we conclude that $J_a'\oplus J_b'$ can't occur as a direct summand either.

\textbf{Step 6:} The five steps above imply that any extension of $J_m$ by $J_n'$ is of the form $J_{m-a}\oplus J_{n-b}'\oplus M'^c$ for some $a,b,c$. To complete the proof of the lemma, we need to show that $da=db=c$, where $\phi\in\Sigma_d$. Comparing dimensions, we conclude that: $dm+dn = d(m-a)+d(n-b)+2c$ which is equivalent to $da+db=2c$. For what follows, we assume that $\phi\neq 0$, without loss of generality. (If $\phi=0$, the same argument as below works after swapping the roles of $e$ and $e'$ with those of $f$ and $f'$ respectively.) 

Next, we observe that $\dim(\text{Im}(e))=dm$ in $J_m$, whereas $\dim(\text{Im}(e))=0$ in $J_n'$. Thus, in any extension $V$ of $J_m$ by $J_n'$, we must have $\dim(\text{Im}(e))\geq dm$. On the other hand, in $J_{m-a}\oplus J_{n-b}'\oplus M'^c$, we have $\dim(\text{Im}(e))=d(m-a)+0+c=dm-da+c.$ Thus, we get that $dm-da+c\geq dm$ which implies that $c\geq da$. Similarly, comparing $\dim(\text{Im}(e'))$, we get the inequality $c\geq db$. These inequalities along with the equality $da+db=2c$ imply that $da=db=c$, completing the proof of the lemma.
\end{proof}




As a consequence of the above lemma, we can express the Hall product as
\[[J_m][J_n'] = \sum_{k=0}^{\min\{m,n\}}c_k[J_{m-k}\oplus J_{n-k}'\oplus M^{dk}],\]
for some coefficients $c_k$. The following lemma determines these coefficients, where for $a \geq 0$ we define
\[\rho_{a}:=\begin{cases}
1& \text{if } a\neq0\\
(1-q^{-d})^{-1} &\text{if } a=0
\end{cases}\\.\]
\begin{lem}\label{lem:ckconst}
In the above Hall product, we have $c_0=q^{d^2mn}$. For $k\geq 1$, we have the following equality:
\[c_k={q^{d^2mn-d^2k^2-dk}}\big\lvert GL_{dk}(\FF_q)\big\rvert \rho_{m-k}\rho_{n-k}(1-q^{-d}).\]
\end{lem}
\begin{proof}
The coefficient $c_k$ is equal to the number of subrepresentations $V\subseteq J_{m-k}\oplus J_{n-k}'\oplus M^{dk}$ such that $V\cong J_n'$ and $(J_{m-k}\oplus J_{n-k}'\oplus M^{dk})/V\cong J_m$. When $k=0$, any such subrepresentation $V$ is the image of a map of the form:
\[J_n'\xrightarrow{f\oplus Id} J_m\oplus J_n',\]
for any $f\in \Hom(J_n',J_m)$. The space $\Hom(J_n',J_m)$ has dimension $d^2mn$ by Lemma~\ref{lem:sathom}, which gives that $c_0=q^{d^2mn}$.

Now, we assume that $k\geq 1$. Furthermore, we assume that $\phi\neq 0$, without loss of generality. (If $\phi=0$, the same argument as below works after swapping the roles of $e$ and $e'$ with those of $f$ and $f'$ respectively.) We first fix some coordinates. On $M^{dk}$, we fix a basis $\{x_i,y_i,z_i,w_i\}$ for $1\leq i\leq dk$, where each such quadruple spans a copy of $M$ and we have 
\[\begin{tikzcd}
	{x_i} && {z_i} \\
	\\
	{w_i} && {y_i}
	\arrow["f"{pos=0.5}, from=1-1, to=1-3]
	\arrow["e"{pos=0.15}, from=1-1, to=3-3, swap]
	\arrow["{e'}"{pos=0.1}, from=1-3, to=3-1]
	\arrow["{-f'}"{pos=0.5}, from=3-3, to=3-1]
\end{tikzcd}.\]
Next, on $J_{n-k}'$, we fix bases $\{r_1,r_2,\cdots,r_{d(n-k)}\}$ at vertex $1$ and $\{s_1,s_2,\cdots,s_{d(n-k)}\}$ at vertex $2$, so that the matrix for the map $e'$ with respect to this basis is the identity, whereas that of $f'$ is a block matrix with $d\times d$ blocks and is equal to the $d(n-k)\times d(n-k)$ matrix
\[M_{{\phi}}(n-k):=\begin{bmatrix}
    M_{{\phi}}       & I_{d\times d} &  &  & \\
           & M_{{\phi}} & I_{d\times d} &  & \\
        & & .&.  & \\
         & & &.  &. \\
      &  & &  & M_{{\phi}}
\end{bmatrix},
\]
where $I_{d\times d}$ is the identity matrix and all the empty blocks are zero. Similarly for $J_{m-k}$, fix bases $\{t_1,t_2,\cdots,t_{d(m-k)}\}$ at vertex $1$ and $\{u_1,u_2,\cdots,u_{d(m-k)}\}$ at vertex $2$, so that the matrix for the map $e$ with respect to this basis is the identity, whereas that of $f$ is a block matrix with $d\times d$ blocks and is equal to $M_{\phi}(m-k)$.

With the above convention, a basis for $J_{m-k}\oplus J_{n-k}'\oplus M^{dk}$ at the vertex $2$ is given by the set $\{y_i,z_i,s_j,u_l\}$ for $1\leq i\leq dk$,and  $1\leq j\leq d(n-k)$, and $1\leq l\leq d(m-k)$. On this basis, $e'$ acts by zero on the $y_i$'s and the $u_i$'s.
On the other hand, if $V\subseteq J_{m-k}\oplus J_{n-k}'\oplus M^{dk}$ such that $V\cong J_n'$, then the dimension of $\text{Im}(e')$ in $V$ is exactly $dn$. This implies that for $V$, we can write a basis at the vertex $2$ of the form
\[\{z_i + a_i + b_i, s_j + \widetilde{a_j} +\widetilde{b_j},\}\]
for $1\leq i\leq dk$ and $1\leq j\leq d(n-k)$, where $a_i,\widetilde{a_j}$ lie in the span of the $y_i$'s and $b_i,\widetilde{b_j}$ lie in the span of the $u_j$'s. The choice of these $a_i$'s, $\widetilde{a_j}$'s, $b_i$'s and $\widetilde{b_j}$'s uniquely determines a subspace at vertex $2$ and we'd like to find the number of ways of choosing these vectors so as to ensure that the spanned subrepresentation is isomorphic to $J_n'$ and the quotient is isomorphic to $J_m$.

With respect to this basis at vertex $2$ and the set $\{w_i,r_j\}$ at the vertex $1$, we can write the matrices for $e'$ and $f'$ in $V$. The matrix for $e'$ is the identity, whereas that for $f'$ is a $dn\times dn$ matrix of the form
\[\begin{bmatrix}
    M_1       & M_2 \\
     0  & M_{\phi}(n-k)
\end{bmatrix},\]
where $M_1$ is a $dk\times dk$ matrix determined by the $a_i$'s and $M_2$ is a $dk\times d(n-k)$ matrix determined by the $\widetilde{a_j}$'s. Furthermore, given any matrices $M_1$ and $M_2$ of the above dimension, we can choose the $a_i$'s and the $\widetilde{a_j}$'s uniquely so that $f'$ has the above form.

As $V\cong J_n'$, we must have that the above matrix is similar to $M_{\phi}(n)$. For this to be the case, we must have that $M_1$ is similar to $M_{\phi}(k)$. The number of such matrices $M_1$ is given by $|GL_{dk}(\FF_q)|/|Z(M_{\phi}(k))|$, where $Z(M_\phi(k))$ is the centralizer of $M_{\phi}(k)$ in $GL_{dk}(\FF_q)$. By \cite[Chapter II, (1.6)]{Mac}, we have that $|Z(M_{\phi}(k))|=q^{dk}(1-q^{-d})$. Hence, the number of such matrices $M_1$, which is equal to the number of valid choices for the vector $a_i$'s, is equal to
\[\frac{|GL_{dk}(\FF_q)|}{q^{dk}(1-q^{-d})}.\]

Furthermore, if $M_1$ is similar to $M_{\phi}(k)$, in order to have the whole matrix above to be similar to $M_{\phi}(n)$, by Lemma~\ref{lem:Jordan}, we have that there are $q^{d^2k(n-k)}\rho_{n-k}(1-q^{-d})$ choices for the matrix $M_2$, and thus, this is the number of valid choices for the vectors $\widetilde{a_j}$'s.

Next, we need to figure out the number of valid choices for the $b_i$'s and $\widetilde{b_j}$'s so that $(J_{m-k} \oplus J'_{n-k} \oplus M^{dk})/V\cong J_m$. For this quotient, we can choose as basis the set $\{x_i, t_j\}$ at vertex $1$ and $\{y_i, u_j\}$ at vertex $2$, where $1\leq i\leq dk$ and $1\leq j \leq d(m-k)$. With respect to these basis, the matrix for $e$ is the identity matrix, whereas that for $f$ is a $dm\times dm$ matrix of the form
\[\begin{bmatrix}
    M_1       & 0 \\
     M_3  & M_{{\phi}}(m-k)
\end{bmatrix},\]
where the matrix $M_3$ is a $d(m-k)\times dk$ matrix that is determined by the $b_i$'s. The matrix $M_1$ was already chosen to be similar to $M_{{\phi}}(k)$, and so, by Lemma~\ref{lem:Jordan} again, we get that we have $q^{d^2k(m-k)}\rho_{m-k}(1-q^{-d})$ choices for the matrix $M_3$, and thus for the vectors $b_i$, so that the above $dn\times dn$ matrix is similar to $M_{\phi}(m)$.

Finally, the $\widetilde{b_j}$'s are $d(n-k)$ arbitrary vectors in a $d(m-k)$ dimensional vector space, and hence, can be chosen in $q^{d^2(m-k)(n-k)}$ ways. So, multiplying all the choices above, we get the Hall number
\begin{align*}c_k &= \frac{|GL_{dk}(\FF_q)|}{q^{dk}(1-q^{-d})}\cdot q^{d^2k(n-k)}\rho_{k,n-k}(1-q^{-d})\cdot q^{d^2k(m-k)}\rho_{k,m-k}(1-q^{-d})\cdot q^{d^2(m-k)(n-k)}\\
&={q^{d^2mn-d^2k^2-dk}}|GL_{dk}(\FF_q)|\rho_{m-k}\rho_{n-k}(1-q^{-d}).
\end{align*}
\end{proof}

\begin{lem} \label{lem:Jordan}
Fix non-negative integers $a,b$ and a monic irreducible polynomial $\pi$ of degree $d$ over $\FF_q$. Consider the block matrix:
\[A=\begin{bmatrix}
    M_{\pi}(a)   & \overline{M} \\
     0  & M_{\pi}(b)
\end{bmatrix},\]
where $\overline{M}$ is some $da\times db$ matrix. Then, the number of matrices $\overline{M}$ such that the matrix $A$ above is similar to $M_{\pi}(a+b)$ is exactly $q^{d^2ab}\rho_{ab}(1-q^{-d})$.
\end{lem}

\begin{proof}

When either $a$ or $b$ is zero, the claim is clear. Henceforth, we assume that $a$ and $b$ are both positive. We can view $\overline{M} = (M_{i,j})$ as a block matrix with $d\times d$ blocks $M_{i,j}$ with $1\leq i \leq a$ and $1\leq j\leq b$. For the purpose of this proof, let $X=M_{\pi}(1)$. 
Then, the matrix $A$ can be written in a $d\times d$ block form, such that all the diagonal blocks are equal to $X$. Furthermore, all but (possibly) one of the superdiagonal blocks are equal to the identity $I$, whereas the $a^{th}$ superdiagonal block is equal to $M_{a,1}$.

Suppose $\pi(t)=\sum\limits_{i=0}^d c_i t^i$ for some $c_i\in\FF_q$. We can view $\pi$ as a polynomial function over the space of square matrices (also denoted by $\pi$). Furthermore, the derivative of $\pi$ on the space of matrices at the point $X$ is a linear map $\pi_X'$ given by
\[\pi_X'(Y) = \sum_{i=0}^dc_i\sum_{j=0}^{i-1}X^jYX^{i-j-1}.\]

Then, the evaluation of $\pi$ on $A$ is given by
\[\pi(A) = \begin{bmatrix}
0 & \pi_X'(I) & * & * & *  & *& \cdots & *\\
0 & 0 & \pi_X'(I) & * & * & *& \cdots & *\\
\vdots&\vdots&\vdots&\vdots&\vdots&\vdots&\vdots&\vdots\\
0 & 0 & \cdots & 0 & \pi_X'(M_{a,1})&* & \cdots & *\\
0 & 0 & \cdots & 0 & 0&\pi_X'(I) & \cdots & *\\
\vdots&\vdots&\vdots&\vdots&\vdots&\vdots&\vdots\\
0 & 0 & \cdots & 0 & 0&0& \cdots & 0\\
\end{bmatrix}.\]
Here, the diagonal blocks are zero because $\pi(X)=0$ by Cayley-Hamilton theorem. Note that all but (possibly) one of the superdiagonal blocks of $\pi(A)$ are equal to $\pi_X'(I)$, which is equal to the derivative of $\pi(t)$ (as a usual polynomial) applied to $X$. Hence, $\pi_X'(I)$ is invertible since the minimal polynomial of $X$ is $\pi(t)$, which is separable.

In order to show that $A$ is similar to $M_{\pi}(a+b)$, we need to show that the minimal polynomial of $A$ is $\pi(t)^{a+b}$. That is equivalent to showing that $\pi(A)$ is nilpotent of order $a+b$, but of no smaller order. By the above computation, and using the fact that $\pi_X'(I)$ is invertible, we get that $\pi(A)$ is nilpotent of order $a+b$ if and only if $\pi_X'(M_{a,1})\neq 0$. So, the claim of the lemma is equivalent to showing that the kernel of the linear map $\pi_X'$ on the space of $d\times d$ matrices has dimension exactly $d^2-d$.

That the final assertion is true can be seen by going to a field extension of $\FF_q$. (We can do this since the rank of a linear map is preserved under field extension.) So, without loss of generality, we can suppose that $\pi(t)$ splits into linear factors and that $X$ is diagonal. Then, it is an easy computation to show that $\pi_X'(Y)$ is zero precisely when all the diagonal entries of $Y$ are zero.
\end{proof}


\begin{lem}\label{lem:jj}
We have the following identity:
\begin{equation} \label{eq:old}
[J_m][J_n'] = q^{d^2mn}[J_m\oplus J_n'] + \sum_{k=1}^{\min\{m,n\}}q^{d^2(m-k)(n-k)-dk}(q-1)^{dk}(1-q^{-d})[M]^{dk}\rho_{m-k}\rho_{n-k}[J_{m-k}\oplus J_{n-k}'].
\end{equation}
\end{lem}
\begin{proof}
Note that any automorphism of the representation $M$ is determined by the image of $x_1$, and all the available options are given by $ax_1+bx_2$ where $a\in \FF_q^{\times}$ and $b\in \FF_q$. Thus, we see that $|\mathrm{Aut}(M)| = q(q-1)$. By a similar argument, we have that $|\mathrm{Aut}(M^n)| = \lvert M_n(\FF_q)\rvert \lvert GL_n(\FF_q)\rvert=q^{n^2}|GL_n(\FF_q)|$ for all $n\geq 1$. 
Then, by a repeated application of Lemma~\ref{lem:sumcoeff} and Corollary~\ref{cor:mhom}, we have the following equality:
\[[M]^{dk} = \prod\limits_{i=1}^{dk-1}|\Hom(M,M^i)|^{-1}\frac{|\mathrm{Aut}(M^{dk})|}{|\mathrm{Aut}(M)|^{dk}}[M^{dk}] = |GL_{dk}(\FF_q)|(q-1)^{-dk}[M^{dk}].\]
Then, again by Lemma~\ref{lem:sumcoeff} and Corollary~\ref{cor:mhom}, we get
\begin{align*}
[M]^{dk}[J_{m-k}\oplus J_{n-k}'] &= |GL_{dk}(\FF_q)|(q-1)^{-dk}|\Hom(J_{m-k}\oplus J_{n-k}', M^{dk})|[J_{m-k}\oplus J_{n-k}'\oplus M^{dk}]\\
&=|GL_{dk}(\FF_q)|(q-1)^{-dk}q^{d^2k(m+n-2k)}[J_{m-k}\oplus J_{n-k}'\oplus M^{dk}].
\end{align*}
Combining this with Lemma \ref{lem:ckconst}  proves the claim.
\end{proof}

We introduce some shorthand notation that will be useful in upcoming computations and will make the above equality \eqref{eq:old} look cleaner: 

\[
\jj_m :=[J_m]\rho_m,\qquad \jj_n':=[J_n']\rho_n,
\]
\[    A_{m,n} :=q^{d^2mn}\rho_m\rho_n[J_m\oplus J'_n],\qquad \xx:=\left(\frac{q-1}{q}[M]\right)^d,\qquad \yy:=\left(\frac{q-1}{q}[M']\right)^d
\]
where, by convention, $A_{r,s} = 0$ if $r<0$ or $s<0$.
\begin{remark}
Note that when $m,n>0$, we have $\jj_m=[J_m], \jj_n'=[J_n']$ and $A_{m,n}=q^{d^2mn}[J_m\oplus J_n']$.
\end{remark}

In terms of this notation, equation~\eqref{eq:old} becomes
\begin{equation} \label{eq:new}
\jj_m\jj_n' = (1-q^{-d})\sum_{k=0}^{\min\{m,n\}}\rho_k\xx^{k}A_{m-k,n-k}.
\end{equation}


So, we note that the product $\jj_m\jj_n'$ can be expressed as a sum involving the elements $A_{m-k,n-k}$ for $k\geq 0$. The following lemma does the inverse job:

\begin{lem}\label{lem:amn}
For all $m,n\geq 0$, we have the following:
\begin{align*}
A_{m,n} &= (1-q^d)\sum_{k\geq 0}\rho_kq^{-dk}\xx^k\jj_{m-k}\jj_{n-k}',\\
A_{m,n} &= (1-q^d)\sum_{k\geq 0}\rho_kq^{-dk}\yy^k\jj_{n-k}'\jj_{m-k}.\end{align*}
\end{lem}

The proof is by an induction on $m+n$ using Equation~\eqref{eq:new}. Combining the results of the two lemmas above, we conclude with the main result of this section, which again follows by a straightforward induction:
\begin{thm} \label{lem:mainreg}
For all $m,n\geq 0$, we have:
\[\jj_m\jj_n' = \jj_n'\jj_m + \frac{(q^d-1)(\xx-\yy)}{q^d\xx-\yy}\sum_{k\geq 1} q^{-dk}((q^d\xx)^k-\yy^k)\jj_{n-k}'\jj_{m-k}.\]
\end{thm}



Next, we state a `logarithmic' version of the above theorem, which will allow us to provide a new proof of Proposition~\ref{prop:rel9}. To that end, define the generating series
\begin{equation} \label{eq:edef}
E_{\phi}(t) := 1 + \sum_{m\geq 1} q^{md}(1-q^{-d})[J_m]t^{md} = \sum_{m\geq 0} q^{md}(1-q^{-d})\jj_mt^{md} = \sum_{m\geq 0}E_{m,\phi} t^{md},
\end{equation}
where we define $E_{m,\phi}:=q^{md}(1-q^{-d})\jj_m$ for $m\geq 0$. Similarly, we define
\begin{equation}F_{\phi}(s):= 1+\sum_{n\geq 1} q^{nd}(1-q^{-d})[J_n']s^{nd}= \sum_{n\geq 0} q^{nd}(1-q^{-d})\jj_n's^{nd} = \sum_{n\geq 0}F_{n,\phi} s^{nd},
\end{equation}
where $F_{n,\phi}:=q^{nd}(1-q^{-d})\jj_n'$ for $n\geq 0$. Next, define the sequence $(M_k)$ for $k\geq 1$ as
\[M_k = q^{-dk}(q^d-1)(\xx-\yy)\frac{(q^d\xx)^k-\yy^k}{q^d\xx-\yy}.\]
and define a generating series
\[M(s,t):=1+\sum_{k\geq 1}q^{2dk}M_kt^{dk}s^{dk}.\]

Then, a reformulation of Theorem \ref{lem:mainreg} is given by the the generating series identity 
    \begin{equation}
        E_{\phi}(t)F_{\phi}(s) = M(s,t)F_{\phi}(s)E_{\phi}(t).
    \end{equation}
(Note that $M(s,t)$ is independent of the choice of $\phi$.) We can factorize $M(s,t)$ as follows:
\begin{align*}
M(s,t)&=1+\sum_{k\geq 1}q^{2dk}M_kt^{dk}s^{dk}\\
&=1+\sum_{k\geq 1}q^{dk}(q^d-1)(\xx-\yy)\frac{(q^d\xx)^k-\yy^k}{q^d\xx-\yy}t^{dk}s^{dk}\\
&=\frac{(1-q^{d}t^ds^d\xx)(1-q^{2d}t^ds^d\yy)}{(1-q^{2d}t^ds^d\xx)(1-q^{d}t^ds^d\yy)}.
\end{align*}
Define the formal logarithms of the series $E_{\phi}(t)$ and $F_{\phi}(s)$:
\begin{align} 
C_{\phi}(t)&:=\log(E_{\phi}(t)) = \sum_{m\geq 1} C_{m,\phi}t^{dm},\label{eq:seqfinalfinal}\\
D_{\phi}(s)&:=\log(F_{\phi}(s)) = \sum_{n\geq 1} D_{n,\phi}s^{dn}.\notag
\end{align}

As $M(s,t)$ commutes with both $E_{\phi}(t)$ and $F_{\phi}(s)$, we get by \cite[Corollary 6.20]{CLLSS18} that 
\[M(s,t) = \exp([C_{\phi}(t),D_{\phi}(s)]).\] 
Therefore, we have the following:
\begin{align*}
[C_{\phi}(t),D_{\phi}(s)]&= \log(M(s,t))\\
&=\log\Bigg(\frac{(1-q^{d}t^ds^d\xx)(1-q^{2d}t^ds^d\yy)}{(1-q^{2d}t^ds^d\xx)(1-q^{d}t^ds^d\yy)}\Bigg)\\
&=-\sum_{k\geq 1}\frac{q^{dk}t^{dk}s^{dk}\xx^k}{k}-\sum_{k\geq 1}\frac{q^{2dk}t^{dk}s^{dk}\yy^k}{k}+\sum_{k\geq 1}\frac{q^{2dk}t^{dk}s^{dk}\xx^k}{k}+\sum_{k\geq 1}\frac{q^{dk}t^{dk}s^{dk}\yy^k}{k}\\
&=\sum_{k\geq 1}\frac{q^{dk}t^{dk}s^{dk}}{k}(q^{dk}-1)(\xx^k-\yy^k)\\
&=\sum_{k\geq 1}\frac{(q-1)^{dk}t^{dk}s^{dk}}{k}(q^{dk}-1)([M]^{dk}-[M']^{dk}).
\end{align*}
Comparing coefficients of $t^{dn}s^{dm}$ on both sides of the above equality, we get:

\begin{cor} \label{cor:Heis}
For all $m,n\geq 1$, we have the following identity:
\[[C_{m,\phi},D_{n,\phi}] = \delta_{m,n}\frac{(q-1)^{dn}(q^{dn}-1)}{n}([M]^{dn}-[M']^{dn}).\]
\end{cor}

Note that since the elements $C_{m,\phi}$ and $D_{n,\phi}$ generate the algebra $Hall_{\phi}$, the above corollary (along with the fact that $[M]$ and $[M']$ are central) completely describes the structure of this algebra. 

Before we proceed, we record another set of generators for the algebra $Hall_{\phi}$ and an analog of Theorem~\ref{lem:mainreg} for these generators. Define the sequences
\[\mathcal{P}_m:= q^{dm(m-1)/2}[R_{\phi}(1)^{\oplus m}],\qquad  \mathcal{P}_n':=q^{dn(n-1)/2}[R_{{\phi}}(1)'^{\oplus n}],\]
for all $m,n\geq 0$. The subalgebra of $Hall(\qrud)$ generated by the $\mathcal{P}_m$'s (resp. $\mathcal{P}_n'$'s) is isomorphic to ${Sym}$ and is equal to the algebra generated by the $\jj_m$'s (resp. $\jj_n'$'s). Thus, the elements $\mathcal{P}_m$ and $\mathcal{P}_n'$ also generate the subalgebra $Hall_{\phi}$ of $Hall(\qrud).$ Then, the following theorem is a fun computational exercise, which follows from the above corollary and the explicit isomorphism of the classical Hall algebra with ${Sym}$ (see \cite[Chapter III, (3.4)]{Mac}). 

\begin{thm} \label{thm:newgen}
For all $m,n\geq 1$, the following identity holds in the Hall algebra:
\[\mathcal P_m\mathcal{P}_n' = \sum_{k\geq 0}(\vv-\vv^{-1})^{dk}\mathcal{P}_{n-k}'\mathcal{P}_{m-k}\prod_{t=0}^{k-1}\left(\frac{\vv^{-dt}[M]^d-\vv^{dt}[M']^d}{\vv^{d(t+1)}-\vv^{-d(t+1)}}\right).\]
\end{thm}

\begin{rem}
The curious looking product in the right hand side of the above expression appears via an application of the $q$-binomial theorem involving the quantum integers defined in Section~\ref{sec:quant}.
\end{rem}

Next, we describe a slightly more explicit formulation for the sequence $(h_n)$ from the previous section in terms of the  sequences defined in equation \eqref{eq:seqfinalfinal}.
\begin{lem}\label{lem:hexplicit}
For all $k\geq 1$, we have the following equality:
\[h_k = \frac{\vv^{-k}}{\vv-\vv^{-1}}\sum_{d|k}\sum_{\pi\in\Sigma_d} C_{k/d, \pi}.\]
Similarly, for all $k\leq -1$, we have the following equality:
\[h_k = -\frac{\kappa^{-2k}\vv^k}{\vv-\vv^{-1}}\sum_{d|k}\sum_{\pi\in\Sigma_d} D_{-k/d, \pi}.\]
\end{lem}


\begin{proof}
We only need to prove the lemma when $k\geq 1$, since for $k\leq -1$, the statement follows by applying $\sigma$. The generating function for the sequence $\psi_d$ is as follows:
\[
\sum_{k\geq 0} \psi_{k-1}t^k = (q-1)\kappa*\sum_{k\geq 0}R_{k}\vv^{-k}t^k.\]
Recall that
\begin{align*}
R_d&=\frac{q^d}{q-1}\sum\left(1-\frac{1}{q^{l_1}}\right)\left(1-\frac{1}{q^{l_2}}\right)\cdots \left(1-\frac{1}{q^{l_r}}\right) [R_{\pi_1}(n_1)\oplus R_{\pi_2}(n_2)\oplus\cdots R_{\pi_r}(n_r)]\\
&=\frac{1}{(q-1)}\sum\prod_{k=1}^rq^{l_k(n_k-1)}(q^{l_k}-1)[R_{\pi_k}(n_k)], \text{ (By Lemma}~\ref{lem:rnexpression})
\end{align*}
where the sum is over all finite sets $\{\pi_1, \pi_2,\cdots, \pi_r\}$ of monic irreducible polynomials such that $\deg(\pi_i)=l_i$ and $\sum l_in_i = d$. Then, using the unique factorization for polynomials over $\FF_q$, the  generating function $P(t)$ defined in \eqref{eq:ptdef} for $\psi_d$ can be factorized as
\begin{align*}
\sum_{k\geq 0} \psi_{k-1}t^k &= \kappa*\prod_{\pi\in \Sigma}\left(1+\sum_{k\geq 0}(q^{deg(\pi)(k-1)}(q^{deg(\pi)}-1)[R_{\pi}(k)](\vv^{-1}t)^{k\cdot deg(\pi)}\right)\\
&=\kappa*\prod_{\pi\in \Sigma} E_{\pi}(\vv^{-1}t),
\end{align*}
where for any $\pi\in\Sigma_d$, where the series $E_{\pi}(t)$ was defined in Equation~\eqref{eq:edef}. Recall that the formal logarithm of this series is given by $C_{\pi}(t)$.
So, we have that
\begin{align*}
\sum_{k\geq 1}h_kt^k &= \frac{1}{\vv-\vv^{-1}}\log\Big(\kappa^{-1}*\sum_{k\geq 0} \psi_{k-1}t^k\Big)\\
&= \frac{1}{\vv-\vv^{-1}}\sum_{\pi}\log(E_{\pi}(\vv^{-1}t))\\
&= \frac{1}{\vv-\vv^{-1}}\sum_{\pi}\sum_{k\geq 0}C_{k,\pi}(\vv^{-1}t)^{dk}.
\end{align*}
Hence, for all $k\geq 1$, we have
\[h_k = \frac{\vv^{-k}}{\vv-\vv^{-1}}\sum_{d|k}\sum_{\pi\in\Sigma_d} C_{k/d, \pi}.\]

\end{proof} 


\begin{prop} \label{prop:rel9''}
For all $k,l\in\ZZ\setminus\{0\}$, we have the following identity:
\[h_l*h_k - h_k*h_l= \delta_{l,-k}\frac{[2l]}{l}\frac{c^l-c^{-l}}{\vv-\vv^{-1}}.\]
\end{prop}

\begin{proof}
If $k,l>0$ or $k,l<0$, it is clear that $h_k$ and $h_l$ commute. Hence, without loss of generality, we suppose that $l>0$ and $k<0$. Now, by Lemma~\ref{lem:regcomm}, 
for any $n,n'>0$ and any $\pi,\pi'\in\Sigma$, we have that
\[[C_{n,\pi},D_{n',\pi'}]=0\]
if $\pi\neq {\pi'}$. Using this fact and the expressions for the $h_k$'s derived in Lemma \ref{lem:hexplicit}, we get that  
\begin{align*}
[h_l,h_k]&= -\frac{\vv^{k-l}\kappa^{-2k}}{(\vv-\vv^{-1})^2}\sum_{d|(k,l)} \sum_{\pi\in \Sigma_d}[C_{l/d,\pi}, D_{-k/d,{\pi}}]\\
&=-\frac{\vv^{k-l}\kappa^{-2k}}{(\vv-\vv^{-1})^2}\sum_{d|(k,l)} \sum_{\pi\in \Sigma_d} \delta_{l/d,-k/d}\frac{(q-1)^{l}(q^{l}-1)}{l/d}([M]^{l}-[M']^{l}) \\
&\text{ (By Corollary~\ref{cor:Heis})}\\
&=-\delta_{l,-k}\frac{\vv^{-2l}\kappa^{2l}}{(\vv-\vv^{-1})^2}\frac{(q-1)^{l}(q^{l}-1)}{l}([M]^{l}-[M']^{l})\sum_{d|l} \sum_{\pi\in \Sigma_d}d\\
&=-\delta_{k,-l}\frac{(q^{l}-1)}{lq^{l}(\vv-\vv^{-1})^2}(c^{-l}-c^{l})\sum_{d|l} d|\Sigma_d|\\
&=\delta_{k,-l}\frac{(q^{l}-1)}{lq^{l}(\vv-\vv^{-1})^2}(c^l-c^{-l})(q^{l}+1)\text{ (By Lemma~\ref{lem:finite})}\\
&=\delta_{l,-k}\frac{[2l]}{l}\frac{c^l-c^{-l}}{\vv-\vv^{-1}},
\end{align*}
completing the proof.
\end{proof}

\begin{lem} \label{lem:finite}
For any $l\geq 1$, we have the following equality:
\[\sum_{d|l}d|\Sigma_d| = q^l+1.\]
\end{lem}

\begin{proof}
Recall that for $d>1$, the set $\Sigma_d$ consists of monic irreducible polynomials of degree $d$, whereas $\Sigma_1$ consists of all monic polynomials of degree $1$ and the polynomial $0$. Over $\FF_q$, we have the factorization
\[x^{q^l}-x=\prod_{\phi} \phi(x),\]
where the product runs over all monic irreducible polynomials whose degree divides $l$. Then the lemma follows by comparing degrees on both sides (and adding 1 to the right to account for the polynomial 0).
\end{proof}

\section{The image and kernel } \label{sec:imker}

Much of the literature about (point-counting) Hall algebras of quivers involves the \emph{composition subalgebra}, which is the subalgebra generated by the simple objects. However, our main theorem shows that the image of the map $\theta_{\vv}:\uu_{\vv}(\widehat{\sl_2})_{1,1}\to \hh^{tw}$ 
is significantly larger than the composition subalgebra of $\qrud$. In this subsection we give an abstract characterization of this image which we call the \emph{spherical subalgebra} -- this definition also applies to other categories and may be of independent interest. 


The main theorem of this section is the following:
\begin{thm}\label{thm:injandsph}
    The map $\theta_{\vv}:\uu_{\vv}(\widehat{\sl_2})_{1,1} \to \hh^{tw}$ is injective, and its image is precisely the spherical subalgebra $\hh^\sph \subset \hh^{tw}$ (see Definition \ref{def:maximal}). 
\end{thm}
\begin{proof}
    Injectivity is proved in Proposition \ref{prop:injec} and Corollary \ref{cor:basis}, and the image of $\theta_{\vv}$ is computed in Proposition \ref{prop:image}.
\end{proof}

\subsection{The spherical subalgebra}\label{sec:spherical}
In this section we define the spherical subalgebra and establish some of its basic properties.

\begin{dfn}\label{def:maximal}
    Let $\mathcal C$ be a finitary abelian category. 
    \begin{enumerate} 
    \item For an object $Y$, let $\alpha(Y) = \big\lvert \Aut(Y)\big\rvert$ (as before), and define
    \[ c(Y) :=  \dim\left( \End_\mathcal{C}(Y) \right).\] 
    We will write $X \preceq Y$ if $c(X) \leq c(Y)$.
    \item \cite{BG16} Define a partial order on the set of isomorphism classes of objects in $\mathcal C$ where $N \vartriangleleft M $  if $M \cong N' \oplus N''$ and there is a nonsplit short exact sequence $N' \to N \to N''$. 
    \item     
    If $C$ is a set of objects in $\mathcal C$, we write 
    \[
    \mathrm{min}_\preceq(C) := \{X \in C \mid X \preceq Y \textrm{ for all } Y \in C\}
    \]
    for the set of objects in $C$ that minimize the partial order $\preceq$.
    \item Given an element $\beta \in K_0(\mathcal C)$, write $\mathcal{C}_\beta$ for the set of objects in the class $ \beta$, and define
    \begin{equation}\label{eq:charKzero} 
    \cf_\beta := \sum_{Y \in \min_\preceq (\mathcal{C}_\beta)}  \alpha(Y) [Y].
    \end{equation}
    We will call a term $Y \in \min_\preceq (\mathcal{C}_\beta)$ in this sum a \emph{minimizer}, or a $\preceq$-\emph{minimizer in class $\beta$.}
\item The \emph{spherical subalgebra}  of $\mathcal C$ is the subalgebra of the Hall algebra $Hall(\cc)$ generated by the $\cf_\beta$:
\[ 
Hall^{\sph}(\mathcal C) := \langle \cf_\beta \mid \beta \in K_0(\mathcal C)\rangle \subset Hall(\mathcal C).
\]
\item Let $\hh^\sph \subset \hh^{tw}$ be the localization of $Hall^\sph(\qrud)$ inside the localization $\hh$. Note that since $Hall^\sph(\qrud)$ is generated by homogeneous elements, the subspace $\hh^\sph$ remains a subalgebra after twisting the product in $\hh$. We therefore will omit the notation ``$tw$'' from $\hh^\sph$ and let the reader infer from context whether the twist is implied.
\end{enumerate} 
\end{dfn}

\begin{rem}
    Instead of the order on objects where $X \preceq Y$ if $\dim(\End(X)) \leq \dim(\End(Y))$, one could instead use a more refined order based on the graded dimension of $\Hom^\bullet(X,X)$. When $\cc = \Rep(\qrud)$, it turns out that all elements $\cf_{(m,n)}$ would be the same with this refined order, with one exception: $\cf^{\mathrm{refined}}_{(2,2)} = (q^2-q)([M]+[M'])$. It follows that the subalgebra generated by the $\cf^{\mathrm{refined}}_{(m,n)}$ is exactly the same as the one generated by the $\cf_{(m,n)}$, so the simpler definition is sufficient for our purposes.
\end{rem}

Our partial order $\preceq$ was inspired by the partial order $\vartriangleleft$ defined in \cite{BG16}; we first restate  one of their key propositions in our notation, which essentially says that the order $\preceq$ is a refinement of $\vartriangleleft$.
\begin{prop}[{\cite[Prop.~4.8]{BG16}}]\label{prop:bg}
For any short exact sequence $N' \to N \to N''$ in  $\mathcal C$, we have 
\[
c(N) \leq c(N'\oplus N'').
\]
If this is an equality, then the short exact sequence splits.
\end{prop}

\begin{rem}\label{rem:spherical}
    We include several elementary remarks: 
    \begin{enumerate} 
    \item If $F: \cc \to \cc$ is an (anti-)autoequivalence , then $F(\Theta_\beta) = \Theta_{F(\beta)}$, so $F$ preserves the spherical subalgebra. 
    \item If $\mathcal C = \Rep(Q)$ for some  quiver $Q$ without oriented cycles, then simple objects are obviously minimal with respect to the ordering $\preceq$, since simples have 1-dimensional endomorphism rings. (This also turns out to be true for the quiver-with-relations $\qrud$.)
    This means the composition subalgebra of $Q$ is a subalgebra of the spherical subalgebra. 
    \item Our results show that the containment is proper for $\qrud$. For example, the composition subalgebra of $\qrud$ is generated by the simples $[I_0]$ and $[P_0]$, so the dimension of the degree $(1,1) \in \ZZ^2 = K_0(\qrud)$ graded piece of the composition subalgebra is 2 (a basis is $[I_0][P_0]$ and $[P_0][I_0]$). However, in the spherical subalgebra, the 3 elements $[I_0][P_0]$, $R_1$, and $R_1'$ all have degree $(1,1)$, and they are linearly independent by \cite[Thm.~2.4]{BG16} (see Theorem \ref{thm:monobasis}).
    \item 
    Given a dimension vector $\uld = (d_1,d_2) \in \NN^2$ for the Rudakov quiver $\qrud$, let \[H(\uld) =  \Hom_{Vect}(\FF_q^{d_1},\FF_q^{d_2})^2 \oplus \Hom_{Vect}(\FF_q^{d_2},\FF_q^{d_1})^2,\] and let $H_{\mathrm{Rud}}(\uld) \subset H(\uld)$ be the subset of quadruples $(e,f,e',f')$ that satisfy the relations of $\qrud$. There is an obvious action of the \emph{gauge group} $\GL(\uld) := \GL(d_1)\times \GL(d_2)$ on $H(\uld)$, and this action preserves the subset $H_{\mathrm{Rud}}(\uld)$. Let $Orb(\uld)$ be the set of $\GL(\uld)$ orbits inside the set $H_\mathrm{Rud}(\uld)$. We then have the equality
    \begin{equation} 
        \frac{\cf_\uld}{\lvert \GL(\uld)\rvert} =  \sum_{Y \in \min_\preceq (Orb(\uld))} \frac{[Y]}{\lvert \GL(\uld)\cdot Y\rvert}
    \end{equation}
    where we have written $[Y]$ for the isomorphism class of the module defined by an element $Y \in H(\uld)$. The equality holds because $\alpha(Y)$ is  the size of the stabilizer of $Y$ in $\GL(\uld)$.
    \end{enumerate}
\end{rem}
The following remark should be compared to Lemma \ref{lem:rnexpression}, and the remark thereafter. 
\begin{remark}\label{rmk:sphericalthoughts}
The term ``spherical subalgebra'' has been used in the literature already in the context of the Hall algebra $Hall(\mathrm{Coh}(X))$ of the category of coherent sheaves on a smooth curve $X$ over $\FF_q$. It turns out that our definition gives the same subalgebra in this context. More explicitly, when $X=\mathbb{P}^1$, the authors of \cite{BS12affine} define
\begin{align*}
    \cfone_{0,d} &:= \sum_{T \in Tor(\mathbb{P}^1): \bar T = (0,d)} [T],\\
    1+\sum_k \cfone_{0,k} z^k &=: \exp\left( \sum_r \frac{T_r}{[r]_v} z^r\right),\\
    1+\sum_k \Theta^{BS}_k z^k &:= \exp\left( (v^{-1}-v)\sum_r T_r z^r\right).
\end{align*}
The first sum is over all torsion sheaves of degree $d$, and the elements $T_r$ and $\cf_r^{BS}$ are defined using the generating series equalities above. Then \cite[Prop.~4.3]{BS12affine} shows that the families $\{\cfone_{0,n}\}$, $\{T_r\}$, and $\{\cf_r^{BS}\}$ each generate the spherical subalgebra. 

In \cite[Remark~4.6]{BS12affine}, the authors give the following additional description of the elements $\cf_r$:
\begin{align}
    \mathcal{T}_{n,x} &:= \mathcal{O}(\mathbb{P}^1)/\mathfrak{m}_x^n,\notag\\
    \Theta^{BS}_r &= v^{-r} \sum_{\substack{(x_1,\cdots,x_m),\\(n_1,\cdots,n_m)}} \prod_{i=1}^m (1-v^{2\deg(x_i)})[\mathcal{T}_{n_i,x_i}],\label{eq:thetadef}
\end{align}
where the sum is over collections $(x_1,\cdots,x_m)$ of distinct points in $\mathbb{P}^1$, and non-negative integers $(n_1,\cdots,n_m)$ satisfying $\sum_i  n_i \deg(x_i) = r$. It is easy to see that $\mathcal{T}_{n,x}$ is the unique $\preceq$-minimal sheaf of class $(0,n) \in K_0(\PP^1)$ that is supported at the point $x$. Also, a torsion sheaf $\mathcal T$ splits canonically as a direct sum $\mathcal{T} = \oplus_{i} \mathcal{T}_{x_i}$, where the $x_i$ are distinct points and each $\mathcal{T}_{x_i}$ is supported at $x_i$. Finally, if $i\not= j$ then $\Hom(\mathcal{T}_{x_i},\mathcal{T}_{x_j}) = 0 = \Ext^1(\mathcal{T}_{x_i},\mathcal{T}_{x_j})$, which implies that the summands in \eqref{eq:thetadef} are exactly the $\preceq$-minimal objects of class $(0,n) \in K_0(\PP^1)$. 
By \cite[Chapter II, (1.6)]{Mac}, if $x_i \in \PP^1$ has degree $d_i$, we have that $\alpha(\mathcal{T}_{n_i,x_i}) = q^{d_in_i}(1-q^{-d_i})$,  which shows that after the substitution $v \mapsto v^{-1}$, we have $\Theta^{BS}_r = \Theta_{(0,r)}\in Hall^{\sph}(\mathrm{Coh}(\mathbb{P}^1))$, up to an overall scalar. 
\end{remark}

Below we compute the $\Theta_\beta$ explicitly when $\cc$ is the category of $\qrud$-modules. It will be helpful to  first prove two general statements which are useful for controlling the terms appearing in the sum defining $\cf_\beta$.

\begin{lem}\label{lem:maxcom}
    Let $Y$ and $Z$ be objects in $\mathcal C$. If $Y\oplus Z \in \min_\preceq\left(\mathcal C_{\beta}\right)$ for some $\beta\in K_0(\cc)$,
    then in the Hall algebra we have the identity $[[Y],[Z]]_t = 0$, for $t := \frac{\lvert \Hom(Z,Y)\rvert} {\lvert \Hom(Y,Z)\rvert}$.
\end{lem}
\begin{proof}

If $W$ is a nonsplit extension of $Y$ by $Z$, Proposition \ref{prop:bg} implies $c(W) < c(Y\oplus Z)$. Combining this with the assumption $Y\oplus Z \preceq W$ implies 
that $\Ext^1(Y,Z) = 0$. By symmetry we have $ \Ext^1(Z,Y)=0$, so Lemma \ref{lem:sumcoeff} implies 
\[
[[Y],[Z]]_t = \frac{\alpha({Y\oplus Z)}}{\alpha(Y) \alpha(Z) } \left( \frac{1}{\big\lvert \Hom(Y,Z)\big\rvert} - \frac{t}{\big\lvert \Hom(Z,Y) \big\rvert} \right) [Y \oplus Z] = 0.
\]
\end{proof}
The following lemma is essentially contained in \cite{BG16}; we state it here in the form we will be using it for the convenience of the reader. 
\begin{lem}\label{lem:nonsplitmin}
    If $N' \to N \to N''$ is a nonsplit short exact sequence, then  we have the strict inequality
    \[
    c(N \oplus Z) < c(N'\oplus N'' \oplus Z)
    \]
    for any $Z$. In particular, $N'\oplus N'' \oplus Z$ is not $\preceq$-minimal for any $Z$.
\end{lem}
\begin{proof}
    From the proof of \cite[Prop.~4.8]{BG16}, for any $Z$ we have the inequalities
    \begin{align*}
        \dim\Hom(N,Z) &\leq \dim\Hom(N'\oplus N'',Z),\\
        \dim\Hom(Z,N) &\leq \dim\Hom(Z,N'\oplus N'').
    \end{align*}
    We then compute
    \begin{align*}
        c(N \oplus Z) &= c(N) + c(Z) + \dim\Hom(N,Z) + \dim\Hom(Z,N)\\
        &\leq c(N) + c(Z) + \dim\Hom(N'\oplus N'',Z) + \dim\Hom(Z,N'\oplus N'')\\
        &< c(N'\oplus N'') + c(Z) + \dim\Hom(N'\oplus N'',Z) + \dim\Hom(Z,N'\oplus N'')\\
        &= c(N'\oplus N''\oplus Z),
    \end{align*}
    where the strict inequality follows from \cite[Prop.~4.8]{BG16} (see Proposition \ref{prop:bg}) and the nonsplit assumption.
\end{proof}

\subsection{Surjectivity: spherical generators for the Rudakov quiver}
We now use our earlier results to study the spherical subalgebra when $\mathcal C$ is the category of $\qrud $-modules. Recall that in Theorem \ref{thm:SQAA} we constructed a map $\theta_{\vv}:\uu_{\vv}(\widehat{\sl_2})_{1,1} \to \hh^{tw}$ from the deformed quantum affine algebra $\uu_v(\widehat{\sl_2})_{1,1}$ to the twisted Hall algebra $\hh^{tw}$. In this section we show 
that the image of the map $\theta_{\vv}$ is exactly the spherical subalgebra $\hh^\sph \subset \hh^{tw}$. This proof relies heavily on the classification of $\qrud$-modules, so we begin with a series of technical lemmas.

\begin{lem}\label{lem:xplusm}
    If $X$ is a $\qrud$-module with dimension vector $(d_1,d_2)$, then 
    \[
    c(X\oplus M) = c(X\oplus M') = c(X) + d_1+d_2+2.
    \]
    In particular, if $X$ and $Y$ have the same dimension vector and $X\prec Y$, then $X\oplus M\prec Y\oplus M$ and $X\oplus M'\prec Y\oplus M'$.
\end{lem}
\begin{proof}
Corollary \ref{cor:mhom} shows $\Hom(M,X) = X_1$ and $\Hom(X,M) = X_2^\vee$, so $c(X\oplus M) = c(X) + d_1 +d_2 + c(M)$ by the bilinearity of the functor $\Hom(-,-)$.
\end{proof}



Recall that we say that a module with dimension vector $(d_1,d_2)$ has positive slope if $d_2-d_1 > 0$. The following two lemmas essentially say that a sum of indecomposables with differently signed slopes cannot be $\preceq$-minimal.
\begin{lem}\label{lem:qcomm} 
    Let $Z$ be in the set $S = \{ I_n, P_n, I'_n, P'_n\,\mid \, n \in \NN \}$. 
    \begin{enumerate} 
    \item If both of $\Ext^1(Z,P_k)$ and $\Ext^1(P_k,Z)$ are zero for some $k\geq 1$, 
    then  $Z \in \{ P_{k-1},P_k,P_{k+1}\} $.
    \item If both of $\Ext^1(Z,P_0)$ and $\Ext^1(P_0,Z)$ are zero, then $Z \in \{I'_{1},P_0,P_1\}$.
    \end{enumerate} 
\end{lem}
\begin{proof}
Both of these claims follow from  Propositions~\ref{prop:braid} and~\ref{prop:braidbb} and Theorem~\ref{thm:mainrel}. 
\end{proof}


\begin{rem}
    For a fixed $k$, the auto-equivalences $\sigma$, $\sigma\circ \tau$, $\tau$ and $\cfone$ act transitively on the set $\{I_k,P_k,I'_k,P'_k\}$, so the statements in the above lemma can be translated to these sequences of modules through these equivalences.
\end{rem}

\begin{lem} \label{lem:noreg}
Let $Z$ be in the set $S = \{ I_n, P_n, I'_n, P'_n\,\mid \, n \in \NN \}$. If $R$ is a regular module, then at least one of $\Ext^1(R,Z)$ or $\Ext^1(Z,R)$ is nonzero. 
\end{lem}

\begin{proof}
Let $Z=P'_n$. (The proof for other $Z$ is similar.) If $R$ is an $(e',f')$-saturated regular representation, then the claim follows from Theorem~\ref{thm:szanto}(1.1f). Hence, by the bilinearity of $\Ext^1(-,-)$, it is sufficient to assume $R = R_\phi(k)$ for some $\phi \in \Sigma_d$. By the proof of Lemma \ref{lem:maxcom}, it suffices to show that there is no constant $t\not= 0$ such that $[R,P'_n]_t = 0$. Towards this end, suppose $[R,P'_n]_t = 0$ for some $t$ and let $X = R_\pi(1)$ with $\pi \in \Sigma_1$ chosen so that $[R,X] = 0$. Then we use the $q$-Jacobi relation \eqref{eq:qjac} (evaluated at $a=t$, $b=1$, $c=q$) to compute
\begin{align*}
    0 &= [R,P'_n]_t\\
    &= [X,[R,P'_n]_t]\\
    &= -tq^{-1}[P'_n,[X,R]]_{qt^{-1}} - q^{-1} [R,[P'_n,X]_q]_t\\
    &= -q^{-1}(q-1)[M][R,P'_{n-1}]_t\\
    &= \cdots\\
    &= (-1)^{n}q^{-n}(q-1)^n[M]^n[R,P'_0]_t\\
    &= (-1)^{n+1}q^{-n-1}(q-1)^n[M]^n[R,[I_0,X]_q]_t\\
    &= (-1)^{n+1}q^{-n-1}(q-1)^n[M]^n[R,I_1]_t,
\end{align*}
where in each inductive step we take a commutator with $X$ and use the identity $[P'_n,X]_q = (q-1)[M][P'_{n-1}]$ from  Lemma \ref{lem:xppcom}, and the fact that $[M]$ is central. Since $R$ and $I_1$ are $(e,f)$-saturated, we may use \cite[Cor.~4.9]{Sz06} to compute 
\begin{align*}
0 &= (-1)^{n+1}q^{-n-1}(q-1)^n[M]^n[R,I_1]_t \\
&= (-1)^{n+1}q^{-n-1}(q-1)^n[M]^n\left( [R][I_1]-tq^{dk}[R][I_1] -t I_{1+dk} -t \sum_{i=1}^{k-1}c_i[R_d(k-i)][I_{1+di}]\right),
\end{align*}
where the $c_i$ are constants. Every term in the sum above is in PBW order (see \cite[Theorem 2.4]{BG16} or Theorem \ref{thm:monobasis}), and the $I_{1+dk}$-term has a nonzero coefficient, which is a contradiction. 

\end{proof}

\begin{cor}\label{cor:slopes}
 If  $Y,Z\in  \{ I_n, P_n, I'_n, P'_n, R_{\phi}(l), R_{\phi}(l)'\,\mid \, n \in \NN, l\geq 1, \phi\in\Sigma \}$ and $Y \oplus Z$ is $\preceq$-minimal, then the sign of $Y$'s slope is that same as that of $Z$'s. 
\end{cor}

\begin{proof}
If $Y\oplus Z$ is minimal, then Proposition~\ref{prop:bg} implies that $\Ext^1(Y,Z)=\Ext^1(Z,Y)=0$; combining this with the two lemmas above implies that $Y$ has positive slope if and only if $Z$ does (since the $P_k$ and $I'_k$ are precisely the indecomposables with positive slope and $P_k'$ and $I_k$ are the ones with negative slope), etc.
\end{proof}

\begin{lem}\label{lem:cfs}
    We have the identities
    \begin{align}
    \cf_{(n,0)} &= \lvert GL_n(\FF_q)\rvert \left[I_0^{\oplus n}\right]\notag\\
    \cf_{(0,n)} &= \lvert GL_n(\FF_q)\rvert \left[P_0^{\oplus n}\right]\notag\\
    \cf_{(1,1)} &=  (q-1)\left(R_1 + R_1'\right)\notag\\
    \cf_{(2,2)} &=  (q-1)\left(q[M] + q[M'] + R_2+R_2'\right)\label{eq:cfcomp}
    \end{align}
    in the  Hall algebra $Hall(\qrud)$. 
\end{lem}
\begin{proof}
    The sums of simples $I_0$ and $P_0$ are obviously the unique objects with dimension vectors $(n,0)$ and $(0,n)$, respectively, which proves the first two identities.
    The sum of simples $I_0 \oplus P_0$ is non-minimal since the simples have nontrivial extensions; all other objects with  dimension vector $(1,1)$ are regular and have $c(R) = 1$ by Theorem \ref{thm:szanto}, which implies the third equality (since all terms in the sum defining $R_1$ and $R_1'$ have the  automorphism group $\FF_q^\times$). To see the last equality, we split into cases using the classification of indecomposable $\qrud$-modules. 
    
    Suppose $Y$ is a $\qrud$-module with dimension $(2,2)$. By dimension arguments, $Y$ is either a sum of regulars or is in one of the following sets:
    \begin{equation*}
    Y_1 := \{M,M'\},\quad Y_2 := \{I_0\oplus P_1, I_0\oplus I_1', P_0 \oplus I_1,P_0\oplus P_1'\},\quad Y_3 := \{R \oplus I_0 \oplus P_0\},\quad Y_4 := \{I_0\oplus I_0\oplus P_0 \oplus P_0\}
    \end{equation*}
    where the $R$ in the definition of $Y_3$ is regular of dimension $(1,1)$.
    
    To prove our claim, we will show that $c(Y) \geq 2$, and the equality is attained only for modules in the sets $S$ or $\sigma(S)$, where $S = \{ M,R_\phi(1), R_\pi(2), R_{\pi}(1)\oplus R_{\pi'}(1)\}$ where $\phi \in \Sigma_2$ and $\pi\not=\pi' \in \Sigma_1$.
    
    By \cite[Lemma 1.1]{Sz06} (see Theorem \ref{thm:szanto}), $c(R_\phi(1))=2=c(R_\pi(2))$. By the same lemma, if $Y = R_\pi(1) \oplus R_{\pi'}(1)$, then $c(Y) = 2$ if $\pi \not= \pi'$ and $c(Y) = 4$ if $\pi = \pi'$. 
    Next, if $Y=R_{\pi}(1)\oplus R_{\pi'}(1)'$, then $c(Y)=4$ by Lemma~\ref{lem:sathom}. Corollary \ref{cor:mhom} shows $c(M)= 2$. This shows that if $Y \in S$ or $Y \in \sigma(S)$, then $c(Y) = 2$, and that sums of regulars that are not in $S$ or $\sigma(S)$ are not minimal. 
    
    Lemma \ref{lem:nonsplitmin} shows objects in $Y_3$ or $Y_4$ are not $\preceq$-minimal, since $I_0$ and $P_0$ have a nontrivial extension. Combining Lemmas \ref{lem:maxcom} and \ref{lem:qcomm}, and some care with notation shows that none of the objects in $Y_2$ are minimal. 
    
    Summarizing, we have shown that the terms appearing in $\Theta_{(2,2)}$ are exactly the terms displayed in equation \eqref{eq:cfcomp}. The coefficients of the terms other than $M$ and $M'$ are computed in Lemma \ref{lem:rnexpression}. 
    Finally, to compute the coefficient $\alpha(M) = q(q-1)$, note that any map $M \to M$ is uniquely determined by the image $x_1 \mapsto ax_1 + bx_2$, and the map is invertible iff $a\not= 0$. 
\end{proof}

\begin{lem}\label{lem:cfs2}
    The elements $R_1$ and $M$ are in the spherical Hall algebra $\hh^\sph$.  
\end{lem}
\begin{proof}
    Using Lemma \ref{lem:cfs} we compute
    \[
    2R_1 =  R_1 + R_1' + R_1 - R_1' =  \frac{1}{q-1}\cf_{(1,1)} + \frac{1}{(q-1)^2}\left[\cf_{(1,0)},\cf_{(0,1)}\right],
    \]
    where the last step uses the identity $[I_0,P_0] = R_1-R_1'$ from Lemma \ref{lem:comm0}. Next, invariance of $\hh^\sph$ under auto-equivalences combined with the identities $[R_1,P_0]_q = (q+1)P_1$ (see Lemma~\ref{lem:xpcom}), $[I_0,P_1]_q = R_2-q[M]$ (see Theorem \ref{thm:relsummary}) and $[R_1,R_1'] = (q+1)([M]-[M'])$ (see Corollary \ref{lem:regcomm}) imply
    \[
     \{R_2-q[M],R_2'-q[M],R_2-q[M'],R_2'-q[M'],[M]-[M'],\Theta_{(2,2)}\} \subset \hh^\sph.
    \]
    We then compute
    \[4q[M] = \frac{1}{q-1}\cf_{(2,2)} - (R_2-q[M]) - (R_2'-q[M]) + q([M]-[M']),\]
    which shows that $[M]$ is in the spherical subalgebra.
\end{proof}

Now we give explicit expressions for $\cf_\beta$ for each $\beta \in \NN^2 \subset K_0(\qrud)$. 
\begin{lem}\label{lem:numerical}
    Suppose $d \in \mathbb{N}_{\geq 1}$ and $m \in \mathbb{N}_{\geq 0}$. Then, there exist unique $j,a,b \in \mathbb{N}_{\geq 0}$ solving the equation
    \[
    (m,m+d) = a(j+1,j+2) + b(j,j+1).
    \]
    Explicitly, $j=\lfloor \frac{m}{d} \rfloor$, and $a=m-jd$ and $b=d-(m-jd)$.
\end{lem}
\begin{proof}
For any $j$, the above choices of $a$ and $b$ are uniquely determined by the fact that $(j+1,j+2)$ and $(j,j+1)$ are linearly independent in $\ZZ^2$. The only $j$ for which both $a$ and $b$ are nonnegative is $j=\lfloor \frac{m}{d} \rfloor$.
\end{proof}
Recall that we write $\qrud(m,n)$ for the set of $\qrud$ representations with dimension vector $(m,n)$.
\begin{lem}\label{lem:cfpos}
    Suppose $m,d \in \mathbb{N}_{\geq 1}$, and let $j = \lfloor \frac m d \rfloor$. Then the only minimizers in $\qrud {(m,m+d)}$ are 
    \[ 
    P(m,m+d) := P_{j+1}^{\oplus (m-jd)} \oplus P_{j}^{\oplus (d + jd - m)}
    \quad \mathrm{and} \quad  \tau(P(m,m+d)).
    \]
    Therefore,
    \begin{align}
    \cf_{(m,m+d)} &= {\alpha(P(m,m+d))}\left([P(m,m+d)] + [\tau(P(m,m+d))]\right)\notag\\
    &= c_{d,m}\alpha(P(m,m+d))\left([P_j]^{d+jd-m}[P_{j+1}]^{m-jd} + [I_{j+1}']^{m-jd}[I_j']^{d+jd-m}\right). \label{eq:lastplease}
    \end{align}

    Similarly, the only minimizers in $\qrud {(m+d,m)}$ are 
    \[ 
    I(m+d,m) := I_{j+1}^{\oplus (m-jd)} \oplus I_{j}^{\oplus (d + jd - m)}
    \quad \mathrm{and} \quad  \tau(I(m+d,m)),
    \]
    so
    \begin{align}
    \cf_{(m+d,m)} &= {\alpha(I(m+d,m))}\left([I(m+d,m)] + [\tau(I(m+d,m))]\right)\notag\\
    &=c_{d,m}\alpha(I(m+d,m))\left([I_{j+1}]^{m-jd}[I_{j}]^{d+jd-m} + [P_{j}']^{d+jd-m}[P_{j+1}']^{m-jd}\right).\label{eq:lastplease2}
    \end{align}
    The constant $c_{d,m}$ in equations~\eqref{eq:lastplease} and~\eqref{eq:lastplease2} is given by:
    \[c_{d,m}^{-1}=\vv^{(d^2-d)/2 - (m-jd)(d+jd-m)}[m-jd]![d+jd-m]!.\]
\end{lem}
\begin{proof}
    We only prove the first set of claims, since the proof in the latter case follows from it by applying $\sigma$.
    Suppose $Y$ is a minimizer in $\qrud{(m,m+d)}$.  Since $d> 0$, we know $Y$ must contain an indecomposable summand with positive slope. By Lemmas~\ref{lem:nonsplitmin} and  \ref{lem:noreg}, this implies $Y$ cannot contain a regular summand. By Lemma~\ref{lem:qcomm}, if $P_k$ is a summand of $Y$, then the only other possible summands are $P_{k-1}$ or $P_{k+1}$ (but not both), and $M$ or $M'$, 
    which means that without loss of generality, we can assume 
    \[Y = P_k^a \oplus P_{k+1}^b \oplus M^r \oplus M'^{r'}.
    \]
    
    First, suppose that $r=r'=0$; then dimension counting using Lemma \ref{lem:numerical} shows the numbers $m$ and $d$ uniquely determine $k$, $a$, and $b$, so the only possibility for $Y$ in this case is $P(m,m+d)$. Similarly, if $I_k'=\tau(P_k)$ is a summand of $Y$ (and $M$ and $M'$ are not), then we must have $Y = \tau(P(m,m+d))$. 

    To show that $M$ and $M'$ cannot be summands of $Y$ we compute $c(Y)$. First, Theorem \ref{thm:szanto} shows 
    \begin{align*}
    c(P(m,m+d)) &= c(P_{j+1}^a \oplus P_{j}^b) \\
    &= c({P_{j+1}^a}) + \dim(\Hom(P_{j+1}^a,P_{j}^b))+ \dim(\Hom(P_{j}^b,P_{j+1}^a)) + c(P_{j}^b)\\
    &=a^2 + 0 + 2ab + b^2 \\
    &= d^2.
    \end{align*}
    Now if $r = 1$, the same dimension argument as in the previous paragraph shows that $Y = P(m-2,m-2+d) \oplus M$, so we use Lemma \ref{lem:xplusm} to compute 
    \begin{equation*}
            c(P(m-2,m-2+d)\oplus M) = d^2+(m-2+m-2+d+2) \geq d^2 + d.
    \end{equation*}
    This shows that $M\oplus P(m-2,m-2+d)$ is not minimal, which shows $r\not= 1$; then Lemma \ref{lem:xplusm} shows $r$ cannot be bigger than $1$ either by induction on $r$. A similar argument shows $r'=0$. To finish the proof, equation~\eqref{eq:lastplease} follows from Theorem~\ref{thm:szanto}[1c], and the computation of the coefficient $c_{d,m}$ follows from Lemma~\ref{lem:sumcoeff} and \cite[Equation 4.9]{Kiri}. 
\end{proof}

\begin{lem}\label{lem:cfequidim}
    For $m \geq 3$, we have the following identity:
    \begin{equation*}
        \cf_{(m,m)} =  (q-1)\left(R_m + R'_m\right).
    \end{equation*}
\end{lem}
\begin{proof}

Suppose $X$ is a minimal object of dimension $(m,m)$. If $X$ is indecomposable, then either $X$ or $\sigma(X)$ is isomorphic to a regular representation $R_\sigma(k)$, where $\sigma \in \Sigma_d$ and $dk = m$. In this case, $c(X) = m$ by Theorem \ref{thm:szanto}.

Now suppose that $Y,Z \subset X$ are  nontrivial indecomposable summands.  
\begin{enumerate}
    \item By Corollary \ref{cor:slopes}, $Y$ cannot be in the set $\{I_n, I'_n, P_n,P'_n\mid n \in \ZZ\}$: if it were, then the slope of the dimension vector of all other  summands would have the same sign as the slope of the dimension vector of $Y$. This is impossible since the slope of $X$ is 0.
    \item If $Y=R_\pi(j)$ and $Z=R_\phi(k)$ are both $(e,f)$-regular, we must have $\pi \not= \phi$ (otherwise they have a nontrivial extension between them, which is ruled out by Lemma \ref{lem:nonsplitmin}). 
    \item If $Y$ is $(e,f)$-regular and $Z$ is $(e',f')$-regular, then $\Hom(Y,Z) = \Hom_{Vect}(Y_1,Z_1)$ by Lemma \ref{lem:sathom}. This implies $c(X) \geq m + \dim(Y_1)\dim(Z_1) > m$ since $Y,Z$ are nontrivial and have slope zero.
    
    \item If $Y = M$, then $\Hom(M,X) = X_1$ by Corollary \ref{cor:mhom}. This implies $c(X) \geq m + c(Z) > m$, where the last inequality follows by the assumption that $Z$ is nontrivial (since $m>2$). An analogous argument implies $M'$ is also not a summand of $X$.
\end{enumerate}
The only remaining possibilities for $X$ are modules of the form
\[
R_{\underline{\phi}}(\underline{n}) := R_{\phi_1}(n_1)\oplus \cdots \oplus R_{\phi_r}(n_r),\qquad \mathrm{or} \qquad \sigma(R_{\underline{\phi}}(\underline{n}))
\]
where the $\phi_i$ are all distinct  and $m = \sum_i n_i \deg(\phi_i)$. Then Lemma \ref{lem:rnexpression} shows the coefficient of the term $R_{\underline{\phi}}(\underline{n})$ in $R_m$ is exactly $\alpha(R_{\underline{\phi}}(\underline{n}))$. Finally, the results in \cite{Sz06} (see Theorem \ref{thm:szanto}) show $c(R_{\underline{\phi}}(\underline{n})) = m $ for any allowed choice of $\underline \phi$ and $\underline n$, so each of these appears as a term in the sum defining $\cf_{(m,m)}$, which completes the proof.
\end{proof}


\begin{prop}\label{prop:image}
    The image of the map $\theta_{\vv}: \uu_{\vv}(\widehat{\sl_2})_{1,1} \to \hh^{tw}$ is exactly the spherical subalgebra $\hh^\sph$.
\end{prop}
\begin{proof}
Corollary~\ref{cor:4gens} shows that the intersection of $Hall(\qrud)$ with the image of $\theta_{\vv}: \uu_{\vv}(\widehat{\sl_2})_{1,1} \to \hh^{tw}$  is generated by the elements $[I_0], [P_0], R_1$ and $[M]$. Lemmas \ref{lem:cfs} and \ref{lem:cfs2} show that these generators are contained in $\hh^\sph$, and Lemmas \ref{lem:cfpos} and \ref{lem:cfequidim} show that all generators $\cf_{(d_1,d_2)}$ of $\hh^\sph$ are contained in the image of $\theta_{\vv}$.
\end{proof}

\subsection{Injectivity, PBW bases} \label{sec:PBW}

In this section, we prove that the map $\theta_{\vv}$ constructed in Theorem~\ref{thm:SQAA} is injective. The main idea of the proof is to construct a spanning set for the algebra $\uu_\vv(\widehat{\sl_2})_{1,1}$ and show that it maps to a linearly independent set under the constructed homomorphism (see Proposition~\ref{prop:injec}). As a consequence, we will also conclude that the spanning set is a basis for $\uu_\vv(\widehat{\sl_2})_{1,1}$ (see Corollary~\ref{cor:basis}). The key ingredient in the proof is the following result:

\begin{thm}[{\cite[Theorem 2.4]{BG16}}] \label{thm:monobasis}
Let $\mathcal{A}$ be a finitary Krull-Schmidt exact category. Then, for any total order on the set of isomorphism classes of indecomposable objects in $\mathcal{A}$, the set of ordered monomials of indecomposable objects forms a basis of the Hall algebra $\hh_{\mathcal{A}}$.
\end{thm}

Define the following set $\mathcal{S}$ of ordered monomials in the generators of $\uu_\vv(\widehat{\sl_2})_{1,1}$:
\[C^{n/2}S^{m}E_{\alpha}^{p_{\alpha}}F_{\beta}^{q_\beta}H_{{\gamma}}^{r_{\gamma}},\]
where ${\alpha}$ and ${\beta}$ are finite (possibly empty) decreasing sequences of integers, ${\gamma}$ is a finite (possibly empty) decreasing sequence of non-zero integers, $p_{\alpha}, q_{\beta}$ and $r_{\gamma}$ are sequences of non-negative integers, and $m,n\in\ZZ$. Here, we are using multi-exponent notation
\[E_{\alpha}^{p_{\alpha}} = E_{\alpha_1}^{p_1}E_{\alpha_2}^{p_2}\cdots E_{\alpha_k}^{p_k},\]
where $\alpha=(\alpha_1,\alpha_2,\dots,\alpha_k),$ etc.

The following lemma is an easy consequence of the defining relations of $\uu_\vv(\widehat{\sl_2})_{1,1}$:

\begin{lem}
The set $\mathcal{S}$ spans the space $\uu_\vv(\widehat{\sl_2})_{1,1}$ as a $\CC$-vector space.
\end{lem}

Under the map $\theta_{\vv}$, the set $\mathcal{S}$ maps (up to scalars) to the set $\tilde{\mathcal{S}}$ of the following monomials in the localized Hall algebra $\hh$:
\[(q-1)^{(a+b)/4}[M]^{a/4}[M']^{b/4}[P_{\alpha}]^{p_{\alpha}}[I_{\beta}']^{q_\beta}[I_{{\gamma}}]^{r_{\gamma}}[P_{\delta}']^{s_{\delta}}R_{\mu}^{m_{\mu}}R_{\nu}'^{n_{\nu}},\]
where $\alpha,\beta,\gamma,\delta,\mu$ and $\nu$ are finite (possibly) empty decreasing sequences of positive integers (we allow the sequences $\alpha$ and $\gamma$ to contain the integer $0$ to account for the two simples), $p_{\alpha}, q_{\beta}, r_{\gamma}, s_{\delta}, m_{\mu}$ and $n_{\nu}$ are sequences of non-negative integers, and $a, b \in \ZZ$. 

\begin{prop} \label{prop:injec}
The set $\tilde{\mathcal{S}}$ is linearly independent in the algebra $\hh$.
\end{prop}

\begin{proof}
Since $[M]^{1/4}$ and $[M']^{1/4}$ are central units in $\hh$, we just need to show that the following elements are linearly independent:
\[X_{\alpha,\beta,\gamma,\delta,\mu,\nu}:=[P_{\alpha}]^{p_{\alpha}}[I_{\beta}']^{q_\beta}[I_{{\gamma}}]^{r_{\gamma}}[P_{\delta}']^{s_{\delta}}R_{\mu}^{m_{\mu}}R_{\nu}'^{n_{\nu}}.\]
For a given sequence $\mu$, we have that
 \[R_{\mu}^{m_{\mu}}=R_{\mu_1}^{m_1}R_{\mu_2}^{m_2}\cdots R_{\mu_k}^{m_k},\]
where $\mu=(\mu_1,\mu_2,\dots, \mu_k)$. Note that we can expand each of the $R_n$ (and the $R_n'$'s) out to express $X_{\alpha,\beta,\gamma,\delta,\mu,\nu}$ as a sum of ordered monomials of indecomposable objects (see Lemma \ref{lem:rnexpression}).

We fix the lexicographic order on the set of sequences $\mu$: we say that $\mu=(\mu_1,\mu_2,\dots,\mu_k)$ is greater than $\mu'=(\mu_1',\mu_2',\dots,\mu_l')$ if any of the following conditions is true:
\begin{itemize}
\item $\mu_1>\mu_1'$
\item $\mu_1=\mu_1'$ and $m_1>m_1'$
\item $\mu_1=\mu_1', m_1=m_1'$ and $\mu_2>\mu_2'$\\ $\vdots$
\end{itemize}

Next, fix an (arbitrary) element $\phi_d\in \Sigma_d$ for all $d\geq 1$. Then, by the definition of $R_n$, we note that the following term occurs as a summand (with non-zero coefficient) of $R_{\mu}^{m_{\mu}}$:
\[[R_{\phi_{\mu_1}}(1)]^{m_{1}}[ R_{\phi_{\mu_2}}(1)]^{m_{2}}\cdots [R_{\phi_{\mu_k}}(1)]^{m_{k}}.\]
 We call this the leading term $L_{\mu}$ of $R_{\mu}^{m_{\mu}}$. 
 Similarly, we define a leading term $L_{\nu}'$ of $R_{\nu}'^{n_{\nu}}$. So, we see that the monomial $X_{\alpha,\beta,\gamma,\delta,\mu,\nu}$ contains the following product of indecomposable elements as a summand:
 \[L_{\alpha,\beta,\gamma,\delta,\mu,\nu}:=[P_{\alpha}]^{p_{\alpha}}[I_{\beta}']^{q_\beta}[I_{{\gamma}}]^{r_{\gamma}}[P_{\delta}']^{s_{\delta}}L_{\mu}L_{\nu}'.\]
 By Theorem~\ref{thm:monobasis}, the elements $L_{\alpha,\beta,\gamma,\delta,\mu,\nu}$ are linearly independent for varying sextuples of sequences $(\alpha,\beta,\gamma,\delta,\mu,\nu)$.

It is clear that the leading term $L_{\mu}$ of $R_{\mu}^{m_{\mu}}$ can not occur as a summand in $R_{\mu'}^{m_{\mu'}}$ if $\mu$ is greater than $\mu'$ in the lexicographic order. As a consequence, the monomial $L_{\alpha,\beta,\gamma,\delta,\mu,\nu}$ can not occur as a summand of $X_{\alpha',\beta',\gamma',\delta',\mu',\nu'}$ unless $\alpha=\alpha'$, $\beta=\beta'$, $\gamma=\gamma'$, $\delta=\delta'$ and $\mu$ and $\nu$ are both lesser than or equal to $\mu'$ and $\nu'$ respectively. Then, by an upper-triangular argument, we conclude that the linear independence of the $L_{\alpha,\beta,\gamma,\delta,\mu,\nu}$'s implies the linear independence of the $X_{\alpha,\beta,\gamma,\delta,\mu,\nu}$'s.
\end{proof}

\begin{cor} \label{cor:basis}
The set $\mathcal{S}$ (resp. $\tilde{\mathcal{S}}$) is linearly independent and forms a $\CC$-basis for the space $\uu_{\vv}(\widehat{\sl_2})_{1,1}$ (resp. $\hh^{\sph}$). Furthermore, the set $\mathcal{S}$ also forms a $\CC(v)$-basis for the space $\uu_v(\widehat{\sl_2})_{1,1}$.
\end{cor}


\begin{proof}
The set $\mathcal{S}$ is linearly independent in $\uu_{\vv}(\widehat{\sl_2})_{1,1}$, since its homomorphic image $\tilde{S}$ is, proving the first claim. To prove the latter claim, we observe that if there is a linear relation between the elements of $\mathcal{S}$ over $\CC(v)$, then that relation must evaluate to zero whenever $v$ is specialized to $v=\vv(=q^{1/2})$ for all but finitely many prime powers $q$. This implies that the relation must itself be zero, proving linear independence.
\end{proof}

\begin{rem} \label{rem:anyorder}
By a straightforward modification of the above proof, it follows that for any ordering on the the set of generators $\{C^{1/2},S,E_{\alpha},F_{\beta},H_{\gamma}\},$ the set of ordered monomials gives a basis of $\uu_{\vv}(\widehat{\sl_2})_{1,1}$.
\end{rem}

\section{The deformed quantum affine algebra}\label{sec:flatpresentation}



In this section we establish some algebraic properties of $\uu_{v}(\widehat{\sl_2})_{1,1}$ and show that it is a flat deformation of $\uu_v^{ad}(\widehat{\sl_2})_{1,1}^{FT}$. We also find an integral form and provide a Levendorskii type finite presentation.

\subsection{Algebraic properties} \label{sec:algeb}

In this section, we describe structural properties of the deformed quantum affine algebra $\uu_v(\widehat{\sl_2})_{1,1}$. First of all, we note that just like the algebras in \cite{FT19}, the algebra $\uu_v(\widehat{\sl_2})_{1,1}$ has a triangular decomposition: let $\uu^{> 0}$, $\uu^0$ and $\uu^{< 0}$ be the subalgebras of $\uu_{v}(\widehat{\sl_2})_{1,1}$ generated by the sets $\{E_{l}: l\in \ZZ\}$, $\{H_{n}, S^{\pm1}, C^{\pm1/2}: n\in \ZZ\setminus\{0\}\}$ and $\{F_{l}: l\in \ZZ\}$ respectively. The following proposition is an immediate consequence of Corollary~\ref{cor:basis} and Remark~\ref{rem:anyorder}:

\begin{prop}
The multiplication map
\[\uu^{>0}\otimes \uu^0\otimes\uu^{< 0}\to \uu_{v}(\widehat{\sl_2})_{1,1}\]
is an isomorphism of $\CC(v)$-vector spaces.
\end{prop}

The above decomposition has an interpretation in terms of slope for the algebra $Hall(\qrud)$. Recall that a representation $V=(V_1,V_2)$ of $\qrud$ is said to have positive, zero or negative slope depending upon whether $\dim(V_1)$ is lesser than, equal to or greater than $\dim(V_2)$, respectively.
Let $Hall^{>0}$, $Hall^{0}$ and $Hall^{<0}$ be the subalgebras of $Hall(\qrud)$ generated by indecomposables having positive, zero and negative slopes, respectively. Then, the following is an analog of the above proposition for $Hall(\qrud)$ and is a consequence of Theorem~\ref{thm:monobasis}:

\begin{prop} \label{prop:Hall3}
The multiplication map
\[Hall^{> 0}\otimes Hall^0\otimes Hall^{< 0}\to Hall(\qrud)\]
is an isomorphism of $\CC$-vector spaces.
\end{prop}

The map $\theta_{\vv}$ constructed in Theorem~\ref{thm:SQAA} sends $\uu^{>0}$ into $\hh^{>0}$, which is defined as the sublagebra of $\hh^{tw}$ obtained by adjoining $((q-1)[M])^{-1/4}$ and $((q-1)[M'])^{-1/4}$ to $Hall^{>0}$, etc. The structure of the slope subalgebras $\hh^{>0}$ and $\hh^{<0}$ was studied in Section~\ref{sec:eachhalf}, whereas that of $\hh^0$ was studied in Section~\ref{sect:regular}. 

The algebra $\uu_v(\widehat{\sl_2})_{1,1}$ has other decompositions that are in some sense `orthogonal' to the above decomposition: fix $k\in\ZZ$. Let $\uu_k^+$ and $\uu_k^+$ be the following subalgebras of $\uu_{v}(\widehat{\sl_2})_{1,1}$: 
\[
\uu_k^+ := \langle \{E_{n+k-1},F_{n-k}, H_n, S^{\pm 1}, C^{\pm1/2}: n\geq 1\}\rangle,
\qquad \uu_k^- := \langle\{E_{n+k},F_{n-k+1}, H_n, S^{\pm 1}, C^{\pm1/2}: n\leq -1\}\rangle.
\]
Then, it is clear from the defining relations that both of these subalgebras are spanned by ordered monomials in their respective generators. These ordered monomials are linearly independent by Corollary~\ref{cor:basis}, giving bases for both of these spaces. Furthermore, both of these algebras can be viewed as left as well as right modules over the commutative algebra $\mathcal{Z}:=\CC[C^{\pm1/2}, S^{\pm 1}]$. Then, we have the following proposition:

\begin{prop}
The multiplication map
\[\uu_k^{+}\otimes_{\mathcal{Z}}\uu_k^{-}\to \uu_{v}(\widehat{\sl_2})_{1,1}\]
is an isomorphism of $\CC(v)$-vector spaces.
\end{prop}

\begin{proof}
The surjectivity is clear. To see injectivity, we note that the algebras $\uu^{\pm}_k$ are free as left (or right) $\mathcal{Z}$-modules, which follows from Corollary~\ref{cor:basis}.
\end{proof}


For $k=0$ and $k=1$, the above decomposition can be interpreted Hall theoretically using Definition~\ref{def:sat}. Let $Hall_0^+ \subset Hall(\qrud)$ (resp.~$Hall_1^+$) be the subspace spanned by $(e,f)$-injective (resp.~$(e,f)$-surjective) objects, and let $Hall_0^-$ (resp.~$Hall_1^-$) be the subspace spanned by $(e',f')$-injective (resp.~$(e',f')$-surjective) objects. By Lemma \ref{lem:saturated}, both $Hall_0^{\pm}$ and $Hall_1^{\pm}$ are subalgebras, not just subspaces.
\begin{prop}\label{prop:hales}
    For $k=0,1$, the multiplication map 
    \[
    \mu: Hall_k^+ \otimes Hall_k^- \to Hall(\qrud)
    \]
    is an isomorphism of vector spaces.
\end{prop}
\begin{proof}
    The proof follows from Theorem \ref{thm:monobasis} after observing that the set of $(e,f)$-surjective representations and $(e',f')$-surjective representations (resp.~$(e,f)$-injective and $(e',f')$-injective representations) exactly partition the set of indecomposable $\qrud$-representations (see Remark~\ref{rem:sat} and the preceding table). 
\end{proof}

For $k=0,1$, the map constructed in Theorem~\ref{thm:SQAA} sends $\uu_k^{\pm}$ into $\hh_k^{\pm}$, which is defined as the sublagebra of $\hh^{tw}$ obtained by adjoining $((q-1)[M])^{-1/4}$ and $((q-1)[M'])^{-1/4}$ to $Hall_k^{\pm}$.

Next, we compute the center of $\uu_v(\widehat{\sl_2})_{1,1}$. Recall that by Corollary~\ref{cor:basis}, the space $\uu_v(\widehat{\sl_2})_{1,1}$ has a $\CC(v)$-basis given by the following set $\mathcal{S}$ of ordered monomials:
\[C^{n/2}S^{m}E_{\alpha}^{p_{\alpha}}F_{\beta}^{q_\beta}H_{{\gamma}}^{r_{\gamma}},\]
where ${\alpha}$ and ${\beta}$ are finite (possibly empty) decreasing sequences of integers, ${\gamma}$ is a finite (possibly empty) decreasing sequence of non-zero integers, $p_{\alpha}, q_{\beta}$ and $r_{\gamma}$ are sequences of non-negative integers, and $m,n\in\ZZ$. Denote by $|p_{\alpha}|$ the sum of all exponents occurring in $E_{\alpha}^{p_{\alpha}}$, etc.

\begin{prop} \label{prop:center}
The center of the algebra $\uu_v(\widehat{\sl_2})_{1,1}$ is equal to $\CC[C^{\pm1/2}]$.
\end{prop}

\begin{proof}
Suppose $z$ is an element in the center. Then $z$ is a $\CC(v)$-linear combination of elements of the form
\[C^{n/2}S^{m}E_{\alpha}^{p_{\alpha}}F_{\beta}^{q_\beta}H_{{\gamma}}^{r_{\gamma}}\]
from the set $\mathcal{S}$. Pick one such monomial $\mathfrak{m}$ in $z$ with non-empty $\alpha$, that is maximal with respect to the lexicographic order on $\alpha$. By relation~\eqref{rrel7}, we note that the commutator $[H_1,\mathfrak{m}]$ is a non-zero linear combination of terms that are bigger than $\mathfrak{m}$ with respect to the lexicographic order. 
However, this contradicts the fact that $[H_1,z]$ must be zero. Hence, the element $z$ can only be a linear combination of elements in which the sequence $\alpha$ is empty. By a similar argument, we can conclude that $\beta$ must also be empty.

So, $z$ must be a linear combination of elements of the form
\[C^{n/2}S^{m}H_{{\gamma}}^{r_{\gamma}}.\]
Next, pick $k$ such that $H_k$ occurs with a non-zero power in a monomial in $z$. Then, by relation~\eqref{rrel9}, it is clear that the commutator of $H_{-k}$ with $z$ can not be zero. This shows that $z$ can only consist of terms of the form $C^{n/2}S^{m}$. To that end, note that
\[[C^{n/2}S^m, E_0] = (1-v^{-m})C^{n/2}S^mE_0.\]
This forces $m$ to be zero in any monomial in $z$. Therefore, $z$ must be a Laurent polynomial in $C^{1/2}$, completing the proof.
\end{proof}

\subsection{Integral form} \label{sec:intform}


In this section, we find an integral form of the algebra $\uu_v(\widehat{\sl_2})_{1,1}$ via the construction of an integral version of the spherical subalgebra $\hh^{\sph}$. Recall that $\hh^{\sph}$ is generated as a $\CC$-algebra by the set:
\[\mathcal{G}:=\{[P_l],[I_l],[P_{l+1}'],[I_{l+1}'],R_{l+1},R'_{l+1},((q-1)[M])^{\pm1/4},((q-1)[M'])^{\pm1/4}:l\in\ZZ_{\geq 0}\}.\]
We constructed a basis $\tilde{\mathcal{S}}$ of $\hh^{\sph}$ in Section~\ref{sec:PBW} consisting of the following ordered monomials in the elements of $\mathcal{G}$:
\[(q-1)^{(a+b)/4}[M]^{a/4}[M']^{b/4}[P_{\alpha}]^{p_{\alpha}}[I_{\beta}']^{q_\beta}[I_{{\gamma}}]^{r_{\gamma}}[P_{\delta}']^{s_{\delta}}R_{\mu}^{m_{\mu}}R_{\nu}'^{n_{\nu}},\]
where $\alpha,\beta,\gamma,\delta,\mu$ and $\nu$ are finite (possibly) empty decreasing sequences of positive integers (we allow the sequences $\alpha$ and $\gamma$ to contain the integer $0$ to account for the two simples), $p_{\alpha}, q_{\beta}, r_{\gamma}, s_{\delta}, m_{\mu}$ and $n_{\nu}$ are sequences of non-negative integers, and $a, b \in \ZZ$.

Note that the set $\tilde{S}$ is independent of the field $\FF_q$ we are working over. Pick monomials $X,Y,Z\in\tilde{\mathcal{S}}$ and let $P^{Z}_{X,Y}(q)$ denote the coefficient of $Z$ in the product $XY$ in $\hh^{\sph}$.

\begin{prop}
For any $X,Y,Z\in\tilde{\mathcal{S}}$, the structure constant $P^{Z}_{X,Y}(q)$ is a Laurent polynomial in $q^{1/2}$ with integer coefficients (depending upon $X,Y$ and $Z$).
\end{prop}

\begin{proof}
We claim that for any two elements $P,Q\in\mathcal{G}$, the commutator $[P,Q]$ can be expressed as a linear combination of ordered monomials in the elements of $\mathcal{G}$ such that the coefficient of each summand is a Laurent polynomial in $q^{1/2}$ with integer coefficients. This claim implies the proposition, since the product $XY$ can then be written as a linear combination of ordered monomials by swapping the order of the multiplicands one-by-one and introducing Laurent polynomials in $q^{1/2}$ with integer coefficients at each step.

When either of $P$ or $Q$ is $((q-1)[M])^{\pm1/4}$ or $((q-1)[M'])^{\pm1/4}$, the above claim is clear. If both $P$ and $Q$ have positive slopes or negative slopes, the claim follows from Propositions~\ref{prop:braid} and~\ref{prop:braidbb}. If one of them has negative slope and the other has positive slope, the claim is a consequence of Theorem~\ref{thm:mainrel}. 

Next, recall the sequences $(\phi_d), (\psi_{-d})$ and $(h_n)$ and the element $\kappa$, defined before Thereom~\ref{thm:SQAA}, that satisfy:
\begin{align*}
\psi_d&= \vv^{-d-1}(q-1)\kappa*R_{d+1},\\
\varphi_{-d}&=\vv^{-d-1}(q-1)\kappa^{2d+1}*R'_{d+1},\\
\sum_{k\geq 0}\psi_{k-1} u^k &= \kappa* \mathrm{exp}\left( (\vv-\vv^{-1}) \sum_{k=1}^\infty h_k u^k \right),\\
\sum_{k\geq 0}\varphi_{-k+1} u^k &= \kappa^{-1} *\mathrm{exp}\left( -(\vv-\vv^{-1}) \sum_{k=1}^\infty h_{-k} u^{-k} \right).\end{align*}
If one of $P$ and $Q$ has zero slope and other other has non-zero slope, then the claim is a consequence of Proposition~\ref{prop:rel78} and the fact that $R_d$ (resp. $R_d'$) can be expressed as polynomial functions of $h_n$'s (resp. $h_{-n}$'s) with $1\leq n\leq d$ with coefficients that are Laurent polynomials in $q^{1/2}$. Furthermore, when $R_d'$ is expressed as
a polynomial of the $h_{-n}'s$, the coefficients are all multiples of $\vv-\vv^{-1}$. Hence, the case when $P$ and $Q$ both have zero slope is a consequence of Proposition~\ref{prop:rel9}.
\end{proof}

The above proposition shows that we can define the \emph{universal integral Hall algebra} $\hh_{v}$ 
as follows:
\begin{defn}
Define $\hh^{\ZZ}_{v}$ to be the $\ZZ[v,v^{-1}]$-algebra generated by the set $\tilde{\mathcal{S}}$ where the product is defined via:
\[X\cdot Y:=\sum_{Z\in \tilde{\mathcal{S}}}P^Z_{X,Y}(v)Z\]
for any $X,Y\in\tilde{\mathcal{S}}.$ We also define the $\CC[v,v^{-1}]$-algebra $\hh_v$:
\[\hh_v:=\CC[v,v^{-1}]\otimes_{\ZZ[v,v^{-1}]}\hh^{\ZZ}_v.\]
Finally, for any $\vv\in\CC^{\times}$, we define the $\CC$-algebra $\hh_{\vv}$ to be the specialization of $\hh_v$ at $v=\vv$.
\end{defn}

It is clear that the specialized algebra $\hh_{\vv}$ for $\vv=q^{1/2}$ (where $q$ is any prime power) is exactly the spherical subalgebra $\hh^{\sph}$ of the localized and twisted Hall algebra of $\qrud$-representations over $\FF_q$. This fact, in turn, implies that the polynomials $P^Z_{X,Y}(v)$ are zero for all but finitely many $Z$ and so the sum in the definition above is well-defined and the product is indeed associative. Furthermore, the space $\hh^{\ZZ}_{v}$ is a free $\ZZ[v,v^{-1}]$-module since a basis is given by the set $\tilde{S}$. As a $\ZZ[v,v^{-1}]$-algebra, $\hh^{\ZZ}_{v}$ is generated by the set $\mathcal{G}$.

By Theorem~\ref{thm:SQAA}, we have a $\CC$-algebra isomorphism:
\[\theta_{\vv}:\uu_{\vv}(\widehat{\sl_2})_{1,1}\to \hh_{\vv}\]
whenever $\vv=q^{1/2}$. Furthermore, since the isomorphism in Theorem~\ref{thm:SQAA} was defined by mapping the generators of  $\uu_{\vv}(\widehat{\sl_2})_{1,1}$ to elements of $\tilde{S}$ with coefficients being Laurent polynomials in $\vv$, we can lift the above map to get that: 
\begin{cor}
There exists a $\CC(v)$-algebra isomorphism:
\[\theta_v:\uu_{v}(\widehat{\sl_2})_{1,1}\to \CC(v)\otimes_{\CC[v,v^{-1}]}\hh_{v},\]
that is given exactly by the formulae in Theorem~\ref{thm:SQAA} after replacing $\vv$ by the parameter $v$.
\end{cor}


We can use $\theta_v$ to pull back the $\CC[v,v^{-1}]$-subalgebra $\hh_v$ of the algebra on the right, to get an integral form for $\uu_{v}(\widehat{\sl_2})_{1,1}$. This integral form can be described explicitly as follows.

\begin{defn} \label{def:intform}
Define 
$\uu_v(\widehat{\sl_2})_{1,1}^{int}$ to be the $\CC[v,v^{-1}]$-subalgebra of $\uu_v(\widehat{\sl_2})_{1,1}$ generated by the elements:
\[\{S^{\pm1},C^{\pm1/2}, E_l,F_l,(v-v^{-1})^{-1}\Psi_n,(v-v^{-1})^{-1}\Phi_{-n}:l\in \ZZ, n\in\ZZ_{\geq 0}\}.\]
\end{defn}
By Theorem~\ref{thm:SQAA}, it is clear that
\[\uu_{v}(\widehat{\sl_2})_{1,1}^{int}=\theta_v^{-1}(1\otimes\hh_v),\]
and that $\uu_{v}(\widehat{\sl_2})_{1,1}^{int}$ is a free $\CC[v,v^{-1}]$-module. A PBW basis for $\uu_{v}(\widehat{\sl_2})_{1,1}^{int}$ is given by the set of ordered monomials:
\[(v-v^{-1})^{-(|r_{\gamma}|+|s_{\delta}|)}C^{n/2}S^{m}E_{\alpha}^{p_{\alpha}}F_{\beta}^{q_\beta}\Psi_{{\gamma}}^{r_{\gamma}}\Phi_{{\delta}}^{s_{\delta}},\]
where ${\alpha}$ and ${\beta}$ are finite (possibly empty) decreasing sequences of integers, ${\gamma}$ is a finite (possibly empty) decreasing sequence of non-negative integers, $p_{\alpha}, q_{\beta}, r_{\gamma}$ and $s_{\delta}$ are sequences of non-negative integers, $|r_{\gamma}|$ and $|s_{\delta}|$ denote the sums of all the elements in the sequences $r_{\gamma}$ and $s_{\delta}$ respectively, and $m,n\in\ZZ$.
\begin{rem} 
The integral form $\uu_v(\widehat{\sl_2})_{1,1}^{int}$ is a slight refinement of the one constructed for the algebra $\uu_v^{ad}(\widehat{\sl_2})^{FT}_{1,1}$ in \cite[Definition 4.4]{FT20}: each of the generators in our integral form is $(v-v^{-1})^{-1}$ times the corresponding generator in the integral form of \cite{FT20}. An even more refined integral form can be obtained by replacing the basis $\tilde{\mathcal{S}}$ by the set of isomorphism classes of representations of $\qrud$. This would have the effect of replacing $E_l$ by elements of the form $E_l^n/[n]!$ for all $n\geq 1$ as generators for the integral form (and similarly for the other generators). That integral form should be more suitable for localizing $v$ to a root of unity, but we don't pursue that here. 
\end{rem}

\subsection{Flat deformation}\label{sec:defat}
In this section, we prove that the algebra $\uu_v(\widehat{\sl_2})_{1,1}$ is a deformation of the adjoint version of the shifted quantum affine algebra defined in \cite{FT19}. We recall its definition (using the notation in \cite{FT19}): 
the algebra $\uu_v^{ad}(\widehat{\sl_2})_{1,1}^{FT}$ is the $\mathbb{C}(v)$-algebra generated by elements $e_l, f_l, \psi^+_{n-1}\psi^-_{1-n}, (\phi^{\pm})^{\pm1}$ for $l\in\ZZ, n\in \ZZ_{\geq 0}$ satisfying the following relations:
 \begin{align*}
        \phi^+\phi^-&=\phi^-\phi^+\\
        \phi^{\pm}e_k (\phi^{\pm})^{-1} &= v^{\pm1} e_k\\
        \phi^{\pm}f_k (\phi^{\pm})^{-1} &= v^{\mp1} f_k\\
        \phi^{\pm}h_n(\phi^{\pm})^{-1}&=h_n\\
        e_{k+1}e_l-v^2 e_l e_{k+1} &= v^2 e_ke_{l+1}-e_{l+1}e_k\\
        f_{k+1}f_l-v^{-2} f_l f_{k+1} &= v^{-2} f_kf_{l+1}-f_{l+1}f_k\\
        [e_k,f_l] &= \frac{\psi_{k+l}^+ - \psi_{k+l}^-}{v-v^{-1}}\\
        [h_l,e_k] &= \frac{[2l]}{l}e_{k+l}\\
        [h_l,f_k] &= \frac{-[2l]}{l}f_{k+l}\\
        [h_l, h_k] &= 0
    \end{align*}
    where the elements $h_k$ are defined via the following generating series:
    \begin{align*}
        \sum_{k\geq 0}\psi^+_{k-1} u^k &= (\phi^+)^2 \mathrm{exp}\left( (v-v^{-1}) \sum_{k=1}^\infty h_k u^k \right) \notag \\
        \sum_{k\geq 0}\psi^-_{-k+1} u^k &= (\phi^-)^2 \mathrm{exp}\left( -(v-v^{-1}) \sum_{k=1}^\infty h_{-k} u^{k} \right). \notag
    \end{align*}

\begin{prop} \label{prop:deform}
Define the $\mathbb{C}(v)[t]$-algebra $\mathcal{A}_t$ generated by elements $E_l, F_l, H_n, S^{\pm1},K^{\pm 1}, C^{\pm 1/2}, \tilde{C}^{\pm1/2}$ for $l\in\ZZ, n\in \ZZ\setminus\{0\}$ satisfying the following relations:
   \begin{align*}
        \tilde{C}^{1/2} &\textrm{ is central}\\
        C^{1/2}-1&=t(\tilde{C}^{1/2}-1)\\
        S^2&=KC^{1/2}\\
        SE_k S^{-1} &= v E_k\\
        SF_k S^{-1} &= v^{-1} F_k\\
        SH_nS^{-1}&=H_n\\
        E_{k+1}E_l-v^2 E_l E_{k+1} &= v^2 E_kE_{l+1}-E_{l+1}E_k\\
        F_{k+1}F_l-v^{-2} F_l F_{k+1} &= v^{-2} F_kF_{l+1}-F_{l+1}F_k\\
        [E_k,F_l] &= \frac{C^{(k-l)/2} \Psi_{k+l} - C^{(l-k)/2}\Phi_{k+l}}{v-v^{-1}}\\
        [H_l,E_k] &= \frac{[2l]}{l}C^{-\lvert l\rvert/2}E_{k+l}\\
        [H_l,F_k] &= \frac{-[2l]}{l}C^{\lvert l\rvert/2}F_{k+l}\\
        [H_l, H_k] &= \delta_{l,-k}\frac{[2l]}{l}\frac{C^l-C^{-l}}{v-v^{-1}}
    \end{align*}
    where the elements $\Psi_k$ and $\Phi_k$ are defined via the following generating series:
    \begin{align*}
        \sum_{k\geq 0}\Psi_{k-1} u^k &= K \mathrm{exp}\left( (v-v^{-1}) \sum_{k=1}^\infty H_k u^k \right) \notag \\
        \sum_{k\geq 0}\Phi_{-k+1} u^k &= K^{-1} \mathrm{exp}\left( -(v-v^{-1}) \sum_{k=1}^\infty H_{-k} u^{k} \right). \notag
    \end{align*}

Then, the $\mathbb{C}[t]$-algebra $\mathcal{A}_t$ is flat. Furthermore, we have the following isomorphisms:
\begin{align*}
    \mathcal{A}_t/(t-1)&\cong\uu_v(\widehat{\sl_2})_{1,1}, \\
    \mathcal{A}_t/(t)&\cong \uu_v^{ad}(\widehat{\sl_2})_{1,1}^{FT}.
\end{align*}
\end{prop}


\begin{proof}
That $\mathcal{A}_t/(t-1)$ is isomorphic to $\uu_v(\widehat{\sl_2})_{1,1}$ is clear from the defining relations. The isomorphism $\mathcal{A}_t/(t)\cong \uu_v^{ad}(\widehat{\sl_2})_{1,1}^{FT}$ is given by the following assignments:
\[
E_k\mapsto e_k(\phi^+)^{-1},\quad
F_k\mapsto f_k(\phi^-)^{-1},\quad 
H_n\mapsto h_n, \quad
\tilde{C}^{\pm{1/2} }\mapsto (\phi^+\phi^-)^{\pm 1}, \quad
S^{\pm 1} \mapsto (\phi^+)^{\pm1}.
\]
Next, we show that the family of algebras $\mathcal{A}_t$ is flat over $t$. This is equivalent to finding a $t$-independent $\mathbb{C}[t]$-basis for the algebra $\mathcal{A}_t$. Consider the elements
\[
\mathcal B := \{(\tilde{C}^{1/2}-1)^{n}S^{m}E_{\alpha}^{p_{\alpha}}F_{\beta}^{q_\beta}H_{{\gamma}}^{r_{\gamma}}\},
\]
where ${\alpha}$ and ${\beta}$ are finite (possibly empty) decreasing sequences of integers, ${\gamma}$ is a finite (possibly empty) decreasing sequence of non-zero integers, $p_{\alpha}, q_{\beta}$ and $r_{\gamma}$ are sequences of non-negative integers, and $m,n\in\ZZ$.

When $t= 1$, the  set $\mathcal{B}$ of elements is a basis for $\mathcal{A}_1:=\mathcal{A}_t/(t-1)$ as a consequence of Corollary~\ref{cor:basis}. For varying $t\neq 0$, there is an isomorphism $\mathcal{A}_t\to \mathcal{A}_1$ given by sending all generators to themselves, except $\tilde{C}^{1/2}$, where the map is \[\tilde{C_t}^{1/2} \mapsto t^{-1}(\tilde{C}_{t=1}^{1/2} + t - 1).\] Thus, $\mathcal B$ is a basis for arbitrary non-zero $t$. 
When $t=0$, the set $\mathcal B$ provides a basis for the algebra $\uu_v^{ad}(\widehat{\sl_2})_{1,1}^{FT}$ by Theorem E.2 in \cite{FT19}, completing the proof. (The theorem in \cite{FT19} computes a basis for the simply connected version of the algebra, but the generalization to the adjoint case is immediate.)
\end{proof}

\begin{rem} \label{rem:Hayashi}
Consider the following maximal commutative subalgebra of $\mathcal{A}_0$:
\[
\mathcal{Z}_0 := \mathbb{C}\left[S^{\pm1}, \tilde{C}^{\pm1/2},h_n: n\in\ZZ\setminus\{0\}\right]\cong \mathbb{C}[S^{\pm1}, \tilde{C}^{\pm1/2}]\otimes {Sym}\otimes {Sym}.
\] 
This deformation of the shifted quantum algebra induces a quantization of the algebra $\mathcal{Z}_0$ as follows: consider the subalgebra $\mathcal{Z}_t$ of $\mathcal{A}_t$ generated by $S^{\pm1}$, $\tilde{C}^{\pm1/2}$ and $H_n$ for $n\in\ZZ\setminus\{0\}$. (The algebra $\mathcal{Z}_{t=1}$ is exactly the algebra $\uu^0$ defined in the previous section.)

Then, by a Hayashi-type construction (see $\cite[\text{Section }7.3]{Gin})$, we can define a Poisson bracket $\{\cdot,\cdot\}$ on the algebra $\mathcal{Z}_0$:
\[\{a,b\}:=\lim_{t\to 0}\frac{1}{t}[\tilde{a},\tilde{b}],\]
where $a,b\in \mathcal{Z}_0$ and $\tilde{a}, \tilde{b}$ are arbitrary lifts of $a,b$ (respectively) to $\mathcal{Z}_t$. Then, we get that
\[S^{\pm1} \text{ and }\tilde{C}^{\pm1/2} \text{ are Poisson central},\]
\[\{h_l,h_k\} := \lim_{t\to 0}\frac{1}{t}[H_l,H_k] = \delta_{l,-k}\frac{4[2l]}{v-v^{-1}}(\tilde{C}^{1/2}-1).\]
This bracket makes $\mathcal{Z}_0$ a Poisson algebra with the Moyal bracket, whose quantization is given by the (quantum) Heisenberg algebra $\mathcal{Z}_t$. (See also the introduction to Section~\ref{sect:regular}.)
\end{rem}


\subsection{Finite presentation of the deformed quantum affine algebra} \label{sec:finpres}

Fix a pair of non-negative integers $(b_1,b_2)\in\ZZ_{\geq 0}$. We recall the definition of the deformed quantum affine algebra $\uu_v(\widehat{\sl_2})_{b_1,b_2}$: it is the $\mathbb{C}(v)$-algebra generated by elements $E_l, F_l, H_n, S^{\pm1}, K^{\pm1}, C^{\pm 1/2}$ for $l\in\ZZ, n\in \ZZ\setminus\{0\}$ satisfying the following relations:
    \begin{align}
        C^{1/2} &\textrm{ is central} \label{rela1}\\
        S^2&=KC^{1/2}\label{rela3.6}\\
        SE_k S^{-1} &= v E_k\label{rela2}\\
        SF_k S^{-1} &= v^{-1} F_k \label{rela3}\\
        SH_nS^{-1}&=H_n\label{rela3.5}\\
        E_{k+1}E_l-v^2 E_l E_{k+1} &= v^2 E_kE_{l+1}-E_{l+1}E_k\label{rela4}\\
        F_{k+1}F_l-v^{-2} F_l F_{k+1} &= v^{-2} F_kF_{l+1}-F_{l+1}F_k\label{rela5}\\
        [E_k,F_l] &= \frac{C^{(k-l)/2} \Psi_{k+l} - C^{(l-k)/2}\Phi_{k+l}}{v-v^{-1}}\label{rela6}\\
        [H_l,E_k] &= \frac{[2l]}{l}C^{-\lvert l\rvert/2}E_{k+l}\label{rela7}\\
        [H_l,F_k] &= \frac{-[2l]}{l}C^{\lvert l\rvert/2}F_{k+l} \label{rela8}\\
        [H_l, H_k] &= \delta_{l,-k}\frac{[2l]}{l}\frac{C^l-C^{-l}}{v-v^{-1}}\label{rela9}
    \end{align}
    where the elements $\Psi_k$ and $\Phi_k$ are defined via the following generating series:
    \begin{align}
        \sum_{k\geq 0}\Psi_{k-b_1} u^k &= K \mathrm{exp}\left( (v-v^{-1}) \sum_{k=1}^\infty H_k u^k \right) \notag \\
        \sum_{k\geq 0}\Phi_{-k+b_2} u^k &= K^{-1} \mathrm{exp}\left( -(v-v^{-1}) \sum_{k=1}^\infty H_{-k} u^{k} \right). \notag
    \end{align}

 It can be shown 
 that the algebra $\uu_v(\widehat{\sl_2})_{b_1,b_2}$ is a deformation of the adjoint version of the shifted quantum affine algebra $\uu_v^{ad}(\widehat{\sl_2})^{FT}_{b_1,b_2}$ from \cite{FT19}. It is clear that the algebra $\uu_v(\widehat{\sl_2})_{b_1,b_2}$ is isomorphic to $\uu_v(\widehat{\sl_2})_{0,b_1+b_2}$; for example, one such isomorphism from $\uu_v(\widehat{\sl_2})_{b_1,b_2}$ to $\uu_v(\widehat{\sl_2})_{0,b_1+b_2}$ is given by the following assignments: 
 \[
 E_k\mapsto E_{k+b_1}, \quad
 F_k\mapsto F_k, \quad
 \Psi_k\mapsto\Psi_{k+b_1}, \quad
 \Phi_k\mapsto \Phi_{k+b_1}, \quad
 S\mapsto S, C^{1/2}\mapsto C^{1/2}.
 \]

Therefore, we will henceforth work with the algebra $\uu_{d}:=\uu_v(\widehat{\sl_2})_{0,d}$ for some $d\in\ZZ_{\geq 0}$. In this section, we will show that $\uu_d$ has a presentation in terms of finitely many generators and relations. This will be a dominant shift analogue of the Levendorskii type presentation for the algebra $\uu_v^{ad}(\widehat{\sl_2})^{FT}_{d_1,d_2}$ provided in \cite[Theorem 5.5]{FT19} in the anti-dominant case. (Although we are working with a deformation of their algebra, a similar presentation works in the undeformed case too.) The proof of this presentation (Theorem~\ref{thm:finpres}) will be in a sense parallel to the proof of our main theorem (Theorem~\ref{thm:SQAA}), which was done by computing some Hall products (Theorem~\ref{thm:relsummary}), and then performing algebraic manipulations to prove the rest of the relations in $\uu_v(\widehat{\sl_2})_{1,1}$.  Under this comparison, the Hall algebra products we compute are analogous to the finite set of relations in $\uu_d$ in Theorem~\ref{thm:finpres}. 

\begin{definition}\label{def:otherelements}
Let $\mathcal F$ be the free algebra generated by  $E_{0},\, F_{d},\, S^{\pm 1},\, C^{\pm 1/2},H_{\pm1}, \{\Phi_l \,\mid\, 2\leq l\leq d-2\}\}.$ To shorten future formulas, we define the following elements in $\mathcal F$: 
\begin{align*}
K&:=S^2C^{-1/2},\\
\Psi_{0}&:= K,\\
\Phi_{d}&:=K^{-1},\\
\Psi_{1}&:=(v-v^{-1})KH_1,\\
\Phi_{d-1}&:=-(v-v^{-1})K^{-1}H_{-1}.
\end{align*}
For $n\geq 0$, $m\geq d$:
\[E_{n+1}:= \frac{1}{[2]}C^{1/2} [H_1,E_n], \qquad F_{m+1}:= -\frac{1}{[2]}C^{-1/2} [H_1,F_m].\]
For $n\leq 0$, $m\leq d$:
\[E_{n-1}:= \frac{1}{[2]}C^{1/2} [H_{-1},E_n], \qquad F_{m-1}:= -\frac{1}{[2]}C^{-1/2} [H_{-1},F_m].\]
For $n<0$, define $\Psi_n=0$, and for $n>d$, define $\Phi_n=0$. Define $\Psi_n$ for $n>1$ and $\Phi_n$ for $n<\min\{2,d-1\}$ via:
    \[[E_0,F_n] = \frac{C^{-n/2} \Psi_{n} - C^{n/2}\Phi_{n}}{v-v^{-1}}.\]
Finally, define the elements $H_n$ when $|n|\geq 2$ via the following generating series:
    \begin{align}
        \sum_{k\geq 0}\Psi_{k} u^k &= K \mathrm{exp}\left( (v-v^{-1}) \sum_{k=1}^\infty H_k u^k \right) \notag \\
        \sum_{k\geq 0}\Phi_{-k+d} u^k &= K^{-1} \mathrm{exp}\left( -(v-v^{-1}) \sum_{k=1}^\infty H_{-k} u^{k} \right). \notag
    \end{align}
\end{definition}

\begin{thm} \label{thm:finpres}
The algebra $\uu_d$ is isomorphic to the quotient of the free algebra 
\[
\mathcal F = \CC(v)\langle E_{0},\, F_{d},\, S^{\pm 1},\, C^{\pm 1/2},H_{\pm1},\, \{\Phi_l \,\mid\, 2\leq l\leq d-2\}\rangle,
\]
by the following relations:
\begin{align}
        C^{1/2} &\textrm{ is central} \label{sel1}\\
        [S,H_{\pm1}]=[S,\Phi_n]&=0 \text{ for } 2\leq n\leq d-2\label{sel1.5}\\
        SE_0 S^{-1} &= v E_0\label{sel2}\\
        SF_d S^{-1} &= v^{-1} F_d \label{sel3}\\
        [H_n,H_m] &= [H_{-m},H_{-n}]=0 \text{ for } 1\leq m,n\leq \max\{2,2d-2\}\label{sel3.5}\\
        E_{1}E_{0}&=v^2 E_{0}E_{1}\label{sel4}\\
        F_{d}F_{d-1}&= v^{-2}F_{d-1}F_{d}\label{sel5}\\
        [E_k,F_l] &= \frac{C^{(k-l)/2}\Psi_{k+l} - C^{(l-k)/2}\Phi_{k+l}}{v-v^{-1}} \text{ for } 0\leq |k|,|l| \leq\max\{2,2d-1\}\label{sel6}\\
        [H_1, H_{-1}] &=[2]\frac{C-C^{-1}}{v-v^{-1}}\label{sel9}\\
        [H_{\pm l},E_{\mp 1}]&=\frac{[2l]}{l}C^{-l/2}E_{\pm l\mp1} \text{ for } 1\leq l\leq \max\{2,2d-2\}\label{sel10}\\
        [H_{\pm l},F_{d\mp1}]&=-\frac{[2l]}{l}C^{l/2}F_{d\pm l\mp 1} \text{ for } 1\leq l\leq \max\{2,2d-2\}.\label{sel11}
    \end{align}
\end{thm}




\begin{rem}

It is clear from  relation~\eqref{sel6} for $2\leq l\leq d-2$,
\[(C-C^{-1})\Phi_l=\frac{C^{-l/2+1}[E_0, F_l]-C^{-l/2}[E_1,F_{l-1}]}{v-v^{-1}}.\]
Thus, if one is willing to invert $C-C^{-1}$ in the algebra, we can get rid of the $\Phi_l$'s from the generating set.
\end{rem}
 
\begin{proof}[Proof of Theorem~\ref{thm:finpres}]
To prove this theorem, we need to show that we can start with the finitely many relations stated in the statement of the theorem, and use them to prove all the relations of the algebra $\uu_d$. Most of our arguments mimic those in the proof of \cite[Theorem 5.5]{FT19}, and we only point out the major differences. 

Firstly, we note that we have the following equality:
\[E_{n+1}= \frac{1}{[2]}C^{1/2} [H_1,E_n]\]
for all $n\geq -1$, by definition and relation~\ref{sel10}. In fact, by taking successive commutators of both sides of the above equality with $H_{-1}$ and using relation~\eqref{sel9}, we get that it is true for all $n\in\ZZ$. By a similar argument, we can show the following:
\begin{align}
[H_{\pm1},E_{n}]&=[2]C^{-1/2}E_{n\pm1},\label{eq:h1e}\\
[H_{\pm1},F_{n}]&=-[2]C^{1/2}F_{n\pm1}\label{eq:h1f}
\end{align}
for all $n\in\ZZ$.

Now, we prove relations in the algebra $\uu_d$. Using the two equalities above along with relations~\eqref{sel1.5}, \eqref{sel2} and \eqref{sel3}, we see that
\begin{align} 
SE_nS^{-1}&=vE_n,\label{eq:se}\\ 
SF_nS^{-1}&=v^{-1}F_n\label{eq:sf}
\end{align}
for all $n\in\ZZ$ by induction. This, along with relation~\eqref{sel1.5}, implies that $S$ commutes with all the $\Psi_n$'s and $\Phi_n$'s which, in turn, shows that $S$ commutes with all the $H_n$'s. Thus, we see that the relations~\eqref{rela1}, \eqref{rela3.6}, \eqref{rela2}, \eqref{rela3} and \eqref{rela3.5} in the algebra $\uu_d$ hold.

\textbf{Step 1:}  Proofs of relations~\eqref{rela4} and \eqref{rela5}: Relation~\eqref{rela4} requires us to prove the equality
\[X_{k,l}:= [E_{k+1},E_l]_{v^2}+[E_{l+1},E_{k}]_{v^2}=0\]
for all $k,l\in \ZZ$. Equation~\eqref{sel4} says that $[E_{1}, E_{0}]_{v^2}=0$, which is equivalent to $X_{0,0}=0$. Then, by an argument similar to the one in \cite[A(ii).a.]{FT19}, we can take successive commutators with $H_{\pm1}$ and $H_{\pm2}$, to show that $X_{k,l}=0$ for all $k$ and $l$ by induction. This proof is based on relation~\eqref{eq:h1e} and the following relation:
\[[H_{\pm2},E_{n}] = \frac{[4]}{2}C^{-1}E_{n\pm2}\]
for all $n\in \ZZ$. This relation for $n=\mp 1$ is true by~\eqref{sel10}, whereas for a general $n$, this follows by taking successive commutators of the above equality with $H_{\pm1}$ and using $\mathbf{B_2}$ of Step 2. (The proof of $\mathbf{B_N}$ for $2\leq N\leq\max\{2, 2d-2\}$ in Step 2 is independent of Step 1.)

The proof of Relation~\eqref{rela5} is similar.

\textbf{Step 2:} Proofs of relations~\eqref{rela7} and~\eqref{rela8}: we only prove the relations when $l>0$, as the case when $l<0$ is similar. Proving relations~\eqref{rela7} and~\eqref{rela8} for $l>0$ is equivalent to proving the following equalities:
\[Y(l,k):=[\Psi_{l+1},E_k]_{v^2}+C^{-1/2}[E_{k+1},\Psi_{l}]_{v^2}=0,\]
\[Y'(l,k):=[\Psi_{l+1},F_k]_{v^{-2}}+C^{1/2}[F_{k+1},\Psi_{l}]_{v^{-2}}=0\]
for all $k\in \ZZ$ and $l\geq 0$. (The equivalence can be seen through an argument similar to the proof of Proposition~\ref{prop:rel78}.) When $l=0$, we have that $Y(l,k)=Y'(l,k)=0$ using equalities~\eqref{eq:h1e},~\eqref{eq:h1f},~\eqref{eq:se} and \eqref{eq:sf}. Furthermore, as a consequence of relation~\eqref{sel10} and~\eqref{sel11}, we have that
\[[H_{l},E_{n}]=\frac{[2l]}{l}C^{-l/2}E_{l+n}, \qquad[H_{l},F_{n}]=-\frac{[2l]}{l}C^{l/2}F_{l+n}\]
for $2\leq l\leq\max\{2,2d-2\}$ and $n=-1$. Then, using relation~\ref{sel3.5} and $\mathbf{B_N}$ below for $2\leq N\leq \max\{2,2d-2\}$, we can take the commutator of the above identity with $H_{\pm1}$, to conclude that it is true for all $n\in\ZZ$. This is equivalent to the equality $Y(l,k)=Y'(l,k)=0$ for all $1\leq l\leq \max\{1,2d-3\}$ and $k\in\ZZ$.

Now, we prove the following statements simultaneously by induction on $N$, which, in particular, will prove relations~\eqref{rela7} and~\eqref{rela8}:
\begin{itemize}
\item{$\mathbf{A_N}$}: $[H_{1},H_r]=[H_{-1},H_{-r}]=0$ for all $1\leq r\leq N$
\item{$\mathbf{B_N}$}: $[H_{-1},H_r]=[H_1,H_{-r}]=0$ for all $2\leq r\leq N$
\item{$\mathbf{C_N}$}: $[E_r,F_s] = \dfrac{C^{(r-s)/2}\Psi_{r+s} - C^{(s-r)/2}\Phi_{r+s}}{v-v^{-1}}$ whenever $\begin{cases}0\leq r,s \text{ and } r+s\leq N+1, \text{ or}\\
r,s+d\leq d \text{ and } d-(N+1)\leq r+s\end{cases}$
\item{$\mathbf{D_N}$}: $Y(r,s)=Y'(r,s)=0$ for all $0\leq r\leq N-1, s\in\ZZ$
\item{$\mathbf{E_N}$}: $[\Psi_r,\Psi_s]=[\Phi_{d-r},\Phi_{d-s}]=0$ for all $0\leq r+s\leq N+1$.
\end{itemize}
For $N\leq \max\{2,2d-2\}$, all of these claims (except $\mathbf{B_N}$) have already been proven or follow directly from the relations assumed in the statement of the theorem. The statement $\mathbf{B_N}$ follows by repeating the proof of Proposition~\ref{prop:rel9'}.


Finally, we do the induction. The proofs of $\mathbf{A_{N+1}}, \mathbf{C_{N+1}}$ and $\mathbf{E_{N+1}}$ exactly mimic the arguments in \cite[A(ii).c.]{FT19}, whereas the proof of $\mathbf{B_{N+1}}$ is the same as the proof of Proposition~\ref{prop:rel9'}.

To prove $\mathbf{D_{N+1}}$, start with the identity $X_{N-1,N-1}=0$ from Step 1 and take the commutator of that with $F_1$:
\[0=[X_{N-1,N-1},F_1]=\frac{2C^{N/2}}{v-v^{-1}}Y(N,N-1).\]
Here, we use the relation~\eqref{rela6} to compute $[E_{N-1},F_1]$ and $[E_N,F_1]$, which follows from relation $\mathbf{C_{N+1}}$. This shows that $Y(N,N-1)=0$. This is equivalent to the following relation:
\[[H_{N+1},E_{N-1}]= \frac{[2(N+1)]}{N+1}C^{-(N+1)/2}E_{2N}.\]
Take commutators of both sides with $H_{\pm1}$ and use $\mathbf{A_{N+1}}$ and $\mathbf{B_{N+1}}$ to conclude that
\[[H_{N+1},E_s]=\frac{[2(N+1)]}{N+1}C^{-(N+1)/2}E_{s+N+1}\]
for all $s\in\ZZ$, which proves that $Y(N,s)=0$ for all $s$. The proof of the equality $Y'(N,s)=0$ is similar.

\textbf{Step 3:} Proof of relation~\eqref{rela6}:
we need to show that
\[[E_k,F_{N-k}] = \frac{C^{(2k-N)/2} \Psi_{N} - C^{(N-2k)/2}\Phi_{N}}{v-v^{-1}}\]
for all $k,N\in \ZZ$.

Without loss of generality, suppose $N\geq 0$. By $\mathbf{C_N}$ from Step 2, we know that the above identity is true when $0\leq k\leq N$. Taking commutator of the above identity with $H_{-1}$, we get
\[[2]C^{-1/2}[E_{k-1},F_{N-k}] - [2]C^{1/2}[E_k, F_{N-k-1}] = \frac{C^{(2k-N)/2}}{v-v^{-1}}[H_{-1},\Psi_N] = [2]\frac{C^{(2k-N)/2}}{v-v^{-1}}(C^{-1}-C)\Psi_{N-1}.\]
Here, we have used equations~\ref{eq:h1e} and \ref{eq:h1f} to compute the commutator of the $E_k$ and $F_{N-k}$ with $H_{-1}$, the equality $[H_{-1},\Phi_N]=0$ which follows from $\mathbf{A_N}$ from Step 2, and the equality:
\[[H_{-1},\Psi_N] = [2](C^{-1}-C)\Psi_{N-1},\]
which is equivalent to $\mathbf{B_N}$ from Step 2. The above identity now allows us to use a negative induction on $k$ as done in \cite[A(ii).e.]{FT19}, proving the identity for all $k\leq 0$. The case when $k\geq N$ is similar.

\textbf{Step 4:} Proof of relation~\eqref{rela9}: this step follows by exactly mimicking the proof of Proposition~\ref{prop:rel9}.
\end{proof}

When $d=2$, we can state a slightly simpler presentation for the algebra $\uu_d$: 

\begin{cor} \label{cor:finpres}
The algebra $\uu_v(\widehat{\sl_2})_{1,1}$ is isomorphic to the algebra generated by the elements 
\[
E_{0},\, F_{0},\, S^{\pm 1},\, C^{\pm 1/2},H_{1},
\]
subject to the following relations:
\begin{align*}
        C^{1/2} &\textrm{ is central}\\
        [S,H_{1}]&=0\\
        SE_{0} S^{-1} &= v E_{0}\\
        SF_0 S^{-1} &= v^{-1} F_0 \\
        [H_1,H_2] &= [H_{-1},H_{-2}]=0\\
        E_{1}E_{0}&=v^2 E_{0}E_{1}\\
        F_{1}F_{0}&= v^{-2}F_{0}F_{1}\\
        [E_k,F_l] &= \frac{C^{(k-l)/2}\Psi_{k+l} - C^{(l-k)/2}\Phi_{k+l}}{v-v^{-1}} \text{ for } 0\leq |k|,|l| \leq 3\\
        [H_1, H_{-1}] &=[2]\frac{C-C^{-1}}{v-v^{-1}}\label{sel9}\\
        [H_{\pm l},E_{\mp 1}]&=\frac{[2l]}{l}C^{-l/2}E_{\pm l\mp1} \text{ for } l=1,2\\
        [H_{\pm l},F_{\mp1}]&=-\frac{[2l]}{l}C^{l/2}F_{\pm l\mp 1} \text{ for } l=1,2.
\end{align*}
(The relations above have been shortened using  Definition \ref{def:otherelements}.)
\end{cor}

\begin{rem}
We have used $E_0$ and $F_0$ in the generating set instead of $E_0$ and $F_1$ as stipulated by Theorem~\ref{thm:finpres} above, since it is clear from its proof that $E_0$ (resp. $F_0$) can be used, along with the other generators to generate all other $E_n$'s (resp. $F_n$'s). Furthermore, we don't need $H_{-1}$ in the generating set since
\[[E_0,F_0]=\frac{\Psi_0-\Phi_0}{v-v^{-1}} = KH_1 - K^{-1}H_{-1},\]
where $K=S^2C^{-1/2}.$
\end{rem}


\begin{rem}
By going through the proof of Theorem~\ref{thm:finpres} with a bit more care, it is possible to find a finite presentation very similar to the one above (with a few extra relations) for the unlocalized spherical subalgebra $Hall(\qrud)^{\sph}$, but we do not pursue that here. For example, the relation $E_1E_0=v^2E_0E_1$ gets replaced by the following two relations:
\begin{align*}[P_1][P_0]&=v^2[P_0][P_1],\\
[I_0'][I_1']&=v^2[I_1'][I_0'],
\end{align*}
etc. The finite generating set for $Hall(\qrud)^{\sph}$ we obtain this way was already seen in Corollary~\ref{cor:4gens}. 
\end{rem}

\bibliography{bibtex}
\bibliographystyle{alpha}

\end{document}